\documentclass[twoside,reqno]{amsart}

\usepackage{amsmath}
\usepackage{amssymb}
\usepackage{amsthm}
\usepackage{mathdots}
\usepackage{mathtools}
\usepackage{graphicx}

\DeclareFontFamily{U}{mathx}{}
\DeclareFontShape{U}{mathx}{m}{n}{<-> mathx10}{}
\DeclareSymbolFont{mathx}{U}{mathx}{m}{n}
\DeclareMathAccent{\widehat}{0}{mathx}{"70}
\DeclareMathAccent{\widecheck}{0}{mathx}{"71}
\usepackage{xcolor}

\usepackage[colorlinks=true]{hyperref}

\newtheorem{theorem}{Theorem}[section]
\newtheorem{proposition}[theorem]{Proposition}
\newtheorem{lemma}[theorem]{Lemma}
\newtheorem{corollary}[theorem]{Corollary}

\theoremstyle{definition}
\newtheorem{definition}[theorem]{Definition}
\newtheorem{remark}[theorem]{Remark}

\numberwithin{equation}{section}

\newcommand{\TtA}{T_{\mkern-2mu\widetilde{A}}}

\newcommand{\kip}{[\,\cdot\, , \cdot\,]}

\newcommand{\sigp}%
{\mathop{\sigma_{\mkern-2.5mu\raisebox{-0.1ex}{\scalebox{.63}{$\mathrm{p}$}}}%
\mkern-1.15mu}\nolimits}

\newcommand{\sigc}%
{\mathop{\sigma_{\mkern-2.75mu\raisebox{-0.1ex}{\scalebox{.65}{$\mathrm{c}$}}}%
\mkern-1.15mu}\nolimits}

\newcommand{\sigr}%
{\mathop{\sigma_{\mkern-2.75mu\raisebox{-0.1ex}{\scalebox{.65}{$\mathrm{r}$}}}%
\mkern-1.15mu}\nolimits}

\newcommand{\iu}{\mathrm{i}}

\def\-{\raisebox{.05pt}{-}}

\newcommand{\lt}{<} \newcommand{\gt}{>}

 \DeclareMathOperator{\BEP}{BEP}

 \DeclareMathOperator{\sgn}{sgn}
 \DeclareMathOperator{\diag}{diag}
 \DeclareMathOperator{\codim}{codim}
 \DeclareMathOperator{\ran}{ran}
 \DeclareMathOperator{\dom}{dom}
 \DeclareMathOperator{\nul}{nul}
 \DeclareMathOperator{\mul}{mul}
 \DeclareMathOperator{\rank}{rank}
 \DeclareMathOperator{\row}{row}

\allowdisplaybreaks

\makeatletter \renewcommand{\p@enumii}{} \makeatother
\makeatletter \renewcommand{\p@enumiii}{} \makeatother

\begin{document}

\title[Linearization and Spectral Equivalence]{ %
Operators without eigenvalues in finite-dimensional vector spaces: \\ Linearization and Spectral Equivalence}

\author{B.~\'{C}urgus}
\address{Department of Mathematics, Western Washington University, Bellingham, WA 98225, USA} \email{\tt curgus@wwu.edu}
\author{A.~Dijksma}
\address{Bernoulli Institute of Mathematics, Computer Science and Artificial Intelligence \\ University of Groningen \\
P.O. Box 407 \\
9700 AK Groningen, The Netherlands} \email{\tt
a.dijksma@rug.nl}

\dedicatory{In memory of Heinz Langer, close friend, \\ dear colleague and inspiring teacher}

\begin{abstract}
In this paper $S$ is a closed symmetric linear relation in a Krein space $\mathfrak H$ with adjoint $S^*$, finite and equal defect numbers $d$, and a boundary mapping $\mathsf{b}: S^* \rightarrow \mathbb{C}^{2d}$ with Gram matrix $\mathsf Q$. We introduce a class $\mathbb A_{S,\mathsf{b}}$ of self-adjoint extensions of $S$ in a Krein space $\widetilde{\mathfrak H}$ containing $\mathfrak H$ as a Krein subspace of finite codimension, together with a class $\mathbb{P}_{\mathsf Q}$ of $d \times 2d$ matrix polynomials. Both classes are equipped with natural equivalence relations. Given $\mathcal{P}(z) \in \mathbb{P}_{\mathsf Q}$, we consider a boundary eigenvalue problem defined by the condition $\mathcal{P}(z)\mathsf{b}(\{f,g\})=0$, $\{f,g\}\in S^*$, and look for its linearizations. By a linearization we mean a linear relation $\widetilde{A} \in \mathbb{A}_{S, \mathsf b}$ such that the Shtraus extension $T_{\widetilde A}(z)$ of $S$ determined by $\widetilde A$ coincides, for all $z \in \overline{\mathbb C}$, with the linear relation $\big\{\{f,g\}\in S^*\,:\,\mathcal P(z)\mathsf b(\{f,g\})=0\big\}$. We prove that this correspondence defines a bijection between the equivalence classes in $\mathbb A_{S, \mathsf b}$ and those in $\mathbb P_{\mathsf Q}$. Moreover, we provide a condition under which the resulting linearizations are spectrally equivalent to the boundary eigenvalue problem: their regular and spectral points coincide, and for each eigenvalue in $\overline{\mathbb C}$ there is a bijection between their Jordan chains.
\end{abstract}


\maketitle

\tableofcontents

\section{Introduction}

Throughout this paper $S$ is a symmetric linear relation in a Krein space $\mathfrak H$ with equal defect numbers $d \in \mathbb{N}$, adjoint $S^*$ and boundary mapping $\mathsf b: S^*\rightarrow \mathbb C^{2d}$ whose Gram matrix is denoted by $\mathsf Q$. In Subsection~\ref{subs-classP} of Section~\ref{classP} we define a class $\mathbb P_{\mathsf Q}$ of $d\times 2d$ matrix polynomials  $\mathcal{P}(z)$. Using these ingredients we define in Subsection~\ref{bep} a boundary eigenvalue problem $\BEP\mkern-2mu(S, \mathsf b, \mathcal{P}(z))$. We introduce its regular and spectral points and Jordan chains of generalized eigenvectors at an eigenvalue $\lambda \in \overline{\mathbb C}$. The first two notions can be defined via the closed linear extension
\begin{equation*} % \label{TP}
T_{S, \mathsf b,\mathcal{P}}(z) = \Big\{\{f,g\}\in S^*\,:\, \mathcal{P}(z)\mathsf b(\{f,g\})=0\Big\}, \quad z \in \overline{\mathbb C},
\end{equation*}
of $S$ in $\mathfrak H$.

The purpose of this paper is to construct a self-adjoint extension $\widetilde A$ of $S$ in a Krein space $\widetilde{\mathfrak H}$ which contains $\mathfrak H$ as a closed Krein subspace of finite codimension, such that
\begin{equation}\label{TATP}
\TtA(z)=T_{S, \mathsf b, \mathcal{P}}(z), \quad z \in \overline{\mathbb C},
\end{equation}
where the left-hand side stands for the \emph{family of Shtraus extensions} of $S$ in $\mathfrak H$ associated with $\widetilde A$ defined by
\begin{equation*} % \label{Tagain}
\TtA(z) =
\left\{
\begin{array}{ll}
\Bigl\{\{\widetilde P_\mathfrak H \widetilde f, \widetilde P_\mathfrak H \widetilde g\} \,:\, 
\{\widetilde f, \widetilde g\} \in \widetilde A,\ \widetilde g - z  \widetilde f \in \mathfrak H \Bigr\}, & z \in \mathbb C,
\\[3mm]
\Bigl\{\{ \widetilde f, \widetilde P_\mathfrak H \widetilde g\} \,:\, 
\{\widetilde f, \widetilde g\} \in \widetilde A,\ \widetilde f \in \mathfrak H\Bigr\} 
= \widetilde P_\mathfrak H \widetilde A|_{\mathfrak H}, & z=\infty.
\end{array}
\right.
\end{equation*}
Here $\widetilde{P}_{\mathfrak H}$ is the orthogonal projection in $\widetilde{\mathfrak H}$ onto $\mathfrak H$. A self-adjoint extension $\widetilde{A}$ of $S$ satisfying \eqref{TATP} will be called a \emph{linearization} of the boundary eigenvalue problem $\BEP\mkern-2mu\bigl(S, \mathsf b, \mathcal{P}(z)\bigr)$. We look for linearizations within the class $\mathbb{A}_{S, \mathsf b}$ defined in Subsection~\ref{subs-classA}. Both classes $\mathbb{P}_{\mathsf Q}$ and $\mathbb{A}_{S, \mathsf b}$ are equipped with natural equivalence relations. The linear relations in \eqref{TATP} depend only on the equivalence classes and not on the particular representatives. In fact, we prove a stronger statement: the equality \eqref{TATP} establishes a bijective correspondence between the equivalence classes of $\mathcal{P}(z) \in \mathbb{P}_{\mathsf Q}$ and those of $\widetilde A \in \mathbb A_{S, \mathsf b}$; see Theorem~\ref{X}.

In addition, for the boundary eigenvalue problem $\BEP\mkern-2mu\bigl(S, \mathsf b, \mathcal{P}(z)\bigr)$,   we determine when it is {\it spectrally equivalent} to its linearization $\widetilde A$. By spectral equivalence we mean that their regular and spectral points coincide and that there is a bijective correspondence between their Jordan chains; see Theorem~\ref{th-speq}.

Linearizations of a boundary eigenvalue problem have been extensively studied in \cite[Theorems~7.1 and~7.4]{DLdS84} from a general point of view. There $\widetilde{\mathfrak H}$ is a Pontryagin space and $\mathfrak H$ is a Hilbert space. In \cite[Theorem~7.1]{DLdS84} $S$ is a closed symmetric linear relation in $\mathfrak H$ with finite, not necessarily equal, defect numbers $d^+$ and $d^-$ and $\mathsf b: S^* \rightarrow \mathbb C^{d^++d^-}$ is a boundary mapping. The theorem states that the equality
\begin{equation} \label{eq-cUb}
T(z) = \Bigl\{
\{f,g\}\in S^* \,:\, \mathcal U(z)\mathsf b(\{f,g\})=0 \Bigr\} \quad \text{for essentially all} \quad z \in \mathbb{C} \setminus\mathbb{R}
\end{equation}
establishes a bijective correspondence between, on the one hand, families \(T(z)\) of
Shtraus extensions of \(S\) defined by closely connected (or minimal) self-adjoint
extensions \(\widetilde A\) of \(S\) in \(\widetilde{\mathfrak H}\) and, on the other hand,
equivalence classes of locally holomorphic matrix functions \(\mathcal U(z)\) on
\(\mathbb C\setminus\mathbb R\), whose size depends on whether \(z\) belongs to the upper
or lower half-plane and which satisfy certain additional properties. In \cite[Theorem~7.4]{DLdS84} the problem is reduced to the polynomial case: $T(z)$ corresponds to a self-adjoint extension $\widetilde A$ of $S$  with $\dim (\widetilde{\mathfrak H} \ominus \mathfrak{H}) < \infty$ if and only if $d^+=d^-$ and $\mathcal U(z)$ can be chosen as a matrix polynomial.

In the present paper we study the direct correspondence between the polynomial $\mathcal P(z)$ and the linearization $\widetilde A$ and we give explicit formulas for the linearizations; see Step~3 in the proof of Theorem~\ref{X}. Our approach is more direct (``bottom-up''), more detailed (including full rank conditions and statements for all $z \in \overline{\mathbb C}$), and relies primarily on tools from linear algebra. We use several results from our previous papers such as the Coupling Theorem~\ref{Coupling}, the Reproducing Kernel Canonical Subspace Theorem~\ref{t-fCrks} and the Model Theorem~\ref{t-poso} which we recall in Appendix~\ref{AppA}. Sufficient details of other results are given to make the paper reasonably self-contained.

Results on self-adjoint linearizations of boundary eigenvalue problems, such as those in
\cite{DLdS84}, also appear in the papers and lecture notes \cite{DLdS86}-\cite{DL95} by Heinz Langer and coauthors. In these publications $S$ is a symmetric linear relation in a Hilbert space sometimes with equal defect numbers so that $S$ can be associated with differential operators or relations; see, for example, \cite{DLdS872}, \cite{DLdS88} and \cite{DL95}. The linearizations $\widetilde A$ are self-adjoint relations in Pontryagin or Krein spaces. The associated Shtraus extensions $\TtA$ of $S$ are characterized by means of characteristic functions of a unitary colligation whose outer spaces are the defect spaces of $S$ at a fixed nonreal point; see \cite{DLdS86}, \cite{DLdS862} and \cite{DLdS87}.

Spectral equivalence results similar to those in Section~\ref{speceq} also appear in, for example, \cite{DLdS872} and \cite{DLdS93}. There differential operators are considered with eigenvalue-dependent boundary conditions determined by locally holomorphic matrix functions, and the relation between the Jordan chains of the operator and those of its linearization is established by different methods. Denote by $R(z)$ the generalized resolvent of $S$ on $\mathfrak H$, that is, the compressed resolvent of the linearization on $\mathfrak H$.  Then in \cite[Proposition~5.3]{DLdS872} it is assumed that for some function $f(z)$ with values in $\mathfrak H$, $R(z)f(z)$ admits an expansion around the isolated eigenvalue. The proof of \cite[Theorem~8.3]{DLdS93} makes use of poles and pole functions of the matrix  Nevanlinna function appearing in the integral representation of $R(z)$ described in earlier sections of that paper.

In \cite{Tr} a nonself-adjoint linearization is considered for a boundary eigenvalue problem associated with a pair of regular differential operators on a compact interval with boundary conditions in which the eigenvalue appears polynomially. Completeness of Jordan chains and spectral properties are established using a detailed asymptotic analysis of the associated Green's function.

We briefly describe the contents of the paper.

In Subsection~\ref{subs-not} below we fix the notation used throughout the paper.

In Subsection~\ref{subs-classP} we define the class $\mathbb{P}_{\mathsf Q}$ of $d \times 2d$ matrix polynomials $\mathcal P(z)$ and prove their basic properties. On $\mathbb{P}_{\mathsf Q}$ we define an equivalence relation. We introduce the row-reversal matrix polynomial $\mathcal R(z)$ of $\mathcal P(z)$. It too belongs to $\mathbb P_\mathsf Q$ and is used to describe the Jordan chains of the boundary eigenvalue problem $\BEP\mkern-2mu\bigl(S, \mathsf b, \mathcal{P}(z)\bigr)$ at the eigenvalue $\infty$ in Subsection~\ref{bep}. Subsection~\ref{subs-two} contains two closely related theorems. The first one provides a factorization of $\mathcal P(z)$ in which one factor contains a matrix polynomial now of smaller size and without constant rows; its proof is long and given in Appendix~\ref{proof}. The second presents a list of formulas for canonical symmetric and self-adjoint extensions of $S$. Both theorems play a key role in the proof of Theorem~\ref{X}.

In Subsection~\ref{partition} we recall the definitions of regular and spectral points in $\overline{\mathbb C}$ of a closed linear relation $R$ in a Krein space, in particular of a self-adjoint relation, as well as the definition of a Jordan chain at an eigenvalue in $\overline{\mathbb C}$. In Subsection~\ref{bep} we introduce the same notions for the boundary eigenvalue problem $\BEP\mkern-2mu\bigl(S, \mathsf{b}, \mathcal{P}(z)\bigr)$, the main object in this paper often abbreviated as $\BEP$.

In Subsection~\ref{subs-classA} we define the class $\mathbb A_{S, \mathsf b}$ of self-adjoint extensions $\widetilde A$ of $S$ in a Krein space $\widetilde{\mathfrak H}$ containing $\mathfrak H$ as a Krein subspace satisfying conditions  (\ref{def-bbA-i1})-(\ref{def-bbA-i4}) of Definition~\ref{def-bbA}. Conditions (\ref{def-bbA-i3}) and (\ref{def-bbA-i4}) replace the closely connectedness condition referred to in the sentence containing \eqref{eq-cUb}; for more details see \cite[Proposition~4.1]{CDCAOT}. They are important and the first part of the title of this paper refers to them. In Proposition~\ref{pr-bbAbasics} we consider the canonical symmetric extension $S_0=\widetilde{A} \cap \mathfrak H^2$ of $S$, sometimes called the trace of $\widetilde A$ in $\mathfrak H$. We use it to find representations of $\widetilde A$ and of the corresponding Shtraus subspace $\TtA(z)$ as extensions of $S_0$. It can be expressed in terms of Shtraus extensions; see formula \eqref{nonsymm}. On $\mathbb A_{S, \mathsf b}$ we define an equivalence relation in Subsection~\ref{subs-classA} and show in Subsection~\ref{subs-mainth} that two of its elements are equivalent if and only if their associated Shtraus extensions coincide; see Theorem~\ref{th-EqSt}. In Sections~\ref{linearization} and~\ref{speceq} we present the main results of this paper: the Linearization Theorem~\ref{X} expressed in terms of a binary relation and the Spectral Equivalence Theorem~\ref{th-speq}.

We assume the reader is familiar with elements of operator theory in indefinite inner product spaces, as in \cite{AI}, \cite{Bo}, \cite{IKL} and \cite{GhB}. 

This paper completes a series of four papers---the first three are \cite{CDLAA20}, \cite{CDLAA23} and \cite{CDCAOT}---with as a common theme: Operators without eigenvalues in finite-dimensional vector spaces. A model for such an operator is the operator $S_\mu$ of multiplication by the independent variable in a canonical vector space $\mathfrak C_\mu$ of vectors whose entries are polynomials, see Definition~\ref{def-Cmu} in Appendix~\ref{AppA}. The Model Theorem~\ref{t-poso} then provides the connection with the class $\mathbb{A}_{S,\mathsf{b}}$ in Definition~\ref{def-bbA} and the class $\mathbb{P}_{\mathsf Q}$ in Definition~\ref{def-bbP}. This connection is one of the main ingredients in the linearization and spectral-equivalence results proved below. 

\subsection{Notation} \label{subs-not}
By $\mathbb{N}$, $\mathbb{R}$ and $\mathbb{C}$ we denote the sets of positive integers, real numbers and complex numbers, and  $\overline{\mathbb{C}} = \mathbb{C}\cup \{\infty\}$ is the one point compactification of $\mathbb{C}$, with the conventions $1/0=\infty$ and $1/\infty=0$.  For $z \in \overline{\mathbb{C}}$, $z^* \in \overline{\mathbb{C}}$ denotes the complex conjugate of $z$, we have $\infty^* = \infty$. If $\Omega \subseteq \overline{\mathbb C}$, then $\Omega^*=\{z^* : z \in \Omega\}$. The asterisk as a superscript is also used to denote the adjoint of a matrix.

Recall that a closed linear relation \(T\) in a Krein space \(\mathfrak{K}\) is a closed subspace of \(\mathfrak{K}^2\). A closed linear operator in \(\mathfrak{K}\) is viewed as a closed linear relation via its graph in \(\mathfrak{K}^2\). For the usual linear operations with closed linear relations, the definitions of the domain \(\dom T\), the range  \(\ran T\), the null-space \(\nul T\), the inverse \(T^{-1}\), and the adjoint \(T^*\) we refer to \cite{DS1}. By \(I_{\mathfrak{K}}\) we denote the identity operator on \(\mathfrak{K}\).

For a set \(X\), we denote its cardinality by \(\#X\). For sets \(X\) and \(Y\), we use \(X\subset Y\) to denote strict inclusion (proper subset), and \(X\subseteq Y\) for inclusion allowing equality.

For $n \in \mathbb{N}$ we denote by $\mathsf{I}_n$ and $\mathsf{0}_n$ the $n \times n$ identity and zero matrix. Otherwise, $\mathsf{0}$ denotes a matrix whose size is implied by the context. For constant matrices we use the sans-serif font and for matrix polynomials we use the calligraphic font accompanied with the variable. For an $m\times n$ matrix $\mathsf{M}$ and $k \in \{1,\ldots,m\}$ by $\bigl.\mathsf{M}\bigr|_k$ we denote the $k$-th row of $\mathsf{M}$. By $\row \mathsf{M}$ we denote the space spanned by the rows $\mathsf{M}\bigl.\bigl.\bigr|_k$, $k \in \{1, \dots , m\}$. By $\ran \mathsf{M}$ we denote the range (the column space) of $\mathsf{M}$.

For $m,n \in \mathbb{N}$ we denote by $\mathbb{C}^{m \times n}[z]$ the space of all matrix polynomials with coefficients in $\mathbb{C}^{m \times n}$. The \emph{degree of such a polynomial} is $-\infty$ if it is the zero polynomial, otherwise it is the highest power of $z$ for which the corresponding matrix coefficient is nonzero. A square matrix polynomial is called \emph{unimodular} if its determinant is a nonzero scalar. We write $\mathbb{C}^m[z]$ for $\mathbb{C}^{m \times 1}[z]$. For vector functions $a(z)$ and $b(z)$, the identity $a(z)\equiv b(z)$ stands for the proposition $a(z)= b(z)$ for all $z \in \mathbb{C}$. Let $\mathcal{S}(z) \in \mathbb{C}^{m \times n}[z]$ be a nonzero matrix polynomial of degree $s$; thus $s \in \{0\}\cup\mathbb{N}$. If
\[
 \mathcal{S}(z)
 = \mathsf{S}_0 + \mathsf{S}_1 z + \cdots + \mathsf{S}_s z^s, \quad \mathsf{S}_j \in \mathbb{C}^{m \times n}, \quad j\in\{0,\ldots,s\} \quad \text{with} \quad \mathsf{S}_s \neq \mathsf{0}
\]
is the expansion of $\mathcal S(z)$  in powers of $z$ and $\mathsf{C}_{\mathcal{S}}$ stands for its $(s+1)m \times n$ {\it coefficient matrix}:

\begin{equation*}  % \label{eqP}
\mathsf{C}_{\mathcal{S}} = \begin{bmatrix} \mathsf{S}_0 \\ \mathsf{S}_1 \\ \vdots \\ \mathsf{S}_s
\end{bmatrix},
\end{equation*}
then
\begin{equation}\label{eq-kTzkT}
\bigcap_{z \in \mathbb{C}} \nul \mathcal{S}(z)
= \nul \mathsf{C}_{\mathcal S}.
\end{equation}

Assume no row of \(\mathcal S(z)\) is identically zero; equivalently, for every
\(k\in\{1,\ldots,m\}\) there exists \(z\in\mathbb C\) such that
\(\mathcal S(z)\bigl.\bigr|_k\neq\mathsf 0\). For every $k \in \{1,\ldots,m\}$ we set
\[
\sigma_k = \deg \bigl(\mathcal{S}(z)\bigl.\bigr|_k \bigr),
\]
and, since the rows of $\mathcal{S}(z)$ are nonzero, we have $\sigma_k \in \{0\}\cup\mathbb{N}$. By definition of the degree of a row polynomial, for every $k \in \{1,\ldots,m\}$ we have
\[
\mathsf{S}_{\sigma_k}\mkern-2mu\bigl.\bigr|_k \neq \mathsf{0} \quad \text{and} \quad
\mathsf{S}_{l}\mkern-1mu\bigl.\bigr|_k = \mathsf{0}
\]
for all $l \in \mathbb{N}$ such that $\sigma_k \lt l \leq s$. That is, for all  $k \in \{1,\ldots,m\}$ we have
\[
\bigl.\mathcal{S}(z)\bigr|_k = \sum_{j=0}^{\sigma_k} \bigl(\bigl.\mathsf{S}_{j}\bigr|_k\bigr) \mkern 2mu z^j
\]
and $\mathsf{S}_{\sigma_k}\bigl.\bigr|_k$ is the leading coefficient of the row polynomial $\mathcal{S}(z)\bigl.\bigr|_k$.

We define the $m\times n$ matrix $\mathsf{S}_{\infty}$ to be the matrix consisting of the leading coefficients of the rows $\bigl.\mathcal{S}(z)\bigr|_k$ with $k \in \{1,\ldots,m\}$. Specifically,
\begin{equation}\label{Sinfinity}
 \mathsf S_\infty = \begin{bmatrix}
\mathsf{S}_{\sigma_1}\mkern-3mu\bigl.\bigr|_1 \\[2pt]
\mathsf{S}_{\sigma_2}\mkern-3mu\bigl.\bigr|_2 \\[2pt]
\vdots \\[2pt]
\mathsf{S}_{\sigma_m}\mkern-3mu\bigl.\bigr|_m
\end{bmatrix}
= \lim_{z \rightarrow \infty}\diag (z^{-\sigma_1}, \ldots, z^{-\sigma_m})
\mathcal{S}(z).
\end{equation}

\section{The class \texorpdfstring{$\mathbb{P}_{\mathsf{Q}}$}{PsubQ} of matrix polynomials} \label{classP}

\subsection{Definition and basic properties} \label{subs-classP}

\begin{definition} \label{def-bbP}
Let $d \in \mathbb{N}$ and let $\mathsf{Q}$ be an invertible  self-adjoint $2d \times 2d$ matrix with $d$ positive and $d$ negative eigenvalues. By $\mathbb{P}_\mathsf{Q}$ we denote the set of $d \times 2d$ matrix polynomials $\mathcal{P}(z)$ with the following properties
\begin{enumerate}
\renewcommand*\theenumi{\alph{enumi}}
\renewcommand*\labelenumi{\rm{(\theenumi)}}
\setlength{\itemsep}{1pt}% \itemsep0em
 \item \label{i-t-m-a}
$\mathcal{P}(z) \mathsf{Q}^{-1}{\mathcal{P}}(z^*)^* = \mathsf{0}$ \
for all $z \in \mathbb{C}$,
 \item \label{i-t-m-b}
$\rank {\mathcal{P}}(z) = d$ \ for all $z \in \mathbb{C}$,

 \item \label{i-t-m-c}
$\rank \mathsf P_{\infty} = d$,

\item \label{i-t-m-d}
$\deg \mathcal{P}(z)\mkern-1mu \bigl.\bigr|_1 \ \geq \ \cdots \ \geq \ \deg \mathcal{P}(z)\mkern-1mu \bigl.\bigr|_{d} \ \geq \ 0 $.
\end{enumerate}
\end{definition}

Items (\ref{i-t-m-a})-(\ref{i-t-m-d}) and \cite[Theorem~3.2]{CDLAA23} imply that if $\mathbb{W}$ is any subset of $\mathbb{C}$ with $\# \mathbb{W} \gt \deg \mathcal{P}(z)$,  then
\begin{equation}\label{eq-i-t-m-e}
\dim \Big( \bigcap_{z \in \mathbb W}\nul \mathcal{P}(z)\Big) = \#\Bigl\{j \in \{1, \ldots, d\}\,:\, \deg \mathcal{P}(z)\mkern-1mu\bigl.\bigr|_j=0 \Bigr\}.
\end{equation}

\begin{remark}
The rank and row-degree conditions {\rm(\ref{i-t-m-b})-(\ref{i-t-m-d})} in
Definition~\ref{def-bbP} are, in fact, largely normalization conditions. Indeed,
\cite[Theorem~4.3]{CDLAA20} shows that if $\mathcal{T}(z)\in\mathbb C^{d\times 2d}[z]$ satisfies $\rank \mathcal{T}(z)=d$ for some $z\in\mathbb C$, then $\mathcal{T}(z)$ admits a factorization
\[
\mathcal{T}(z)=\mathcal{W}(z)\mathcal{P}(z), \qquad z\in\mathbb C,
\]
where $\mathcal{W}(z)\in\mathbb C^{d\times d}[z]$ satisfies $\det\mathcal{W}(z)\not\equiv 0$ and $\mathcal{P}(z)\in\mathbb C^{d\times 2d}[z]$ satisfies {\rm(\ref{i-t-m-b})-(\ref{i-t-m-d})}. Moreover, this factor $\mathcal{P}(z)$ is essentially unique: if also $\mathcal{T}(z) = \mathcal{W}_1(z)\mathcal{P}_1(z)$ with $\det\mathcal{W}_1(z)\not\equiv 0$ and $\mathcal{P}_1(z)$ satisfying {\rm(\ref{i-t-m-b})-(\ref{i-t-m-d})}, then
$\mathcal{P}_1(z)=\mathcal{V}(z)\mathcal{P}(z)$ for some unimodular $d\times d$ matrix polynomial $\mathcal{V}(z)$.
\end{remark}

\begin{definition} \label{def-bbPeq}
Two polynomials $\mathcal{P}(z)$ and $\mathcal{S}(z)$ from $\mathbb{P}_\mathsf{Q}$ are called \emph{equivalent}, in notation $\mathcal{P}(z) \sim  \mathcal{S}(z)$, if there exists a unimodular $d \times d$ matrix polynomial $\mathcal W(z)$ such that
\begin{equation}\label{equi}
\mathcal{S}(z) = \mathcal{W}(z)\mathcal{P}(z) \quad
\text{for all} \quad z \in \mathbb{C}.
\end{equation}
For $\mathcal{P}(z) \in \mathbb{P}_{\mathsf{Q}}$ its equivalence class relative to \(\sim\) is denoted by \([\mathcal{P}(z)]\). We denote the set of all \(\sim\)-equivalence classes in  \(\mathbb{P}_{\mathsf{Q}}\) by \(\mathbb{P}_{\mathsf{Q}}/\mathord{\sim}\).
\end{definition}

\begin{lemma} \label{4statements}
The following statements hold.

\begin{enumerate}
\renewcommand*\theenumi{\roman{enumi}}
\renewcommand*\labelenumi{\rm{(\theenumi)}}
\setlength{\itemsep}{1pt}% \itemsep0em

\item \label{4statements-i1}
The relation $\sim$ is an equivalence relation on $\mathbb{P}_{\mathsf{Q}}$.
\end{enumerate}

Assume $\mathcal{P}(z)$ and $\mathcal{S}(z)$ from $\mathbb{P}_{\mathsf{Q}}$ are equivalent and \eqref{equi} holds for some unimodular $d \times d$ matrix polynomial $\mathcal W(z)$ and set $\mu_j = \deg \mathcal{P}(z)\mkern-1mu\bigl.\bigr|_j$ for \(j \in \{1,\ldots,d\}\).

\begin{enumerate}
\setcounter{enumi}{1}
\renewcommand*\theenumi{\roman{enumi}}
\renewcommand*\labelenumi{\rm{(\theenumi)}}
\setlength{\itemsep}{1pt}% \itemsep0em

\item \label{4statements-i2}
Then also $\deg \mathcal{S}(z)\mkern-1mu\bigl.\bigr|_j= \mu_j$ for all $j \in \{1, \ldots, d\}.$

\item \label{4statements-i3}
If $\mathcal{W}(z)=\big[ w_{jk}(z) \big]_{j,k=1}^d$ and $\mathcal{W}(z)^{-1}=\big[ w^{jk}(z) \big]_{j,k=1}^d$, then
\begin{equation} \label{coeff}
\max\bigl\{ \deg w_{jk}(z), \, \deg w^{jk}(z) \bigr\} \leq \mu_j-\mu_k  \quad \text{for all} \quad j,k \in \{1, \ldots, d\}.
\end{equation}

\item \label{4statements-i4}
If $\mathcal D(z)=\diag(z^{\mu_1}, \ldots,z^{\mu_d})$, then the $d \times d$ matrix
\[
\mathsf{W} = \lim_{z \rightarrow \infty} \mathcal D(z)^{-1}\mathcal W(z)\mathcal D(z)
\]
is well-defined, it equals $\mathsf{S}_\infty \mathsf P_\infty^*\left(\mathsf P_\infty\mathsf P_\infty^*\right)^{-1}$ and is invertible. Moreover,
\(\mathsf{S}_\infty = \mathsf{W} \mathsf{P}_\infty.\)
\end{enumerate}
\end{lemma}

\begin{proof}
(\ref{4statements-i1}) The class of unimodular $d \times d$ matrix polynomials is closed under taking inverses and products.  This implies
that the relation $\sim$ in $\mathbb{P}_{\mathsf{Q}}$ is reflexive, symmetric and transitive, hence an equivalence relation.

(\ref{4statements-i2}) This is proved in \cite[Corollary~4.2]{CDLAA20}.

(\ref{4statements-i3}) Properties (\ref{i-t-m-b}) and (\ref{i-t-m-c}) of $\mathcal{P}(z)\in \mathbb{P}_{\mathsf{Q}}$ imply that  $\mathcal{P}(z)$ has the predictable degree property; see \cite[Theorem~A.2]{McE} (or \cite[Theorem~4.1]{CDLAA20}). This means that if $u(z)=\bigl[ u_k(z)\bigr]_{k=1}^d$ is a $1\times d$ matrix polynomial, then
\[
\deg u(z)\mathcal{P}(z) = \max \bigl\{ \deg u_k(z)+\mu_k\,:\, k \in \{1,\ldots, d\} \bigr\}.
\]
From this, (\ref{4statements-i2}) and \eqref{equi} we obtain that for $j \in \{1,\ldots,d\}$
\[
\mu_j = \deg \mathcal{S}(z)\bigl.\bigr|_j
= \deg \mathcal{W}(z)\mkern-1mu\bigl.\bigr|_j \mathcal{P}(z)
= \max \Bigl\{ \deg w_{jk}(z) + \mu_k \,:\, k \in \{1,\ldots, d\}\Bigr\}.
\]
This implies \eqref{coeff} for \(\mathcal{W}(z)\). The proof for \(\mathcal{W}(z)^{-1}\) is analogous.

(\ref{4statements-i4})
Set $\mathcal{M}(z)=\mathcal{D}(z)^{-1}\mathcal{P}(z)$ for \(z\in\mathbb{C}\mkern-1mu\setminus\mkern-2mu\{0\}\), so that $\lim_{z \rightarrow \infty}\mathcal{M}(z) = \mathsf{P}_\infty$. Since the $d \times 2d$ matrices $\mathcal{M}(z), z\in\mathbb{C}\mkern-1mu\setminus\mkern-2mu\{0\}$, and $\mathsf{P}_\infty$ have full rank,  the $d \times d$ matrices $\mathcal{M}(z)\mathcal{M}(z)^*$ and $\mathsf{P}_\infty\mathsf{P}_\infty^*$ are invertible.

It follows from \eqref{equi} that
\begin{equation}\label{eq-needl}
\mathcal{D}(z)^{-1}\mathcal{W}(z) \mathcal{D}(z)\mathcal{M}(z) = \mathcal{D}(z)^{-1} \mathcal{S}(z) \quad
\text{for all} \quad z \in \mathbb{C}\setminus\{0\},
\end{equation}
and hence
\[
\mathcal{D}(z)^{-1}\mathcal{W}(z)\mathcal{D}(z)
=
\mathcal{D}(z)^{-1}\mathcal{S}(z)\mathcal{M}(z)^*
\bigl(\mathcal{M}(z)\mathcal{M}(z)^*\bigr)^{-1}
\quad
\text{for all} \quad z \in \mathbb{C}\setminus\{0\}.
\]
Letting \(z\to \infty\), yields that $\mathsf{W}$ is well-defined and $\mathsf{W} = \mathsf{S}_\infty \mathsf{P}_\infty^* \bigl(\mathsf{P}_\infty \mathsf{P}_\infty^*\bigr)^{-1}$. To prove that \(\mathsf{W}\) is invertible, observe that since $\mathcal{W}(z)$ is unimodular we have \(\det \mathcal{W}(z) \equiv w \in \mathbb{C}\mkern-1mu\setminus\mkern-2mu\{0\}\). Therefore
\begin{align*}
0 \neq w & = \lim_{z \rightarrow \infty} \det \mathcal W(z) \\
 & = \lim_{z \rightarrow \infty} \det \left(
  \mathcal D(z)^{-1}\mathcal W(z)\mathcal D(z)\right) \\
 & = \det \bigl(\mathsf{S}_\infty \mathsf{P}_\infty^*%
 (\mathsf{P}_\infty \mathsf{P}_\infty^*)^{-1}\bigr)\\
 & = \det \mathsf{W}.
\end{align*}

The equality \(\mathsf{W} \mathsf{P}_\infty = \mathsf{S}_\infty\) follows from \eqref{eq-needl} by letting \(z\to \infty\).
\end{proof}

\begin{definition}
The \emph{row-reversal} $\mathcal R(z)$ of $\mathcal{P}(z) \in \mathbb{P}_{\mathsf{Q}}$ is the $d \times 2d$ matrix polynomial:
\begin{equation*}
\mathcal{R}(z) = \begin{cases}
\diag \bigl( z^{\mu_1}, \ldots, z^{\mu_d} \bigr) \mathcal{P}(1/z)
 & \text{for} \quad  z \in \mathbb{C}\setminus\{0\},  \\[9pt]
\displaystyle\lim_{z\to \infty} \diag \bigl( z^{\mu_1}, \ldots, z^{\mu_d} \bigr) \mathcal{P}(1/z)
=   \mathsf{P}_\infty & \text{for} \quad z = 0,
\end{cases}
\end{equation*}
where $\mu_j = \deg \mathcal{P}(z)\bigl.\bigr|_j$ for $j \in\{1, \ldots, d\}$.
\end{definition}

\begin{lemma} \label{reverse}
The following statements hold.

\begin{enumerate}
\renewcommand*\theenumi{\roman{enumi}}
\renewcommand*\labelenumi{\rm{(\theenumi)}}
\setlength{\itemsep}{1pt}% \itemsep0em

\item \label{reverse-i1}
If $\mathcal R(z)$ is the row-reversal of $\mathcal{P}(z) \in \mathbb{P}_{\mathsf{Q}}$, then $\mathcal R(z) \in \mathbb{P}_{\mathsf{Q}}$, \(\mathsf{R}_\infty = \mathcal{P}(0)\), and
\[
\deg \mathcal{R}(z)\mkern-1mu\bigl.\bigr|_j = \deg \mathcal{P}(z)\mkern-1mu\bigl.\bigr|_j, \quad    j \in \{1,\ldots,d\}.
\]

\item \label{reverse-i11}
The row-reversal of $\mathcal R(z)$ is $\mathcal{P}(z)$.

\item \label{reverse-i2}
If $\mathcal{P}(z)$ and $\mathcal{S}(z)$ in $\mathbb{P}_{\mathsf{Q}}$ are equivalent, then so are their row-reversals $\mathcal{R}(z)$ and $\mathcal{T}(z)$.

\item \label{reverse-i3}
For every \(z \in \mathbb{C}\setminus\{0\}\) we have $\nul \mathcal{R}(1/z) = \nul \mathcal{P}(z)$.
\end{enumerate}
\end{lemma}

\begin{proof} We use the notation of the previous lemma.

(\ref{reverse-i1})
We have
\[
\mathsf{R}_\infty = \lim_{z \rightarrow \infty} \mathcal{D}(z)^{-1}\mathcal{R}(z)
= \lim_{z \rightarrow \infty} \mathcal{P}(1/z) = \mathcal{P}(0) = \mathsf{P}_0.
\]
Since $\rank \mathcal{P}(0) = d$ the last equality implies $\deg \mathcal{R}(z)\mkern-1mu\bigl.\bigr|_j =\mu_j$ for all  $j\in\{1, \ldots, d\}$. The remaining properties of \(\mathbb{P}_{\mathsf{Q}}\) are straightforward consequences of the definition.

Statement (\ref{reverse-i11}) follows from (\ref{reverse-i1}) and the definition of the row-reversal.

(\ref{reverse-i2}) We claim that if $\mathcal V(z)$ is the $d\times d$ matrix
 \begin{equation*}
 \mathcal V(z)=\left\{
\begin{array}{lll} \mathcal D(z) \mathcal W(1/z)\mathcal{D}(z)^{-1}, &
z \in \mathbb{C}\mkern-2mu\setminus\mkern-2mu\{0\}, \\[2mm]
\mathsf{S}_\infty \mathsf P_\infty^*\left(\mathsf P_\infty\mathsf P_\infty^*\right)^{-1}, & z=0,
\end{array}
\right.
\end{equation*}
then $\mathcal{V}(z)$ is a unimodular $d \times d$ matrix polynomial and $\mathcal{V}(z)\mathcal{R}(z) = \mathcal{T}(z)$, $z \in \mathbb{C}$. That it is a matrix polynomial follows from Lemma~\ref{4statements}(\ref{4statements-i3}) and (\ref{4statements-i4}). The latter implies that $\lim_{z \rightarrow 0} \mathcal V(z) = \mathsf{W} = \mathcal{V}(0)$. Since $\mathcal{W}(z)$ is unimodular, $\mathcal{V}(z)$ is also unimodular.

It is straightforward to verify the equality $\mathcal{V}(z)\mathcal{R}(z) = \mathcal{T}(z)$ for $z\neq 0$. That it also holds for $z=0$ follows by taking the limit $z \rightarrow 0$ in this equality, using $\mathcal R(0)=\mathsf P_\infty$ and  $\mathcal{T}(0) = \mathsf{S}_\infty$, and Lemma~\ref{4statements}(\ref{4statements-i4}).

(\ref{reverse-i3})
By definition of the reversals, for \(z \in \mathbb{C}\setminus\{0\}\) we have \(\mathcal{R}(1/z) = \mathcal{D}(z)^{-1}\mathcal{P}(z)\) and (\ref{reverse-i3}) follows since \(\mathcal{D}(z)\) is invertible.
\end{proof}

\subsection{Two closely related theorems} \label{subs-two}

We continue with a factorization of $\mathcal{P}(z)\in\mathbb{P}_{\mathsf Q}$. Using the factors obtained here, we construct in Section~\ref{linearization} (see Step~3 of
Theorem~\ref{X}) a linearization $\widetilde{A}\in\mathbb{A}_{S,\mathsf b}$ of the boundary eigenvalue problem introduced in Section~\ref{RegSpec}. The following theorem is a polynomial analog of \cite[Theorem~5.1]{ACD}. Since its proof is lengthy, we defer it to Appendix~\ref{proof}.

\begin{theorem}\label{modelP}
Let $d_0,d\in\mathbb{N}$ with $d_0<d$. Let $\mathsf Q$ be a self-adjoint invertible $2d\times 2d$ matrix with $d$ positive and $d$ negative eigenvalues. Let $\mathcal P(z)$ be a $d\times 2d$ matrix polynomial of degree $p\in\mathbb N$. For $j \in \{1,\ldots,d\}$ set
\(
\mu_j =\deg \mathcal P(z)\bigl.\bigr|_j \in \{0\}\cup\mathbb{N},
\)
and assume that $\mathcal{P}(z) \in \mathbb{P}_{\mathsf Q}$, where in
item {\rm (\ref{i-t-m-d})} of Definition~{\rm\ref{def-bbP}} we assume
\begin{equation*} % \label{eq-abcd-i5}
p = \mu_1 \geq  \cdots \geq \mu_{d_0} \gt \mu_{d_0+1} = \cdots = \mu_d = 0.
\end{equation*}
Let $\mathsf B$ be the $(d-d_0)\times 2d$ constant matrix consisting of the bottom
$d-d_0$ rows of $\mathcal P(z)$.

Then there exist a unimodular $d \times d$ matrix polynomial $\mathcal W(z)$,
a $d_0 \times 2d_0$ matrix polynomial $\widehat{\mathcal P}(z)$, and a
$2d_0 \times 2d$ matrix $\mathsf B_0$ such that
\begin{equation}\label{reprP}
\mathcal W(z)\mathcal{P}(z) =
 \begin{bmatrix}  \widehat{\mathcal{P}}(z) & \mathsf{0} \\[5pt]
 \mathsf{0} & \mathsf{I}_{d-d_0} \end{bmatrix}
\begin{bmatrix}  \mathsf B_0 \\[5pt] \mathsf B  \end{bmatrix},
\end{equation}
and
\begin{enumerate}
\renewcommand*\theenumi{\roman{enumi}}
\renewcommand*\labelenumi{\rm{(\theenumi)}}
\item \label{abcd-c1}
the $(d+d_0) \times 2d$ matrix $\displaystyle\begin{bmatrix}  \mathsf B_0 \\ \mathsf B  \end{bmatrix}$ has full rank $d+d_0$,

\item \label{abcd-c2}
$\mathsf{B} \mathsf{Q}^{-1} \mathsf{B}^* = \mathsf{0}$ and
$\mathsf{B}_0 \mathsf{Q}^{-1} \mathsf{B}^* = \mathsf{0}$,

\item\label{abcd-c3}
$\mathsf B_0\,\mathsf Q^{-1}\mathsf B_0^*$ is invertible with $d_0$ positive and $d_0$ negative eigenvalues,

\item \label{abcd-c4}
$\widehat{\mathcal{P}}(z) \mathsf Q_0^{-1} {\widehat{\mathcal{P}}}(z^*)^* = \mathsf{0}$ for all $z \in \mathbb{C}$, where $\mathsf{Q}_0=\bigl(\mathsf{B}_0 \mathsf{Q}^{-1} \mathsf{B}_0^*\bigr)^{-1}$,

\item \label{abcd-c5}
 $\rank {\widehat{\mathcal{P}}}(z) = d_0$ for all $z \in \mathbb{C}$,

\item \label{abcd-c6}
$\rank \widehat{\mathsf{P}}_{\infty} = d_0$, where
$\widehat{\mathsf{P}}_{\infty} =\lim_{z \rightarrow \infty}\diag\bigl(z^{-\mu_1}, \ldots, z^{-\mu_{d_0}}\bigr)\widehat{\mathcal{P}}(z)$,

\item \label{abcd-c7}
the row degrees of $\widehat{\mathcal{P}}(z)$ satisfy
\[
p = \deg \widehat{\mathcal{P}}(z)\bigl.\bigr|_1 = \mu_1 \geq \cdots \geq \deg \widehat{\mathcal{P}}(z)\bigl.\bigr|_{d_0} = \mu_{d_0} \geq 1.
\]
\end{enumerate}
\end{theorem}

\begin{remark} \label{rem1}
The matrix polynomial on the right-hand side of \eqref{reprP} belongs to $\mathbb{P}_{\mathsf{Q}}$ and hence this matrix polynomial and $\mathcal{P}(z)$ are equivalent. By Lemma~\ref{reverse} the same holds for their row-reversals, that is there exists a unimodular $d \times d$ matrix polynomial $\mathcal{V}(z)$ such that
\begin{equation}\label{reprR}
\mathcal{V}(z)\mathcal{R}(z) =
 \begin{bmatrix}  \widehat{\mathcal{R}}(z) & 0 \\[5pt] 0 & \mathsf{I}_{d-d_0} \end{bmatrix}
\begin{bmatrix}  \mathsf B_0 \\[5pt] \mathsf B  \end{bmatrix},
\end{equation}
where $\mathcal R(z)$ (respectively, $\widehat{\mathcal{R}}(z)$) is the row-reversal of $\mathcal{P}(z)$ (respectively, $\widehat{\mathcal{P}}(z)$). According to Lemma~\ref{4statements}(\ref{4statements-i3}) the entries of $\mathcal{W}(z)$ in \eqref{reprP} and $\mathcal{V}(z)$ in \eqref{reprR} satisfy \eqref{coeff}.
\end{remark}

\begin{remark} \label{rem2}
The properties \eqref{abcd-c4}-\eqref{abcd-c7} in Theorem~\ref{modelP} imply that $\widehat{\mathcal{P}}(z)\in {\mathbb P}_{\mathsf{Q}_0}$. Hence \eqref{eq-i-t-m-e} and (\ref{abcd-c7}) yield that for any $\mathbb{W} \subseteq \mathbb{C}$ with $\# \mathbb W \gt\deg \widehat{\mathcal{P}}(z)$
\begin{equation}\label{={0}}
\bigcap_{z \in \mathbb W}\nul \widehat{\mathcal{P}}(z)=\{0\}.
\end{equation}
\end{remark}

The theorem below is a reformulation of \cite[Lemmas~3.4 and~3.5]{CDR}. These lemmas were originally stated for relations $S$ and $T$ in a Hilbert space, but using indefinite inner products as in \cite[Subsection~2.1]{CDCAOT} their validity can be established in an indefinite setting as well. We have included  (\ref{list-iB})(\ref{list-iB-i3}), a classical result, to make the theorem complete. The Hilbert space version of (\ref{list-iB})(\ref{list-iB-i3}) is well known; see \cite[Theorem~XII.4.30]{DFII}, and \cite[Theorem~2.3]{CDCAOT} for the Krein space version. Special cases of the theorem below will be used to construct a linearization $\widetilde{A} \in \mathbb{A}_{S,\mathsf{b}}$ of a given $\BEP$ with $\mathcal{P}(z) \in \mathbb{P}_{\mathsf{Q}}$  and to prove a converse construction; see Theorem~\ref{X}.

\begin{theorem} \label{list}
Let $S$ be a closed symmetric linear relation in a Krein space $\mathfrak H$ with defect numbers $d^{\pm} \in \mathbb{N}$, and set  $\delta=d^++d^-$. Let $\tau$ be an integer with $0 < \tau < \delta $, and let $\mathsf{b}: S^*
\rightarrow \mathbb C^\delta$ be a boundary mapping for $S$ with Gram matrix $\mathsf{Q}$. In what follows $T$  is a closed linear relation in $\mathfrak H$.

\begin{enumerate}
\renewcommand*\theenumi{\Roman{enumi}}
\renewcommand*\labelenumi{\rm{(\theenumi)}}
\setlength{\itemsep}{1pt}%

\item \label{list-iA}
The following statements are equivalent.

\begin{enumerate}
\renewcommand*\theenumii{\arabic{enumii}}
\renewcommand*\labelenumii{\rm{({\small \theenumii})}}
\setlength{\itemsep}{1pt}%
\item \label{list-iia}
$S \subset T \subset S^*$ and $\dim T/S = \tau$.
\item \label{list-iib}
There exists a $(\delta-\tau)\times \delta$ matrix $\mathsf A$ with full rank $\delta-\tau$ such that
\[
 T = \Bigl\{ \{f,g\} \in S^*\,:\, \mathsf A \mathsf b (\{f,g\})=0 \Bigr\}.
\]
\item \label{list-iic}
There exists a $\tau\times \delta$ matrix $\mathsf B$ with full rank $\tau$ such that
\[
T^* = \Bigl\{ \{f,g\} \in S^*\,:\, \mathsf{B} \mathsf{b}(\{f,g\})=0 \Bigr\}.
\]
\end{enumerate}

\item \label{list-iB}
Assume {\rm (\ref{list-iia})-(\ref{list-iic})} in {\rm (\ref{list-iA})}. The following statements hold.

\begin{enumerate}
\renewcommand*\theenumii{\alph{enumii}}
\renewcommand*\labelenumii{\rm{(\theenumii)}}
\setlength{\itemsep}{1pt}%
\item \label{list-iB-i1}
$\mathsf B \mathsf Q^{-1}\mathsf A^* = \mathsf{0}$, and $\mathsf A$ and $\mathsf B$ are unique up to multiplication from the left by an invertible matrix.

\item \label{list-iB-i2}
If there exists a $\tau \times \delta$ matrix $\mathsf C$ with full rank $\tau$ such that $\mathsf{C} \mathsf{Q}^{-1}\mathsf{A}^* = \mathsf{0}$ and
$V=\bigl\{\{f,g\} \in S^*\,:\, \mathsf{C} \mathsf{b}(\{f,g\}) = 0\bigr\}$, then $V=T^*$.

\item \label{list-iB-i3}
$T=T^*$ if and only if $d^+ = d^- = \tau$ and $\mathsf{A} \mathsf{Q}^{-1}\mathsf{A}^* = \mathsf{0}$.

\item \label{list-iB-i4}
$T \subseteq T^*$ if and only if $\mathsf{B} \mathsf{Q}^{-1}\mathsf{B}^* = \mathsf{0}$.
\end{enumerate}

\item \label{list-iC}
Assume {\rm (\ref{list-iia})-(\ref{list-iic})} in {\rm (\ref{list-iA})} and $T \subset T^*$, strict inclusion. Then the following statements hold.

\begin{enumerate}
\renewcommand*\theenumii{\alph{enumii}}
\renewcommand*\labelenumii{\rm{(\theenumii)}}
\setlength{\itemsep}{1pt}
\item \label{list-iC-i1}
$\tau \leq \min\{d^+,d^-\}$ and the defect numbers of $T$ are $e^\pm =d^\pm-\tau$.
\item \label{list-iC-i2}
The full rank matrix $\mathsf A$ can be written in the block form $\begin{bmatrix} \mathsf{B}_0 \\ \mathsf{B}\end{bmatrix}$, where the submatrix $\mathsf B_0$ satisfies the following properties:

\begin{enumerate}
\renewcommand*\theenumiii{\roman{enumiii}}
\renewcommand*\labelenumiii{\rm{(\theenumiii)}}

\item \label{list-iC-i2-i1}
$\mathsf B_0$ is a $\rho \times \delta$ matrix and $\rank \mathsf B_0=\rho$, where $\rho=e^+ + e^-$;

\item \label{list-iC-i2-i2}
$\mathsf{B} \mathsf{Q}^{-1} \mathsf{B}_0^* = \mathsf{0}$;

\item \label{list-iC-i2-i3}
$\mathsf{B}_0 \mathsf{Q}^{-1} \mathsf{B}_0^*$ is an invertible $\rho \times \rho$ matrix;

\item \label{list-iC-i2-i4}
$\mathsf{b}_0 := \mathsf{B}_0 \mathsf{b}\bigl.\bigr|_{T^*}$ is a boundary mapping for $T$ with Gram matrix $\mathsf{Q}_0 := \bigl(\mathsf{B}_0\mathsf{Q}^{-1}\mathsf{B}_0^*\bigr)^{-1}$.
\end{enumerate}
\end{enumerate}

\end{enumerate}
\end{theorem}

\begin{remark}
Although Theorem~\ref{list} is stated for \(0 < \tau < \delta\), notice that part~(\ref{list-iC}) also holds for \(\tau=0\). In that case \(T=S\), \(\rho=\delta\), and one may take \(\mathsf A=\mathsf B_0=\mathsf{I}_\delta\); hence \(\mathsf b_0=\mathsf b\) and \(\mathsf Q_0=\mathsf Q\). The matrices \(\mathsf{A}, \mathsf{B},\) and \(\mathsf{B}_0\) in Theorem~\ref{list} are unique if we require them to be in reduced row echelon form.
\end{remark}

The theorems in this section are connected because they describe symmetric extensions of $S$ and their adjoints in terms of a pair of matrices $\mathsf B_0$ and $\mathsf B$. In Step 2 of Theorem~\ref{X} the symmetric linear relation $S_0=\widetilde A \cap \mathfrak H^2$ and its adjoint $S_0^*=\big\{\{\widetilde P_{\mathfrak H}f,\widetilde P_{\mathfrak H} g\}\,:\, \{f,g\}\in \widetilde A\big\}$ satisfy $S\subset S_0 \subset S_0^* \subset S^*$ and can be described by Theorem~\ref{list} by matrices $\mathsf{B}_0$ and $\mathsf{B}$. On the other hand in Step 3 of Theorem~\ref{X} where the starting point is a given  polynomial $\mathcal{P}(z) \in \mathbb{P}_{\mathsf Q}$, the representation \eqref{reprP} and the intersection \eqref{={0}} imply that
\begin{align*}
S_0 & = \bigcap_{z \in \mathbb W}  T_{S, \mathsf b,\mathcal{P}}(z) \\
& = \bigcap_{z \in \mathbb W} \Big\{\{f,g\} \in S^*\,:\,
\mathcal{P}(z) \mathsf{b}(\{f,g\}) = 0 \Big\} \\
& = \bigcap_{z \in \mathbb W} \Big\{\{f,g\} \in S^*\,:\, \mathsf{B}  \mathsf{b}(\{f,g\}) = 0, \ \widehat{\mathcal P}(z)\mathsf{B}_0 \mathsf{b}(\{f,g\}) = 0 \Big\}\\
& = \bigg\{\{f,g\} \in S^*\,:\, \begin{bmatrix} \mathsf B_0 \\ \mathsf B \end{bmatrix} \mathsf{b}(\{f,g\})=0 \bigg\}
\end{align*}
and so $S_0$ and $S_0^*$ satisfy \eqref{eq-S0B0B} and \eqref{S0*}.

The following corollary, and in particular its proof, establishes a further connection between the two theorems.

\begin{corollary}\label{co-tt}
Let \(\mathsf{P}\in \mathbb{P}_{\mathsf{Q}}\) be a constant \(d\times 2d\) matrix, and denote by \(\mathsf{B}\) the \((d-d_0)\times 2d\) matrix consisting of the bottom \(d-d_0\) rows of \(\mathsf{P}\). Then there exist an invertible \(d\times d\) matrix \(\mathsf{W}\), a \(d_0\times 2d_0\) matrix \(\widehat{\mathsf{P}}\) of rank \(d_0\), and a \(2d_0\times 2d\) matrix \(\mathsf{B}_0\) such that
\begin{equation} \label{eq-co-tt}
\mathsf{W}\mathsf{P}=
\begin{bmatrix}
\widehat{\mathsf{P}} & \mathsf{0}\\
\mathsf{0} & \mathsf{I}_{d-d_0}
\end{bmatrix}
\begin{bmatrix}
\mathsf{B}_0\\
\mathsf{B}
\end{bmatrix},
\end{equation}
and
\begin{enumerate}
\renewcommand*\theenumi{\roman{enumi}}
\renewcommand*\labelenumi{\rm{(\theenumi)}}
\item\label{co-tt-c1}
the \((d+d_0)\times 2d\) matrix \(\begin{bmatrix}\mathsf{B}_0\\ \mathsf{B}\end{bmatrix}\) has full row rank \(d+d_0\),
\item\label{co-tt-c2}
\(\mathsf{B}\mathsf{Q}^{-1}\mathsf{B}^*=\mathsf{0}\) and \(\mathsf{B}_0\mathsf{Q}^{-1}\mathsf{B}^*=\mathsf{0}\),
\item\label{co-tt-c3}
\(\mathsf{B}_0\,\mathsf{Q}^{-1}\mathsf{B}_0^*\) is invertible with \(d_0\) positive and \(d_0\) negative eigenvalues,
\item\label{co-tt-c4}
\(\widehat{\mathsf{P}}\,\mathsf{Q}_0^{-1}\widehat{\mathsf{P}}^{*}=\mathsf{0}\), where \(\mathsf{Q}_0=\bigl(\mathsf{B}_0\mathsf{Q}^{-1}\mathsf{B}_0^*\bigr)^{-1}\).
\end{enumerate}
\end{corollary}

\begin{proof}
We use Theorem~{\rm\ref{list}} where we assume that \(d^+ = d^- = d\) and \(\tau = d - d_0\) with \(0 \lt d_0 \lt d\). Since \(\mathsf{P}\mathsf{Q}^{-1}\mathsf{P}^*=\mathsf{0}\) we have  \(\mathsf{B}\mathsf{Q}^{-1}\mathsf{B}^*=\mathsf{0}\).  Let \(T\) be the symmetric closed linear relation  extension of \(S\) such that
\[
T^* = \Bigl\{ \{f,g\} \in S^* : \mathsf{B}\mathsf{b}(\{f,g\}) = 0 \Bigr\}.
\]
By Theorem~\ref{list}(\ref{list-iC})(\ref{list-iC-i1}), the defect numbers of \(T\) are \(d_0\). By part~(\ref{list-iC})(\ref{list-iC-i2}) of the same theorem, there exists a \(2d_0\times 2d\) matrix \(\mathsf B_0\) such that (\ref{co-tt-c1}) and (\ref{co-tt-c2}) hold and \(\mathsf B_0\mathsf Q^{-1}\mathsf B_0^*\) is invertible. Further, $\mathsf{b}_0 = \mathsf{B}_0 \mathsf{b}\bigl.\bigr|_{T^*}$ is a boundary mapping for $T$ with Gram matrix $\mathsf{Q}_0 =  \bigl(\mathsf{B}_0\mathsf{Q}^{-1}\mathsf{B}_0^*\bigr)^{-1}$. Therefore \(\mathsf{B}_0\,\mathsf{Q}^{-1}\mathsf{B}_0^*\) has \(d_0\) positive and \(d_0\) negative eigenvalues.

Set
\begin{equation} \label{eq-AP}
A = \Bigl\{ \{f,g\} \in S^* : \mathsf{P}\mathsf{b}(\{f,g\}) = 0 \Bigr\}.
\end{equation}
Then \(A\) is self-adjoint by Theorem~\ref{list}(\ref{list-iB})(\ref{list-iB-i3}).
Since \(\mathsf{B}\) consists of the bottom \(d-d_0\) rows of \(\mathsf{P}\) we have
\[
S \subset T \subset A \subset T^* \subset S^*.
\]
As \(A\) is a self-adjoint extension of \(T\), there exists a \(d_0\times 2d_0\) matrix \(\widehat{\mathsf{P}}\) of rank \(d_0\) such that \(\widehat{\mathsf{P}} \mathsf{Q}_0^{-1} \widehat{\mathsf{P}}^* = \mathsf{0}\) and
\[
A = \Bigl\{ \{f,g\} \in T^* : \widehat{\mathsf{P}} \mathsf{b}_0(\{f,g\}) = 0 \Bigr\}.
\]
Rewriting the last equality using the preceding formulas for \(T^*\) and \(\mathsf{b}_0\) we obtain
\begin{align*}
  A & =  \Bigl\{ \{f,g\} \in S^* : \mathsf{B} \mathsf{b}(\{f,g\}) = 0,  \widehat{\mathsf{P}}\mathsf{B}_0 \mathsf{b}(\{f,g\}) = 0 \Bigr\} \\
 & =  \Bigl\{ \{f,g\} \in S^* : \begin{bmatrix}
\widehat{\mathsf{P}} & \mathsf{0}\\
\mathsf{0} & \mathsf{I}_{d-d_0}
\end{bmatrix}
\begin{bmatrix}
\mathsf{B}_0\\
\mathsf{B}
\end{bmatrix}\mathsf{b}(\{f,g\}) = 0 \Bigr\}.
\end{align*}
Since \(\mathsf{b}: S^* \to \mathbb{C}^{2d}\) is a surjection, it follows from \eqref{eq-AP} that the matrix \(\mathsf{P}\) and the matrix product on the last displayed line above have the same null space. Therefore there exists an invertible \(d\times d\) matrix \(\mathsf{W}\) such that \eqref{eq-co-tt} holds.
\end{proof}

\section{The boundary eigenvalue problem \texorpdfstring{$\BEP\mkern-2mu\bigl(S, \mathsf{b}, \mathcal{P}(z)\bigr)$}{BEP(S,b,P(z))}} \label{RegSpec}

\subsection{Partition of the complex plane}\label{partition}
We recall the well-known partition of the complex plane into four disjoint sets: the \emph{resolvent set} and the \emph{point, continuous and residual spectrum} of a closed linear relation $R$ in a Krein space $\mathfrak K$ which, in the same order, are defined as follows (see \cite[p.~106]{CoSn}):
\begin{align*}
\rho(R)&= \Bigl\{
\lambda \in \mathbb C : \nul(R-\lambda I_{\mathfrak K}) = \{0\}, \
(R-\lambda I_{\mathfrak K})^{-1}\ \text{is a bounded operator on} \ \mathfrak{K}
\Bigr\},
\\
\sigp(R)
& =\Bigl\{ \lambda \in \mathbb C: \nul(R-\lambda I_{\mathfrak K}) \neq \{0\}\Bigr\},
\\
\sigc(R)
& = \Bigl\{
\lambda \in \mathbb C : \nul(R-\lambda I_{\mathfrak K}) = \{0\},\
\overline{\ran}(R-\lambda I_{\mathfrak K}) = \mathfrak K, \
\ran(R-\lambda I_{\mathfrak K})\neq \mathfrak K \Bigr\}, \\
& = \Bigl\{
\lambda \in \mathbb C : \nul(R-\lambda I_{\mathfrak K}) = \{0\},\
\overline{\ran}(R-\lambda I_{\mathfrak K})=\mathfrak K, \
\lambda \not \in \rho(R)
\Bigr\},
\\
\sigr(R)
&= \Bigl\{
\lambda \in \mathbb C : \nul(R-\lambda I_{\mathfrak K}) = \{0\},\ \overline{\ran}(R -\lambda I_{\mathfrak K})\neq \mathfrak{K} \Bigr\}.
\end{align*}
We include $\infty$ in the sets $\rho(R)$, $\sigp(R)$, $\sigc(R)$ and $\sigr(R)$ if and only if \(0\) belongs to $\rho(R^{-1})$, $\sigp(R^{-1})$, $\sigc(R^{-1})$, and $\sigr(R^{-1})$, respectively. We use the same notation for these extended sets. Thus
\begin{equation}\label{C4sets}
\overline{\mathbb{C}} = \rho(R) \cup \sigp(R) \cup \sigc(R) \cup\sigr(R), \quad \text{disjoint union.}
\end{equation}
For $\lambda \in \rho(R)$ the resolvent operator $(R-\lambda I_{\mathfrak K})^{-1}$ is a bounded operator on $\mathfrak K$. If $\infty \in \rho(R)$, then the resolvent operator of $R$ at $\infty$ is by definition the resolvent operator of $R^{-1}$ at $0$: $(R^{-1}-0 I_{\mathfrak K})^{-1} = R$ and $R$ is a bounded operator on $\mathfrak K$.

Furthermore, \(\infty \in \sigc(R)\) if and only if \(R\) is a densely defined operator with \(\dom R \subset \mathfrak{K}\), a proper subset.

The points in $\rho(R)$ are called \emph{regular points} of $R$. The set $\sigma(R):= \overline{\mathbb C}\setminus \rho(R)$ is called the \emph{spectrum} of $R$ and its elements  are called \emph{spectral points} of $R$. The points in $\sigp(R)$ are called \emph{eigenvalues} of $R$. Hence $\lambda \in \mathbb C$ is an eigenvalue of $R$ if there is a nonzero vector $h_0 \in \mathfrak K$ such  that $\{h_0,\lambda  h_0\} \in R$; $ h_0$ is called an \emph{eigenvector of $R$ corresponding to  $\lambda$}. A sequence
\begin{equation} \label{eqJcA}
h_0, h_1, \ldots, h_n \in \mathfrak{K}
\end{equation}
is a \emph{Jordan chain at the eigenvalue $\lambda$ of $R$} if
\[
h_{-1} = 0, \quad h_0 \neq 0, \quad  \bigl\{h_k, \lambda h_k+ h_{k-1} \bigr\} \in R \quad \text{for all} \quad k \in \{0,1,\ldots,n\};
\]
here $h_0$ is an eigenvector of $R$ at $\lambda$ and the vectors $h_1, \ldots, h_n$ are called \emph{generalized eigenvectors of $R$}. The number $n+1$ is the \emph{maximal length of the Jordan chain} \eqref{eqJcA} if and only if $h_n \not\in \ran\bigl(R - \lambda I_{\mathfrak K}\bigr)$. The point $\infty$ is an eigenvalue of $R$ if there is a nonzero vector $h_0 \in \mathfrak K$ such that $\{0, h_0\} \in R$; $ h_0$ is called an \emph{eigenvector of $R$ corresponding to $\infty$}. The sequence in  \eqref{eqJcA}
is a \emph{Jordan chain at the eigenvalue $\infty$ of $R$} if
\[
h_{-1} = 0, \quad h_0 \neq 0, \quad  \{h_{k-1}, h_k\} \in R \quad \text{for all} \quad k \in \{0,1,\ldots,n\};
\]
here $h_0$ is an eigenvector of $R$ at $\infty$ and the vectors $h_1, \ldots, h_n$ are called \emph{generalized eigenvectors of $R$}. The number $n+1$ is the \emph{maximal length of the Jordan chain} \eqref{eqJcA} if and only if $h_n \not\in \dom R $. Clearly, $\infty$ is an eigenvalue of $R$ and has a Jordan chain associated with it if and only if $0$ is an eigenvalue of $R^{-1}$ with the same Jordan chain.

We continue with a simple lemma.

\begin{lemma} \label{lemR}
Assume $R$ is a self-adjoint linear relation in a Krein space $\mathfrak{K}$. Then
\begin{equation} \label{eq-sr}
\sigr(R)\subseteq \mathbb{C}\setminus \mathbb{R},
\end{equation}
and the following equalities are equivalent:
\begin{enumerate}
\renewcommand*\theenumi{\roman{enumi}}
\renewcommand*\labelenumi{\rm{(\theenumi)}}
\setlength{\itemsep}{1pt}
\item \label{pr-rs-i11}
$\sigp(R) =\sigp(R)^*$.
\item \label{pr-rs-i12}
$\sigr(R)=\emptyset$.
\item \label{pr-rs-i13}
$\overline{\mathbb C}=\rho(R) \cup \sigp(R)\cup \sigc(R)$, \ disjoint union.
\end{enumerate}
Moreover, these equalities are equivalent to the corresponding ones with \(R^{-1}\) in place of \(R\).
\end{lemma}

\begin{proof}
We first show $\infty \not\in \sigr(R)$. If $\infty \in \sigr(R)$, then $0\in \sigr(R^{-1})$ which implies
\[
\mathfrak K \neq \overline{\ran}\, R^{-1}=\overline{\dom}\, R,
\]
hence there is a nonzero vector $f \in (\dom\,R)^\perp$. It follows that $\{0, f\} \in R^*=R$. Hence $0 \in \sigp(R^{-1}) \cap \sigr(R^{-1})=\emptyset$. This contradiction implies that $\infty \not\in \sigr(R)$.

It follows that $\infty \not\in \sigr(R^{-1})$, because with $R$ also $R^{-1}$ is self-adjoint. Hence $0\notin \sigr(R)$. Let $\lambda\in\mathbb{R}$. Replacing $R$ by the self-adjoint relation $R-\lambda I_{\mathfrak{K}}$ we find that $0\notin \sigr(R-\lambda I_{\mathfrak{K}})$; hence, equivalently, $\lambda \notin \sigr(R)$. We conclude that $\sigr(R)\cap(\mathbb{R}\cup\{\infty\})=\emptyset$,  that is, \eqref{eq-sr} holds.

In what follows, we use the fact that $\nul(R-\lambda I_{\mathfrak{K}})$ is closed in $\mathfrak{K}$ and apply the well-known identities: For all $\lambda \in \mathbb{C}$ we have 
\begin{equation}\label{wellknown}
\nul(R-\lambda I_{\mathfrak{K}})
= \bigl(\ran(R-\lambda^* I_{\mathfrak{K}})\bigr)^{\perp}\mkern-2mu, \quad \overline{\ran}(R-\lambda I_{\mathfrak{K}})
= \bigl(\nul(R-\lambda^* I_{\mathfrak{K}})\bigr)^{\perp}\mkern-2mu.
\end{equation}

Assume (\ref{pr-rs-i11}). To prove (\ref{pr-rs-i11})$\Rightarrow$(\ref{pr-rs-i12}), we prove its contrapositive. Assume that $\sigr(R) \neq \emptyset$ and $\lambda \in \sigr(R)$. Then, by what we just proved, $\lambda \in \mathbb{C}\setminus\mathbb{R}$, and $\lambda \notin  \sigp(R)$. By the second equality in \eqref{wellknown}, $\nul(R-\lambda^* I_{\mathfrak{K}}) = \bigl(\overline{\ran}(R-\lambda I_{\mathfrak{K}})\bigr)^{\perp}\neq \{0\}$. Hence $\lambda^* \in \sigp(R) = \sigp(R)^*$ whereas $\lambda \notin  \sigp(R)$, that is, (\ref{pr-rs-i11}) is not true.

Assume (\ref{pr-rs-i12}). Then (\ref{pr-rs-i13}) follows from the partition \eqref{C4sets}.

Assume (\ref{pr-rs-i13}). Then (\ref{pr-rs-i11}) follows from the fact that the sets $\rho(R)$ and $\sigc(R)$ are symmetric with respect to the real axis. The latter follows from the former, the equalities in \eqref{wellknown} and $\infty=\infty^*$.

The last statement follows from the fact that (\ref{pr-rs-i11}) holds if and only if \(\sigp(R^{-1}) =\sigp(R^{-1})^*\).
\end{proof}

\begin{remark} \label{re-defop}
The equivalent conditions in Lemma~\ref{lemR} play an essential role in the proofs of spectral equivalence in Section~\ref{speceq}. A particularly important class of operators satisfying (\ref{pr-rs-i11})-(\ref{pr-rs-i13})  in Lemma~\ref{lemR} consists of definitizable self-adjoint operators in Krein spaces. By definition their resolvent sets are nonempty. This notion was introduced by Heinz Langer in~\cite{Langer1965} and has been studied extensively in subsequent work; see, for example,~\cite{Langer1982}. It was extended to
self-adjoint linear relations in~\cite{DS2}. Moreover, if $R$ is a definitizable self-adjoint relation in $\mathfrak K$, then by \cite[Proposition~5.2]{DS2} the nonreal spectrum of $R$ is finite, symmetric with respect to the real axis, and consists only of eigenvalues. Thus (\ref{pr-rs-i11}) and, consequently,  (\ref{pr-rs-i12}) and (\ref{pr-rs-i13}) in Lemma~\ref{lemR} hold for \(R\).
\end{remark}

\subsection{The \texorpdfstring{$\BEP\mkern-2mu\bigl(S, \mathsf{b}, \mathcal{P}(z)\bigr)$}{BEP(S,b,P(z))}, its spectrum and Jordan chains} \label{bep}

Let $S$ be a closed symmetric linear relation in a Krein space
$\bigl(\mathfrak H,\kip_{\mathfrak H}\bigr)$ with equal defect numbers
$d\in\mathbb N$ and adjoint $S^*$. Denote by $\mathsf b: S^* \rightarrow \mathbb C^{2d}$ a boundary mapping for $S$. By definition it is a linear and surjective mapping with kernel $\nul \mathsf b=S$. With $\mathsf b$ there is associated
the Lagrange identity
\[
\iu \big( [f,k ]_\mathfrak{H}- [ g, h ]_\mathfrak{H} \big) =
 \mathsf{b}(\{h,k\})^*\mathsf{Q}\mathsf{b}(\{f,g\}) \quad 
  \text{for all} \quad \{f,g\}, \{h,k\}\in S^*,
 \]
in which $\mathsf Q$ is an invertible self-adjoint $2d\times 2d$ matrix with $d$
positive and $d$ negative eigenvalues, called the \emph{Gram matrix for} $\mathsf b$.

In this subsection we consider the following Boundary Eigenvalue Problem associated with $S$, $\mathsf b$ and $\mathcal{P}(z) \in \mathbb{P}_{\mathsf{Q}}$.

\medskip

\noindent \textbf{Boundary Eigenvalue Problem} $\BEP\mkern-2mu\bigl(S, \mathsf b, \mathcal{P}(z)\bigr)$\textbf{.} \emph{For all $\lambda \in \mathbb C$ and  all $h \in \mathfrak H$ determine the existence and the uniqueness of a solution $\{f,g\} \in \mathfrak H^2$ (or simply $f \in \mathfrak H$) of the system}
\begin{equation} \label{BEPsys2}
\{f,g\} \in S^*, \quad g-\lambda f= h \quad \text{and}   \quad
\mathcal{P}(\lambda)\mathsf{b}(\{f,g\}) = 0.
\end{equation}
{\em For $\lambda=\infty$ and all $h \in \mathfrak H$ determine the existence and uniqueness of a solution $\{h, f\} \in \mathfrak H^2$ (or simply $f \in \mathfrak H$) of the system}
\begin{equation} \label{BEPsys2inf}
\{h, f\} \in S^* \quad \text{and}   \quad
\mathsf P_{\mkern-2mu\infty}\mathsf{b}(\{h, f\}) = 0.
\end{equation}

\medskip

When the context is clear we abbreviate $\BEP\mkern-2mu\bigl(S, \mathsf b, \mathcal{P}(z)\bigr)$ to $\BEP$.  With the $\BEP$ we associate the family $ T_{S,\mathsf{b},\mathcal{P}} (z)$ of closed linear extensions of $S$ in $\mathfrak H$ by
\begin{equation}\label{preShtraus}
T_{S,\mathsf{b},\mathcal{P}}(z)
= \Big\{ \{f,g\}\in S^* \, :\, \mathcal{P}(z)\mathsf b(\{f,g\})=0 \Big\}, \quad z \in \overline{\mathbb C}.
\end{equation}
As with $\BEP$, when the context is clear we abbreviate $T_{S,\mathsf{b},\mathcal{P}}(z)$ to $T_{\mathcal{P}}(z)$.

\begin{lemma} \label{le-PeSSts}
Let \(\mathcal{P}(z), \mathcal{S}(z) \in \mathbb{P}_{\mathsf{Q}}\). Then \(T_{S,\mathsf{b},\mathcal{P}}(z) = T_{S,\mathsf{b},\mathcal{S}}(z)\) for all \(z \in \mathbb{C}\) if and only if \(\mathcal{P}(z) \sim \mathcal{S}(z)\).
\end{lemma}
\begin{proof}
The ``if'' part of the proof follows from Definition~\ref{def-bbPeq} and the fact that \(\mathcal{W}(z)\) is unimodular.

For the ``only if'' part, assume that \(T_{\mathcal{P}}(z) = T_{\mathcal{S}}(z)\) for all \(z \in \mathbb{C}\).
Since \(\mathsf{b}: S^{*} \to \mathbb{C}^{2d}\) is a surjection, it follows from this assumption that for each \(z \in \mathbb{C}\) we have \(\nul \mathcal{P}(z) = \nul  \mathcal{S}(z)\). By Lemma~\ref{le-eqnuls} there exists a unimodular \(d\times d\) polynomial \(\mathcal{W}(z)\) such that \(\mathcal{W}(z) \mathcal{S}(z) \equiv \mathcal{P}(z)\); that is, \(\mathcal{P}(z) \sim \mathcal{S}(z)\).
\end{proof}

\begin{definition} \label{def-sp-BEP}
Let $\BEP =\BEP\mkern-2mu\bigl(S,\mathsf b,\mathcal P(z)\bigr)$ and let
$\lambda\in\overline{\mathbb C}$.
\begin{enumerate}
\renewcommand*\theenumi{\alph{enumi}}
\renewcommand*\labelenumi{\rm{(\theenumi)}}
\setlength{\itemsep}{1pt}

\item \label{def-sp-BEP-i1}
If $\lambda\in\rho\bigl(T_{S,\mathsf b,\mathcal P}(\lambda)\bigr)$, then $\lambda$ is said to belong to the \emph{resolvent set} $\rho(\BEP)$ of the $\BEP$ and it is called a regular point of it. In this case, the resolvent operator of the $\BEP$ at $\lambda$ is, by definition, the resolvent operator of $T_{S,\mathsf b,\mathcal P}(\lambda)$ at $\lambda$.

\item \label{def-sp-BEP-i2}
If $\lambda\in\sigc\bigl(T_{S,\mathsf b,\mathcal P}(\lambda)\bigr)$, then $\lambda$ is said to belong to the \emph{continuous spectrum} $\sigc(\BEP)$ of the $\BEP$.

\item \label{def-sp-BEP-i3}
If $\lambda\in\sigp\bigl(T_{S,\mathsf b,\mathcal P}(\lambda)\bigr)$, then $\lambda$ is said to belong to the \emph{point spectrum} $\sigp(\BEP)$ of the $\BEP$, and is called an \emph{eigenvalue} of it. The eigenelements of the $\BEP$ at $\lambda$ are, by definition, the eigenelements of $T_{S,\mathsf b,\mathcal P}(\lambda)$ corresponding to $\lambda$.

\item \label{def-sp-BEP-i4}
If $\lambda\in\sigr\bigl(T_{S,\mathsf b,\mathcal P}(\lambda)\bigr)$, then $\lambda$ is said to belong to the \emph{residual spectrum} $\sigr(\BEP)$ of the $\BEP$.

\end{enumerate}
\end{definition}

\begin{corollary} \label{co-TPadjoint}
If $\mathcal{P}(z)\in \mathbb{P}_{\mathsf Q}$, then in
$(\mathfrak H, \kip_{\mathfrak H})$
\begin{equation}\label{eq-TPadjoint}
T_{S,\mathsf b,\mathcal P}(z)^{*}=T_{S,\mathsf b,\mathcal P}(z^*) 
\quad  \text{for all} \quad z\in \overline{\mathbb C}.
\end{equation}
In particular, for $z \in \mathbb{R}\cup\{\infty\}$ the relation $T_{S,\mathsf b,\mathcal P}(z)$ is self-adjoint in $\mathfrak H$, and
\[
\sigr(\mathrm{BEP})\subseteq \mathbb C\setminus \mathbb R.
\]
\end{corollary}

\begin{proof}
The equality \eqref{eq-TPadjoint} follows from Theorem~\ref{list}(\ref{list-iA}) and~(\ref{list-iB}) with $T=T_{S,\mathsf b,\mathcal P}(z)$, $\mathsf{A} = \mathcal{P}(z)$, and $\mathsf{B} = \mathcal{P}(z^*)$, and from Definition~\ref{def-bbP} in which $z$ is replaced by $z^*$. It implies that $T_{S,\mathsf b,\mathcal P}(\lambda)$ is self-adjoint for $\lambda\in \mathbb R\cup\{\infty\}$. This in turn, together with \eqref{eq-sr} in Lemma~\ref{lemR} and Definition~\ref{def-sp-BEP}(\ref{def-sp-BEP-i4}), implies the inclusion.
\end{proof}

\smallskip

In the preceding definition the spectral sets of $\BEP$ are defined in terms of the spectral notions of the linear relation $T_{\mathcal P}(\lambda) = T_{S,\mathsf b,\mathcal P}(\lambda)$ at the same parameter value $\lambda$.  This is natural, since solvability of \eqref{BEPsys2} for a given $h$ is precisely the
solvability of $\bigl(T_{\mathcal P}(\lambda)-\lambda I_{\mathfrak H}\bigr)f=h$.

For Jordan chains, however, the situation is subtler. The relation $T_{\mathcal P}(\lambda)$ captures the boundary condition only at the fixed point $\lambda$, while, it turns out, the algebraic structure of the corresponding generalized eigenspace for the BEP depends on the derivatives of $\mathcal P(z)$ at $\lambda$, and the generalized eigenvectors need not belong to $\dom T_{\mathcal P}(\lambda)$.

\begin{definition} \label{def-Jc} 
For $\lambda \in \mathbb{C}$ we say that a sequence
\begin{equation} \label{eqJc}
f_0, f_1, \ldots, f_n \in \mathfrak{H}
\end{equation}
forms a \emph{Jordan chain} for the $\BEP$ at the eigenvalue $\lambda$ if $f_0\neq 0$ and with $f_{-1}=0$  the following relations hold
for all $k \in \{0,1,\ldots,n\}$:
\begin{equation*}
\{f_k, \lambda f_k + f_{k-1}\} \in S^*  \quad \text{and} \quad
\sum_{j=0}^k
\frac{1}{j!}\mathcal{P}^{(j)}\mkern-1mu(\lambda) \mkern1mu \mathsf{b}\mkern-1mu\bigl(\bigl\{f_{k-j}, \lambda f_{k-j}+f_{k-j-1}\bigr\}\bigr) = 0.
\end{equation*}
The vector $f_0$ in \eqref{eqJc} is an eigenvector and the vectors $f_1, \ldots, f_n$ are called \emph{generalized eigenvectors} of  the $\BEP$ at $\lambda$.
The number $n+1$ is called the \emph{maximal length of the Jordan
chain} if the sequence \eqref{eqJc} cannot be extended to a longer Jordan chain, that is, if $f_n \not\in \ran(S^*-\lambda I_{\mathfrak H})$ or, if $\{g,f_n\} \in S^*-\lambda I_{\mathfrak H}$ for some $g \in \mathfrak H$, then
\begin{equation*}
\mathcal{P}(\lambda)\mathsf{b}\bigl(\{g, \lambda g + f_n\}\bigr)
+ \sum_{j=1}^{n+1} \frac{1}{j!}\mathcal{P}^{(j)}\mkern-1mu (\lambda) \mkern1mu
\mathsf{b}\mkern-1mu\bigl(\bigl\{f_{n+1-j}, \lambda f_{n+1-j}+f_{n-j}\bigr\}\bigr) \neq 0.
\end{equation*}
\end{definition}

\smallskip

Next we define Jordan chains of the $\BEP\mkern-2mu\bigl(S,\mathsf b,\mathcal P(z)\bigr)$ at $\infty$. Since derivatives of $\mathcal P(z)$ at $\infty$ are not directly available, we will first show that a related boundary eigenvalue problem which involves \(S^{-1}\) and the row-reversal \(\mathcal{R}(z)\) of \(\mathcal{P}(z)\) is in some sense inverse of the boundary problem $\BEP\mkern-2mu\bigl(S,\mathsf{b},\mathcal{P}(z)\bigr)$.

It is convenient to introduce the involution  \(\iota : \mathfrak{H}^2 \to \mathfrak{H}^2\) defined by \(\iota\{f,g\} = \{g,f\}\) for all \(\{f,g\} \in \mathfrak{H}^2\). Then the inverse of \(R\) is \(\iota R\).

If $\mathsf{b}\colon S^*\to \mathbb{C}^{2d}$ is a boundary mapping of the closed symmetric linear relation $S$ in a Krein space  $\mathfrak{H}$ and $\mathsf{b}$ has Gram matrix $\mathsf{Q}$, then $\mathsf{b}\mkern0.25mu\iota \colon (\iota S)^* = S^{-*}\to \mathbb{C}^{2d}$ is a boundary mapping of $\iota S$ with Gram matrix $-\mathsf{Q}$.

The following equalities hold for every \(\lambda \in \overline{\mathbb{C}}\):
\begin{align*}
\bigl( T_{S,\mathsf{b},\mathcal{P}}(\lambda) \bigr)^{-1}
& = \iota \bigl( T_{S,\mathsf{b},\mathcal{P}}(\lambda) \bigr)\\
& = \iota\mkern0.5mu\mathsf{b}^{-1}\bigl( \nul \mathcal{P}(\lambda) \bigr) \\
& = (\mathsf{b}\mkern0.25mu\iota)^{-1}\bigl( \nul \mathcal{P}(\lambda) \bigr) \\
& = (\mathsf{b}\mkern0.25mu\iota)^{-1}\bigl( \nul \mathcal{R}(1/\lambda) \bigr).
\end{align*}
Hence
\begin{equation} \label{eq-la1ola}
\bigl( T_{S,\mathsf{b},\mathcal{P}}(\lambda) \bigr)^{-1}
= T_{S^{^{-1}}\mkern-5mu, \mkern2mu\mathsf{b}\mkern1mu\iota,\mathcal{R} }(1/\lambda) \quad \text{for all} \quad \lambda \in \overline{\mathbb{C}}.
\end{equation}
By the spectral mapping theorem, the spectral nature of $\lambda\in\overline{\mathbb C}$ for $T_{S,\mathsf b,\mathcal P}(\lambda)$ coincides with the spectral nature of $1/\lambda$ for $\bigl(T_{S,\mathsf b,\mathcal P}(\lambda)\bigr)^{-1}$. Hence, by \eqref{eq-la1ola}, it also coincides with the spectral nature of $1/\lambda$ for \(T_{S^{^{-1}}\mkern-5mu, \mkern2mu\mathsf{b}\mkern1mu\iota, \mathcal{R} }(1/\lambda)\).

This relation motivates the next definition: Jordan chains at $\infty$ for the boundary eigenvalue problem $\BEP\mkern-2mu\bigl(S,\mathsf b,\mathcal P(z)\bigr)$ are defined as the Jordan chains at \(0\) for the related, in some sense inverse boundary eigenvalue problem \(\BEP\mkern-2mu\bigl(S^{-1}, \mathsf{b}\mkern1mu\iota, \mathcal{R}(z)\bigr)\).

\begin{definition}
We say that the sequence \eqref{eqJc} forms a \emph{Jordan chain for the {\rm BEP} at the eigenvalue $\infty$} if $f_0\neq 0$ and with $f_{-1}=0$ the following relations hold for all $k \in \{0,1,\ldots,n\}$:
\begin{equation*}
\{f_{k-1},  f_k \} \in S^*  \quad \text{and} \quad
\sum_{j=0}^k \frac{1}{j!}\mathcal R^{(j)}(0)
\mathsf{b}\bigl(\{f_{k-j-1},  f_{k-j}\}\bigr)  = 0,
\end{equation*}
where $\mathcal R(z)$ is the row-reversal polynomial of $\mathcal{P}(z) \in \mathbb{P}_{\mathsf{Q}}$.
In this case the vector $f_0$ in \eqref{eqJc} is an eigenvector and the vectors $f_1, \ldots, f_n$ are called \emph{generalized eigenvectors} of  the BEP at $\infty$.
The number $n+1$ is called the \emph{maximal length of the Jordan chain} if the sequence \eqref{eqJc} cannot be extended to a longer Jordan chain, that is, if
$f_n \not\in \dom S^*$
or, if $\{f_n,g\} \in S^*$ for some $g \in \mathfrak H$, then
\begin{equation*}
\mathcal R(0)\mathsf{b}\bigl(\{f_n,g\}\bigr)
+ \sum_{j=1}^{n+1} \frac{1}{j!}\mathcal R^{(j)}(0)
\mathsf{b}\bigl(\{f_{n-j}, f_{n+1-j}\}\bigr) \neq 0.
\end{equation*}
\end{definition}

The sets $\rho(\BEP)$, $\sigc(\BEP)$ and $\sigp(\BEP)$ stay the same if $\mathcal{P}(z)$ in the boundary eigenvalue problem $\BEP\mkern-2mu\bigl(S, \mathsf b, \mathcal{P}(z)\bigr)$ is replaced by an equivalent polynomial in $\mathbb{P}_{\mathsf{Q}}$. The following lemma shows that the Jordan chains of the boundary eigenvalue problems with equivalent matrix polynomials also remain the same.

\begin{lemma} \label{Chains}
Assume that \(\mathcal P(z),\mathcal S(z)\in\mathbb P_{\mathsf Q}\) are equivalent, and let \(\lambda\in\overline{\mathbb C}\). Then the sequence
$f_0, f_1, \ldots, f_n \in \mathfrak H$ forms a (maximal) Jordan chain for the $\BEP\mkern-2mu\bigl(S, \mathsf{b}, \mathcal{P}(z)\bigr)$ at $\lambda$ if and only if it forms a (maximal) Jordan chain for the $\BEP\mkern-2mu\bigl(S, \mathsf b, \mathcal{S}(z)\bigr)$ at $\lambda$.
\end{lemma}
\begin{proof} 
We first consider the case $\lambda \in \mathbb C$. Assume $\mathcal{S}(z)=\mathcal W(z)\mathcal{P}(z)$ where $\mathcal W(z)$ is unimodular. Since $\mathcal W(z)^{-1}$ is also unimodular, it suffices to prove the ``only if'' part. Assume $f_0, f_1, \ldots, f_n \in \mathfrak{H}$ form a Jordan chain for the $\BEP\mkern-2mu\bigl(S, \mathsf b, \mathcal{P}(z)\bigr)$ at the eigenvalue $\lambda$. Then $\bigl\{f_k,\lambda f_k+f_{k-1}\bigr\} \in S^*$ and
\[
\sum_{j=0}^k \frac{1}{j!} \mathcal{P}^{(j)}\mkern-2mu(\lambda)\mathsf{b}_{k-j} = 0 
\quad \text{where} \quad 
\mathsf{b}_k = \mathsf{b}\bigl(\bigl\{f_k, \lambda f_k+f_{k-1}\bigr\}\bigr), \quad k \in \{0,1, \ldots, n\}.
\]
The proposition is proved when we show these equalities also hold if we replace $\mathcal{P}(z)$ by $\mathcal{S}(z)$. We calculate
\begin{align*}
\sum_{j=0}^k\frac{1}{j!} \mathcal{S}^{(j)}\mkern-2mu(\lambda)\mathsf{b}_{k-j}
& = \sum_{j=0}^k\sum_{l=0}^j \frac{1}{l!\,(j-l)!} \mathcal{W}^{(l)}\mkern-2mu(\lambda)\mathcal{P}^{(j-l)}\mkern-2mu(\lambda) \mathsf{b}_{k-j}\\
& = \sum_{l=0}^k \frac{1}{l!} \mathcal{W}^{(l)}\mkern-2mu(\lambda)
\Biggl( \sum_{j = l}^k \frac{1}{(j-l)!} \mathcal{P}^{(j-l)}\mkern-2mu(\lambda) \mathsf{b}_{k-j} \Biggr) \\
& = \sum_{l=0}^k \frac{1}{l!} \mathcal{W}^{(l)}\mkern-2mu(\lambda)
\Biggl(
\sum_{i = 0}^{k-l} \frac{1}{i!} \mathcal{P}^{(i)}\mkern-2mu(\lambda) \mathsf{b}_{k-l-i} \Biggr) \\
&=0.
\end{align*}
The last equality holds because for each $l \in \{0,1, \ldots, k\}$  the sum in the big parenthesis on the right is $0$.

For $\lambda=\infty$ the argument is the same after replacing $\mathcal P(z)$ and
$\mathcal S(z)$ by their row-reversals, which are equivalent by Lemma~\ref{reverse}.
Maximality is preserved as well, since a Jordan chain is maximal if and only if it cannot be extended by one further vector, and the above equivalence holds for chains of arbitrary finite length.
\end{proof}

\section{Shtraus subspaces and the class \texorpdfstring{$\mathbb{A}_{S,\mathsf{b}}$}{A{S,b}} of self-adjoint extensions} \label{classA}

\subsection{Definitions and basic properties} \label{subs-classA}

\begin{definition} \label{def-Sts}
Let $\bigl(\widetilde{\mathfrak{H}},\kip_{\widetilde{\mathfrak{H}}}\bigr)$ be a Krein space containing $\bigl(\mathfrak{H},\kip_{\mathfrak{H}}\bigr)$ as a Krein subspace. For a self-adjoint relation $\widetilde{A}$ in $\bigl(\widetilde{\mathfrak{H}},\kip_{\widetilde{\mathfrak{H}}}\bigr)$ we define the \emph{family of Shtraus subspaces} in $\mathfrak H$ associated with $\widetilde A$ by
\begin{equation} \label{Tagain}
\TtA(z) =
\left\{
\begin{array}{ll}
\Bigl\{
\{\widetilde P_\mathfrak{H} \widetilde f, \widetilde P_\mathfrak H \widetilde g\}\,:\, \{\widetilde f, \widetilde g\} \in \widetilde A,\ \widetilde g - z  \widetilde f \in \mathfrak H
\Bigr\}, &  z \in \mathbb C,
\\[3mm]
\Bigl\{
\{ \widetilde f, \widetilde P_\mathfrak H \widetilde g\}\,:\, \{\widetilde f, \widetilde g\} \in \widetilde A,\ \widetilde f \in \mathfrak H
\Bigr\} = \widetilde P_\mathfrak H \widetilde A\bigl.\bigr|_{\mathfrak H}, & z=\infty,
\end{array}
\right.
\end{equation}
where $\widetilde{P}_{\mathfrak{H}}$ is the orthogonal projection in $\bigl(\widetilde{\mathfrak{H}},\kip_{\widetilde{\mathfrak{H}}}\bigr)$ onto $\mathfrak{H}$.
\end{definition}

These subspaces are named after A.V. Shtraus \cite{Str65}, \cite{Str70}. Using the algebra of sets one can show that
\begin{equation} \label{S1}
\bigl(\TtA(z)-z I_{\mathfrak H} \bigr)^{-1} =
\widetilde{P}_{\mathfrak H}\bigl(\widetilde{A} - z I_{\widetilde{\mathfrak H}}\bigr)^{-1}\bigl.\bigr|_{\mathfrak H} \qquad \text{for all} \qquad z \in \mathbb{C}.
\end{equation}
Note that
\begin{equation}\label{incl}
\widetilde A  \cap \mathfrak H^2 \subseteq \bigcap_{z \in \overline{\mathbb C}} \TtA(z).
\end{equation}
If $\widetilde A$ in $\widetilde{\mathfrak H}$ is a self-adjoint extension of $S$ in $\mathfrak H$, which we consider in this paper, then $S\subseteq \TtA(z) \subseteq S^*$, and because of these inclusions $\TtA(z)$ is also called a \emph{Shtraus extension} of $S$. Since $S$ has finite defect numbers it is a closed subspace of $\mathfrak H^2$ for each $z \in \overline{\mathbb C}$.

\begin{definition} \label{def-bbA}
Let $S$ be a closed symmetric linear relation in a Krein space $\bigl(\mathfrak{H}, \kip_{\mathfrak H}\bigr)$ with equal defect numbers $d \in \mathbb{N}$ and adjoint $S^*$. Let $\mathsf{b}: S^* \rightarrow \mathbb C^{2d}$ be a boundary mapping for $S$ with \(2d\times 2d\) Gram matrix $\mathsf{Q}$. By $\mathbb{A}_{S,\mathsf{b}}$ we denote the set of all self-adjoint relations \(\widetilde{A}\) in a Krein space \(\bigl(\widetilde{\mathfrak H}, \kip_{\widetilde{\mathfrak{H}}}\bigr)\) such that
\begin{enumerate}
\renewcommand*\theenumi{\alph{enumi}}
\renewcommand*\labelenumi{\rm{(\theenumi)}}
\setlength{\itemsep}{1pt}

\item \label{def-bbA-i1}
$\bigl(\mathfrak H, \kip_{\mathfrak H}\bigr)$ is a Krein subspace of \(\bigl(\widetilde{\mathfrak H}, \kip_{\widetilde{\mathfrak{H}}}\bigr)\).

\item \label{def-bbA-i2}
\(\widetilde{A}\) is an extension of \(S\).

\item \label{def-bbA-i3}
The orthogonal complement \(\widehat{\mathfrak{H}}\) of \(\mathfrak{H}\) in \(\widetilde{\mathfrak{H}}\) is finite dimensional. 

\item \label{def-bbA-i4}
$\widehat{S} = \widetilde A \cap \widehat{\mathfrak H}^2$  is a (closed symmetric) operator without eigenvalues.
\end{enumerate}
\end{definition}

\begin{definition} \label{def-bbAeq}
Self-adjoint extensions $\widetilde{A}_1$ and $\widetilde{A}_2$ of $S$ in $\mathbb{A}_{S,\mathsf{b}}$ are called \emph{equivalent}, in notation $\widetilde{A}_1 \sim \widetilde{A}_2$, if (in self-evident notation) there exists an isomorphism $\Psi: \widehat{\mathfrak H}_1 \rightarrow \widehat{\mathfrak H}_2$ such that
\begin{equation} \label{eq-IPsiA1eq}
(I_{\mathfrak H} \oplus \Psi) \widetilde{A}_1 = \widetilde{A}_2(I_{\mathfrak{H}} \oplus \Psi).
\end{equation}
The equivalence class relative to \(\sim\) for which $\widetilde{A} \in \mathbb{A}_{S,\mathsf{b}}$ is a representative is denoted by \([\widetilde{A}]\). We denote the set of all \(\sim\)-equivalence classes in  \(\mathbb{A}_{S,\mathsf{b}}\) by \(\mathbb{A}_{S,\mathsf{b}}/\mathord{\sim}\).
\end{definition}

\begin{lemma} \label{le-relpr}
If $\widetilde{A}_1 \sim \widetilde{A}_2$ as in Definition~{\rm\ref{def-bbAeq}}, then \(\widetilde{A}_1\cap \mathfrak H^{2}=\widetilde{A}_2\cap \mathfrak{H}^{2}\) and \(\Psi \widehat{S}_1 = \widehat{S}_2 \Psi\) where \(\widehat{S}_j = \widetilde{A}_j\cap \widehat{\mathfrak H}_j^{2}\) for \(j\in\{1,2\}\).
\end{lemma}

\begin{proof}
Let \(\mathfrak{X}\) and \(\mathfrak{Y}\) be nonempty sets, \(A\) a relation in \(\mathfrak{X}\), \(B\) a relation in \(\mathfrak{Y}\) and let \(\Phi: \mathfrak{X} \to \mathfrak{Y}\) be a bijection. Then
\begin{equation} \label{eq-relpr}
\Phi A = B \Phi  \quad \Leftrightarrow \quad  
B = \Bigl\{ \{\Phi x_1, \Phi x_2\} \in \mathfrak{Y}^2 : \{x_1, x_2\} \in A \Bigr\}.
\end{equation}
The proof is a verification of the set equality \(A = \Phi^{-1}B\Phi\): \(\{x_1,x_2\}\in A\) if and only if \(\{x_1,y_1\}\in \Phi\) and \(\{y_1,y_2\}\in B\) and \(\{y_2,x_2\}\in \Phi^{-1}\), which is equivalent to \(y_j = \Phi x_j, j \in\{1,2\}\) and \(\{y_1,y_2\}\in B\), that is  \(\{\Phi x_1,\Phi x_2\}\in B\). Applying the equivalence in \eqref{eq-relpr} to \eqref{eq-IPsiA1eq} we obtain that \eqref{eq-IPsiA1eq} is equivalent to
\begin{equation} \label{eq-4st-i2}
\widetilde{A}_2 = \Biggl\{
{\scriptstyle
\left\{  \begin{bmatrix} f \\[1.5pt]  \Psi \widehat{f} \end{bmatrix} ,
 \begin{bmatrix} g \\[1.5pt] \Psi \widehat{g} \end{bmatrix} \right\}}
 \, : \,
{\scriptstyle
\left\{  \begin{bmatrix} f \\[1.5pt]  \widehat{f} \end{bmatrix} ,
 \begin{bmatrix} g \\[1.5pt]  \widehat{g} \end{bmatrix} \right\}} \in \widetilde{A}_1
  \Biggr\}.
\end{equation}
From \eqref{eq-4st-i2}, and thus from \eqref{eq-IPsiA1eq}, it follows that $\widetilde{A}_2\cap \mathfrak H^{2}=\widetilde{A}_1\cap \mathfrak{H}^{2}$
and
\[
\widehat{S}_2 = \left\{ \{\Psi \widehat{f},\Psi \widehat{g}\} \ : \ \{\widehat{f},\widehat{g}\}\in \widehat{S}_1 \right\},
\]
which, by \eqref{eq-relpr}, is equivalent to \(\Psi \widehat{S}_1 = \widehat{S}_2 \Psi\).
\end{proof}

In the next proposition we collect basic properties of the relations in \(\mathbb{A}_{S,\mathsf{b}}\). Part~(\ref{pr-bbAbasics-i2}) characterizes the canonical self-adjoint extensions of \(S\) in \(\mathfrak{H}\),
and part~(\ref{pr-bbAbasics-i3}) introduces a special representation of a self-adjoint relation in \(\mathbb{A}_{S,\mathsf{b}}\) and characterizes the corresponding Shtraus subspaces.

\begin{proposition} \label{pr-bbAbasics}
Let \(\widetilde{A} \in \mathbb{A}_{S,\mathsf{b}}\). The following statements hold.

\begin{enumerate}
\renewcommand*\theenumi{\Roman{enumi}}
\renewcommand*\labelenumi{\rm{(\theenumi)}}
\setlength{\itemsep}{2pt}
\item \label{pr-bbAbasics-i1}
\(S_0 := \widetilde{A} \cap \mathfrak{H}^2\) is a closed symmetric linear relation extension of \(S\) in \(\mathfrak{H}\) with equal defect numbers \(d_0\) where \(0 \leq d_0 \leq d\). The closed symmetric linear relation \(\widehat{S}\) in \(\widehat{\mathfrak{H}}\) has the same defect numbers \(d_0\).

\item \label{pr-bbAbasics-i2}
The following six statements are equivalent.

\hspace*{-15pt} \begin{minipage}[ht]{0.6\linewidth}
\begin{enumerate}
\renewcommand*\theenumii{\alph{enumii}}
\renewcommand*\labelenumii{\rm{(\theenumii)}}
\setlength{\itemsep}{1pt}
\item \label{pr-bbA-i2-i1}
\(d_0 = 0\), that is, \(S_0\) is self-adjoint.

\item \label{pr-bbA-i2-i2}
\(\TtA(z) \equiv \widetilde{A}\).

\item \label{pr-bbA-i2-i3}
\(\TtA(z)\) is constant.

 \end{enumerate}
\end{minipage}
\begin{minipage}[ht]{0.3\linewidth}
\begin{enumerate}
\setcounter{enumii}{3}
\renewcommand*\theenumii{\alph{enumii}}
\renewcommand*\labelenumii{\rm{(\theenumii)}}
\setlength{\itemsep}{1pt}

\item \label{pr-bbA-i2-i4}
\(\widehat{\mathfrak H} = \{0\}\).
\item  \label{pr-bbA-i2-i5}
\(S_0 = \widetilde{A}\).
\item  \label{pr-bbA-i2-i6}
\([\widetilde{A}] = \{ \widetilde{A} \}\).
\end{enumerate}
\end{minipage}

\item \label{pr-bbAbasics-i3}
Assume that \(0 \lt d_0 \leq d\). Then \(S_0 \subset S_0^{*}\). In the notation from Theorem~{\rm\ref{list}(\ref{list-iC})(\ref{list-iC-i2})} with \(T=S_0\), namely the \(2d_0\times 2d\) matrix \(\mathsf B_0\) and
the boundary mapping \(\mathsf b_0=\mathsf B_0\mathsf b|_{S_0^{*}}\) with Gram matrix
\(\mathsf Q_0=\bigl(\mathsf B_0\mathsf Q^{-1}\mathsf B_0^*\bigr)^{-1}\), the following statements hold:

\begin{enumerate}
\renewcommand*\theenumii{\roman{enumii}}
\renewcommand*\labelenumii{\rm{(\theenumii)}}
\setlength{\itemsep}{1pt}
\item \label{pr-bbAbasics-i3-i1}
There exists a unique boundary mapping \(\widehat{\mathsf{b}}:\widehat{S}^{*} \to \mathbb{C}^{2d_0}\) for \(\widehat{S}\) with Gram matrix \(-\mathsf{Q}_0\) such that
\begin{equation} \label{eq-tAhb}
\widetilde{A} =  \Biggl\{
{\scriptstyle \left\{  \begin{bmatrix} f \\[1.5pt] \widehat{f} \end{bmatrix} ,
 \begin{bmatrix}
g \\[1.5pt]  \widehat{g} \end{bmatrix} \right\}}
 \, :
 \begin{array}{l}
 \{f,g\} \in  S_0^{*}, \ \ \{ \widehat{f},  \widehat{g}\} \in  \widehat{S}^{*}, \\[2.5pt]
 \mathsf{b}_0(\{f, g\})
 + \widehat{\mathsf{b}} (\{ \widehat{f}, \widehat{g}\}) = 0
 \end{array} \mkern-5mu
  \Biggr\}.
\end{equation}
Moreover, the equality in \eqref{eq-tAhb} establishes a bijection between all \(\widetilde{A} \in \mathbb{A}_{S,\mathsf{b}}\) such that \(S_0 = \widetilde{A} \cap \mathfrak{H}^2\) and \(\widehat{S} = \widetilde{A} \cap \widehat{\mathfrak{H}}^2\) and all boundary mappings \(\widehat{\mathsf{b}}:\widehat{S}^{*} \to \mathbb{C}^{2d_0}\) for \(\widehat{S}\) with Gram matrix \(-\mathsf{Q}_0\).

\item \label{pr-bbAbasics-i3-i2}
Let \(\widehat{\mathcal{P}}(z) \in \mathbb{P}_{\mathsf{Q}_0}\) be a \(d_0\times 2d_0\) matrix polynomial and assume additionally that $\widehat{\mu}_j = \deg \widehat{\mathcal{P}}(z)\mkern-1mu\bigl.\bigr|_j \in \mathbb{N}$ for $j \in \{1, \ldots,d_0\}$, and \(\widehat{\mu} = \bigl(\widehat{\mu}_1, \ldots, \widehat{\mu}_{d_0}\bigr)\). Then the family of Shtraus subspaces associated with $\widetilde{A}$ defined in \eqref{eq-tAhb} is given by the formula
\begin{equation}\label{eq-SPbbA}
\TtA(z) =
\Bigl\{\{f,g\}\in S_0^{*}\,:\,\widehat{\mathcal{P}}(z) \mathsf{b}_0(\{f,g\}) = 0 \Bigr\} \ \ \text{for all} \ \ z \in \mathbb{C} 
\end{equation}
if and only if the quadruple  \(\bigl(\mathfrak{C}_{\widehat{\mu}}, K_{\mathsf{Q}_0,\widehat{\mathcal{P}}}, S_{\widehat{\mu}}, \mathsf{b}_{\widehat{\mu},\widehat{\mathcal{P}}}\bigr)\) is a model for the quadruple \(\bigl(\widehat{\mathfrak{H}}, \kip_{\widehat{\mathfrak{H}}}, \widehat{S}, \widehat{\mathsf{b}}\bigr)\).  Moreover, if \eqref{eq-SPbbA} holds, then 
\[
\TtA(\infty) = \Bigl\{ \{f,g\}\in S_0^{*}\,:\,\widehat{\mathsf{P}}_\infty \mathsf{b}_0(\{f,g\}) = 0 \Bigr\}. 
\]
\end{enumerate}
\end{enumerate}
\end{proposition}

\begin{proof}
(\ref{pr-bbAbasics-i1}) The claims that \(S_0\) and \(\widehat{S}\) are closed symmetric linear relations follow from the definitions. Since $S\subset S_{0}\subset S_{0}^{*}\subset S^{*}$, the defect numbers of $S_{0}$ are finite and since $\widetilde{A}$ is a self-adjoint extension of $S_0$ in $\mathfrak{H}\oplus \widehat{\mathfrak{H}}$ with $\dim \widehat{\mathfrak{H}} \in \mathbb{N}$, they are equal. As in the proposition we denote them by $d_{0}$. The inclusions mentioned earlier imply \(d_{0} \in \{0,\ldots,d\}\). Since \(\widetilde{A}\) is a self-adjoint extension of \(S_0\oplus\widehat{S}\), the Coupling Theorem~\ref{Coupling}(\ref{Coupling-i1}) implies that the defect numbers of \(\widehat{S}\) are also \(d_0\).

\smallskip

(\ref{pr-bbAbasics-i2})
If \(S_0\) is self-adjoint, then, by Theorem~\ref{Coupling}(\ref{Coupling-i2}), \(\widehat{S} = \widetilde{A}\cap \widehat{\mathfrak{H}}^2\) is also self-adjoint. The only self-adjoint operator without eigenvalues is the trivial one \(\{0,0\}\) in the trivial space \(\widehat{\mathfrak{H}} = \{0\}\). Thus (\ref{pr-bbA-i2-i1})$\Rightarrow$(\ref{pr-bbA-i2-i4}).
If \(\widehat{\mathfrak H} = \{0\}\), then \(\widetilde{A} \subset \mathfrak{H}^2\), and hence \(S_0 = \widetilde{A}\). So, (\ref{pr-bbA-i2-i4})$\Rightarrow$(\ref{pr-bbA-i2-i5}).
If \(S_0=\widetilde{A}\), then \(S_0\) is self-adjoint since \(\widetilde{A}\) is. Hence (\ref{pr-bbA-i2-i5})$\Rightarrow$(\ref{pr-bbA-i2-i1}).

Assume that \(\widehat{\mathfrak H} = \{0\}\) and \(\widetilde{A}_1 \in [\widetilde{A}]\). By Definition~\ref{def-bbAeq} there exists an isomorphism $\Psi: \widehat{\mathfrak H} \rightarrow \widehat{\mathfrak H}_1$ such that $\Psi\widehat{S}=\widehat{S}_1\Psi$ and
$(I_{\mathfrak H} \oplus \Psi) \widetilde{A} = \widetilde{A}_1(I_{\mathfrak{H}} \oplus \Psi)$. Since \(\widehat{\mathfrak H} = \{0\}\), we have \(\widehat{\mathfrak H}_1 = \{0\}\) and \(\Psi = 0\). Hence $I_{\mathfrak H} \widetilde{A} = \widetilde{A}_1 I_{\mathfrak{H}}$, implying \(\widetilde{A}_1 = \widetilde{A}\). Thus \([\widetilde{A}] = \{\widetilde{A}\}\). This proves  (\ref{pr-bbA-i2-i1})$\Rightarrow$(\ref{pr-bbA-i2-i6}). The contrapositive of (\ref{pr-bbA-i2-i6})$\Rightarrow$(\ref{pr-bbA-i2-i1}) follows from Corollary~\ref{co-EqV}.

If (\ref{pr-bbA-i2-i4}) holds, then \(P_{\mathfrak{H}} = I_{\mathfrak{H}}\) and (\ref{pr-bbA-i2-i2}) follows from Definition~\ref{def-Sts}, hence  (\ref{pr-bbA-i2-i4})$\Rightarrow$(\ref{pr-bbA-i2-i2}), while (\ref{pr-bbA-i2-i2})$\Rightarrow$(\ref{pr-bbA-i2-i3}) is straightforward. The contrapositive of (\ref{pr-bbA-i2-i3})$\Rightarrow$(\ref{pr-bbA-i2-i1}) follows from part (\ref{pr-bbAbasics-i3})(\ref{pr-bbAbasics-i3-i2}) of this proof, completing the proof of (\ref{pr-bbAbasics-i2}).

\smallskip

(\ref{pr-bbAbasics-i3})(\ref{pr-bbAbasics-i3-i1}) follows from Theorem~\ref{Coupling}(\ref{Coupling-i4}).

\smallskip

(\ref{pr-bbAbasics-i3})(\ref{pr-bbAbasics-i3-i2}) follows from Theorem~\ref{th-SfMt} applied to the closed symmetric linear relations \(S_0\) in the Krein space \(\mathfrak{H}\) and \(\widehat{S}\) in the Krein space \(\widehat{\mathfrak{H}}\).
\end{proof}

\begin{remark} \label{rmk-nota}
Although this is not emphasized in our notation, the class $\mathbb{A}_{S,\mathsf{b}}$ of self-adjoint relations is defined with five fixed objects: the Krein space $(\mathfrak H,\kip_{\mathfrak{H}})$, a closed symmetric linear relation $S$ in $\mathfrak H$ with equal defect numbers \(d \in \mathbb{N}\), and its boundary mapping $\mathsf{b}:S^*\to\mathbb{C}^{2d}$ with Gram matrix $\mathsf{Q}$. For each $\widetilde{A} \in \mathbb{A}_{S,\mathsf{b}}$, Definition~\ref{def-bbA} associates three objects with $\widetilde{A}$: the extended Krein space $(\widetilde{\mathfrak{H}},\kip_{\widetilde{\mathfrak{H}}})$, the finite-dimensional Pontryagin exit space $(\widehat{\mathfrak{H}},\kip_{\widehat{\mathfrak{H}}})$ such that $\widetilde{\mathfrak{H}} = \mathfrak{H} \oplus \widehat{\mathfrak{H}}$, and the trace \(\widehat{S} = \widetilde{A} \cap \widehat{\mathfrak{H}}^2\) of $\widetilde{A}$ in the exit space. Proposition~\ref{pr-bbAbasics} associates five further objects with $\widetilde A$: the trace $S_0 = \widetilde{A}\cap\mathfrak{H}^2$ of $\widetilde{A}$ in $\mathfrak{H}$, its equal defect numbers \(d_0 \in \{0\}\cup\mathbb{N}\), its boundary mapping $\mathsf{b}_0:S_0^{*}\to \mathbb{C}^{2d_0}$ with Gram matrix $\mathsf{Q}_0$, and the boundary mapping $\widehat{\mathsf{b}}:\widehat{S}^{*}\to\mathbb{C}^{2d_0}$ with Gram matrix $-\mathsf{Q}_0$. In summary:

\smallskip
\renewcommand{\arraystretch}{1.4}
\noindent
\begin{tabular}%
{|@{\hspace{3pt}}p{0.19\textwidth}|%
@{\hspace{3pt}}p{0.285\textwidth}|@{\hspace{3pt}}p{0.455\textwidth}|}
\hline
 & \(\mathbb{A}_{S,\mathsf{b}}\) & \(\widetilde{A} \in \mathbb{A}_{S,\mathsf{b}}\) 
 \\ \hline
Definition~\ref{def-bbA} &
\(\bigl(\mathfrak{H},\kip_{\mathfrak{H}}\bigr), \ S, \ d, \ \mathsf{b}, \ \mathsf{Q}\) &
\(\bigl(\widetilde{\mathfrak{H}},\kip_{\widetilde{\mathfrak{H}}}\bigr), \
\widetilde{\mathfrak{H}} = \mathfrak{H}\oplus \widehat{\mathfrak{H}}, \ \widehat{S} = \widetilde{A} \cap \widehat{\mathfrak{H}}^2\) \\ \hline
Proposition~\ref{pr-bbAbasics} &  & \(S_0 = \widetilde{A} \cap \mathfrak{H}^2, \ d_0, \ \mathsf{b}_0, \ \mathsf{Q}_0, \ \widehat{\mathsf{b}}\)  \\ \hline
\end{tabular}

\smallskip\noindent In the remainder of the paper, we use the notation summarized in the table above for the thirteen objects associated with  $\widetilde{A}\in \mathbb{A}_{S,\mathsf{b}}$. Whenever we attach an index to $\widetilde{A} \in \mathbb{A}_{S,\mathsf{b}}$, we attach the same index to each of the corresponding objects as we already did in Definition~\ref{def-bbAeq}.
\end{remark}

The following proposition shows that the inclusion \eqref{incl} is an equality if \(\widetilde{A} \in \mathbb{A}_{S,\mathsf b}\). Note that it is not a priori clear that the right-hand side of \eqref{nonsymm} is symmetric. General equalities of this type where the right-hand side is symmetric can be found in \cite[Theorem~3.4 and Corollary]{DLdS86}, see also \cite{Br} and \cite{Mc}.

\begin{proposition} \label{th-S0Ss}
For every \(\widetilde{A} \in \mathbb{A}_{S,\mathsf b}\) we have
\begin{equation}\label{nonsymm}
\widetilde{A} \cap \mathfrak{H}^2 = \bigcap_{z \in \mathbb W} \TtA(z),
\end{equation}
where $\mathbb{W} \subseteq \mathbb{C}$ with $\# \mathbb{W} \gt \dim \widehat{\mathfrak H}$.
\end{proposition}

\begin{proof}
Set $S_0 = \widetilde A \cap \mathfrak{H}^2$. Then \(S_0\) is a closed symmetric linear  relation in $\bigl(\mathfrak{H},\kip_{\mathfrak{H}}\bigr)$ with defect numbers equal to \(d_0\) where \(0 \leq d_0 \leq d\). If $d_0 = 0$, then \(\widehat{\mathfrak H} = \{0\}\) and by Proposition~\ref{pr-bbAbasics}(\ref{pr-bbAbasics-i2}) we have $\widetilde{A} = \widetilde{A} \cap \mathfrak{H}^2 = \TtA(z)$ for all \(z \in \mathbb{C}\). Thus \eqref{nonsymm} holds trivially.

If \(0 \lt d_0 \leq d\), the defect numbers of \(\widehat{S}\) are also \(d_0\) and \(\widetilde{A}\) is a self-adjoint coupling of \(S_0\) and \(\widehat{S}\).  By Proposition~\ref{pr-bbAbasics}(\ref{pr-bbAbasics-i3})(\ref{pr-bbAbasics-i3-i2}) there exists a \(d_0\times 2 d_0\) matrix polynomial \(\widehat{\mathcal{P}}(z)\) such that
\[
\TtA(z) =
\Bigl\{\{f,g\}\in S_0^{*}\,:\,\widehat{\mathcal{P}}(z) \mathsf{b}_0(\{f,g\})=0\Bigr\}
= \mathsf{b}_0^{-1}\mkern-2mu\Bigl(\nul \bigl(\widehat{\mathcal{P}}(z)\bigr)\Bigr).
\]
Consequently, for arbitrary nonempty subset \(\mathbb{W} \subseteq \mathbb{C}\) we have
\[
\bigcap_{z \in \mathbb W}\TtA(z) =
\mathsf{b}_0^{-1}\mkern-2mu\biggl(
\bigcap_{z \in \mathbb W} \nul \bigl(\widehat{\mathcal{P}}(z)\bigr) \biggr).
\]
By \eqref{={0}},
\[
\bigcap_{z \in \mathbb W} \nul \bigl(\widehat{\mathcal{P}}(z)\bigr) =\{0\}
\]
whenever \(\#\mathbb{W} \gt \deg\widehat{\mathcal{P}}(z)\). By the construction of \(\widehat{\mathcal{P}}(z)\) we have \(\deg\widehat{\mathcal{P}}(z)\leq \dim \widehat{\mathfrak H}\). Therefore for \(\mathbb{W} \subseteq \mathbb{C}\) with \(\#\mathbb{W} \gt \dim \widehat{\mathfrak H}\) we have
\begin{equation*}
\bigcap_{z \in \mathbb W}\TtA(z)
= \mathsf{b}_0^{-1}\mkern-2mu\Bigl(\nul \bigl(\widehat{\mathcal{P}}(z)\bigr)\Bigr)
=  \mathsf{b}_0^{-1}\mkern-2mu\bigl(\{0\}\bigr) = S_0. \qedhere
\end{equation*}
\end{proof}

\subsection{Shtraus extensions and the equivalence relation} \label{subs-mainth}

We now come to the main theorem of this section. It shows that $\widetilde{A}_1 \sim \widetilde{A}_2$ if and only if $T_{\widetilde{A}_1}(z)\equiv T_{\widetilde{A}_2}(z)$. In addition, item (\ref{th-EqSt-i3}) strengthens Lemma~\ref{le-relpr}: after adding the third condition in (\ref{th-EqSt-i3}), the implication from that lemma becomes an equivalence.

\begin{theorem} \label{th-EqSt}
Let $\widetilde A_1, \widetilde A_2 \in \mathbb{A}_{S,\mathsf{b}}$. The following statements are equivalent.
\begin{enumerate}
 \renewcommand*\theenumi{\roman{enumi}}
\renewcommand*\labelenumi{\rm{(\theenumi)}}
\setlength{\itemsep}{1pt}

\item \label{th-EqSt-i1}
\(T_{\mkern-1mu\widetilde{A}_1}\mkern-2.5mu(z) = T_{\mkern-1mu\widetilde{A}_2}\mkern-2.5mu(z)\) \ for all $z \in \mathbb{C}$.

\item \label{th-EqSt-i2}
\(\widetilde A_1 \sim \widetilde A_2\).

\item \label{th-EqSt-i3}
\(\widetilde{A}_1 \cap \mathfrak{H}^2 = \widetilde{A}_2 \cap \mathfrak{H}^2\) and there exists an isomorphism \(\Theta:\widehat{\mathfrak{H}}_1 \to \widehat{\mathfrak{H}}_2\) such that  \(\Theta \widehat{S}_1 = \widehat{S}_2 \Theta\) and for all \(\{\widehat{f},\widehat{g}\} \in \widehat{S}_1^*\) we have
\(\widehat{\mathsf{b}}_2\mkern-1mu ( \{\Theta \widehat{f}, \Theta \widehat{g}\} )
= \widehat{\mathsf{b}}_1\mkern-1mu (\{\widehat{f},\widehat{g}\})\).
\end{enumerate}

If one (and hence all) of the above statements holds, then $T_{\widetilde{A}_1}\mkern-2mu(\infty) = T_{\widetilde{A}_2}\mkern-2mu(\infty)$.
\end{theorem}

\begin{proof}
First assume that \(\widetilde{A}_1\) satisfies any of the six equivalent conditions in Proposition~\ref{pr-bbAbasics}(\ref{pr-bbAbasics-i2}). Then the equivalences (\ref{th-EqSt-i1})$\Leftrightarrow$(\ref{th-EqSt-i2})$\Leftrightarrow$ (\ref{th-EqSt-i3}) follow from Proposition~\ref{pr-bbAbasics}(\ref{pr-bbAbasics-i2}).

Let $\widetilde{A}_1, \widetilde{A}_2 \in \mathbb{A}_{S,\mathsf{b}}$ and assume that neither \(\widetilde{A}_1\cap \mathfrak{H}^2\) nor  \(\widetilde{A}_2\cap \mathfrak{H}^2\) is self-adjoint.

In the proof that follows we will use the fact that the equality in~(\ref{eq-IPsiA1eq}) is equivalent to \eqref{eq-4st-i2}
where $\Psi: \widehat{\mathfrak H}_1 \rightarrow \widehat{\mathfrak H}_2$ is an isomorphism for which $\Psi\widehat S_1 = \widehat S_2\Psi$ holds.

(\ref{th-EqSt-i1})$\Rightarrow$(\ref{th-EqSt-i2}). Assume (\ref{th-EqSt-i1}). By Proposition~\ref{th-S0Ss} we have
\[
\widetilde{A}_1 \cap \mathfrak{H}^2 = \widetilde{A}_2 \cap \mathfrak{H}^2,
\]
and we denote this closed symmetric linear relation by \(S_0\). By Proposition~\ref{pr-bbAbasics}(\ref{pr-bbAbasics-i3})(\ref{pr-bbAbasics-i3-i1}) there exists a boundary mapping $\widehat{\mathsf{b}}_k: \widehat{S}^{*}_k \rightarrow \mathbb{C}^{2d_0}$ for $\widehat{S}_k$ with Gram matrix $-\mathsf{Q}_0$ such that
\begin{equation} \label{eq-tA1tA2}
\widetilde{A}_k =  \Biggl\{
{\scriptstyle \left\{  \begin{bmatrix} f \\[1.5pt] \widehat{f} \end{bmatrix} ,
 \begin{bmatrix}
g \\[1.5pt]  \widehat{g} \end{bmatrix} \right\}}
 \, : \,
 \begin{array}{l}
 \{f,g\} \in  S_0^{*}, \ \ \{ \widehat{f}, \widehat{g}\} \in  \widehat{S}_k^*,  \\[2.5pt]
 \mathsf{b}_0\mkern-1mu (\{f, g\}) + \widehat{\mathsf{b}}_k\mkern-1mu (\{ \widehat{f}, \widehat{g}\}) = 0
 \end{array}
  \mkern-3mu \Biggr\}, \quad k \in \{1,2\}.
\end{equation}

By the Model Theorem~\ref{t-poso} there exists a \(d_0\times 2d_0\) matrix polynomial  \(\widehat{\mathcal{P}}(z)\) such that the quadruple  \(\bigl(\mathfrak{C}_{\widehat{\mu}}, K_{\mathsf{Q}_0,\widehat{\mathcal{P}}}, S_{\widehat{\mu}}, \mathsf{b}_{\widehat{\mu},\widehat{\mathcal{P}}}\bigr)\) is a model for \(\bigl(\widehat{\mathfrak{H}}_1, \kip_{\widehat{\mathfrak{H}}_1}, \widehat{S}_1, \widehat{\mathsf{b}}_1\bigr)\). Let \(\Phi_1:\widehat{\mathfrak{H}}_1 \to \mathfrak{C}_{\widehat{\mu}}\) be the corresponding (modeling) isomorphism. The ``if'' part of Proposition~\ref{pr-bbAbasics}(\ref{pr-bbAbasics-i3})(\ref{pr-bbAbasics-i3-i2}) yields
\begin{equation}\label{eq-tA1ThP}
T_{\widetilde{A}_1}\mkern-3mu(z) =
\Bigl\{\{f,g\}\in S_0^{*}\,:\,\widehat{\mathcal{P}}(z)\mathsf b_0(\{f,g\}) = 0\Bigr\}.
\end{equation}
The ``only if'' part of Proposition~\ref{pr-bbAbasics}(\ref{pr-bbAbasics-i3})(\ref{pr-bbAbasics-i3-i2}), the assumption \(T_{\widetilde{A}_1}\mkern-2.5mu(z) = T_{\widetilde{A}_2}\mkern-2.5mu(z)\) \ for all $z \in \mathbb{C}$, and \eqref{eq-tA1ThP} imply that \(\bigl(\mathfrak{C}_{\widehat{\mu}}, K_{\mathsf{Q}_0,\widehat{\mathcal{P}}}, S_{\widehat{\mu}}, \mathsf{b}_{\widehat{\mu},\widehat{\mathcal{P}}}\bigr)\) is a model for \(\bigl(\widehat{\mathfrak{H}}_2, \kip_{\widehat{\mathfrak{H}}_2}, \widehat{S}_2, \widehat{\mathsf{b}}_2\bigr)\) as well. Let \(\Phi_2:\widehat{\mathfrak{H}}_2 \to \mathfrak{C}_{\widehat{\mu}}\) be the corresponding isomorphism. Set $\Psi = (\Phi_2)^{-1} \Phi_1 : \widehat{\mathfrak{H}}_1 \rightarrow \widehat{\mathfrak{H}}_2$. Then \(\Psi\) is an isomorphism such that \(\{\Psi \widehat{f}, \Psi \widehat{g}\} \in \widehat{S}_2^{*}\) if and only if \(\{\widehat{f},\widehat{g}\}\in \widehat{S}_1^*\), and
\(\widehat{\mathsf{b}}_2\bigl(\{\Psi \widehat{f}, \Psi \widehat{g}\}\bigr) = \widehat{\mathsf{b}}_1\bigl(\{\widehat{f},\widehat{g}\}\bigr)\) for all
\(\{\widehat{f},\widehat{g}\}\in \widehat{S}_1^*\). Therefore, using \eqref{eq-tA1tA2}, first with \(k=2\) and then with \(k=1\) at the end we get
\begin{align*}
  \widetilde{A}_2 & = \Biggl\{
{\scriptstyle \left\{  \begin{bmatrix} f \\[1.5pt] \Psi \widehat{f} \end{bmatrix} ,
 \begin{bmatrix}
g \\[1.5pt] \Psi \widehat{g} \end{bmatrix} \right\}}
 \, : \,
 \begin{array}{l}
 \{f,g\} \in  S_0^{*}, \ \ \{ \widehat{f},  \widehat{g}\} \in  \widehat{S}_1^{*},  \\[2.5pt]
 \mathsf{b}_0\mkern-1mu (\{f, g\}) + \widehat{\mathsf{b}}_2\mkern-1mu (\{ \Psi\widehat{f}, \Psi\widehat{g}\}) = 0
 \end{array}
  \mkern-3mu \Biggr\} \\
   & = \Biggl\{
{\scriptstyle \left\{  \begin{bmatrix} f \\[1.5pt] \Psi \widehat{f} \end{bmatrix} ,
 \begin{bmatrix}
g \\[1.5pt] \Psi \widehat{g} \end{bmatrix} \right\}}
 \, : \,
 \begin{array}{l}
 \{f,g\} \in  S_0^{*}, \ \ \{ \widehat{f},  \widehat{g}\} \in  \widehat{S}_1^{*}, \\[2.5pt]
 \mathsf{b}_0\mkern-1mu (\{f, g\}) + \widehat{\mathsf{b}}_1\mkern-1mu (\{\widehat{f},\widehat{g}\}) = 0
 \end{array}
  \mkern-3mu \Biggr\} \\
 & = \Biggl\{
{\scriptstyle
\left\{  \begin{bmatrix} f \\[1.5pt]  \Psi \widehat{f} \end{bmatrix} ,
 \begin{bmatrix} g \\[1.5pt] \Psi \widehat{g} \end{bmatrix} \right\}}
 \, : \,
{\scriptstyle
\left\{  \begin{bmatrix} f \\[1.5pt]  \widehat{f} \end{bmatrix} ,
 \begin{bmatrix} g \\[1.5pt]  \widehat{g} \end{bmatrix} \right\}} \in \widetilde{A}_1
  \Biggr\}.
\end{align*}
This proves \eqref{eq-4st-i2} which is equivalent to \eqref{eq-IPsiA1eq}. Hence (\ref{th-EqSt-i2}) is proved.

\smallskip

(\ref{th-EqSt-i2})$\Rightarrow$(\ref{th-EqSt-i1}).
Assume (\ref{th-EqSt-i2}). Let $z \in \mathbb{C}$ be arbitrary. By \eqref{eq-4st-i2}, $\{f,g\} \in T_{\widetilde{A}_2}\mkern-2.5mu(z)$ is equivalent to the statement: there exists
${\scriptstyle \left\{\begin{bmatrix} f \\[1.5pt]  \widehat{f}_1 \end{bmatrix},
\begin{bmatrix} g \\[1.5pt]  \widehat{g}_1 \end{bmatrix} \right\}} \in \widetilde{A}_1$ with $z \Psi \widehat{f}_1- \Psi \widehat{g}_1=0$. Since \(\Psi\) is an isomorphism, the last statement is equivalent to: there exists
${\scriptstyle \left\{\begin{bmatrix} f \\[1.5pt]  \widehat{f}_1 \end{bmatrix},
\begin{bmatrix} g \\[1.5pt]  \widehat{g}_1 \end{bmatrix} \right\}} \in \widetilde{A}_1$  with $z \widehat{f}_1-\widehat{g}_1=0$, which in turn is equivalent to $\{f,g\} \in T_{\widetilde{A}_1}\mkern-2.5mu(z)$. Since \(z \in \mathbb{C}\) was arbitrary, (\ref{th-EqSt-i1}) is proved.

\smallskip

(\ref{th-EqSt-i2})$\Rightarrow$(\ref{th-EqSt-i3}). Assume (\ref{th-EqSt-i2}). By Lemma~\ref{le-relpr}, \(\widetilde{A}_1 \cap \mathfrak{H}^2 = \widetilde{A}_2 \cap \mathfrak{H}^2\) and  \eqref{eq-tA1tA2} holds. Let \(\Psi:\widehat{\mathfrak{H}}_1 \to \widehat{\mathfrak{H}}_2\) be the isomorphism for which \eqref{eq-IPsiA1eq} and  \eqref{eq-4st-i2} hold. By \eqref{eq-4st-i2}, ${\scriptstyle
\left\{  \begin{bmatrix} f \\[1.5pt]  \Psi \widehat{f} \end{bmatrix} ,
 \begin{bmatrix} g \\[1.5pt] \Psi \widehat{g} \end{bmatrix} \right\}} \in \widetilde{A}_2$
if and only if \(\{f,g\} \in S_0^{*}, \{\widehat{f},\widehat{g}\} \in \widehat{S}_1^{*}\) and
\begin{equation}\label{eq-aux1}
\mathsf{b}_0(\{f,g\}) + \widehat{\mathsf{b}}_1(\{\widehat{f},\widehat{g}\}) = 0.
\end{equation}
By \eqref{eq-tA1tA2} with \(k=2\) we have ${\scriptstyle
\left\{  \begin{bmatrix} f \\[1.5pt]  \Psi \widehat{f} \end{bmatrix} ,
 \begin{bmatrix} g \\[1.5pt] \Psi \widehat{g} \end{bmatrix} \right\}} \in \widetilde{A}_2$
if and only if \(\{f,g\} \in S_0^{*}\), \(\{\Psi\widehat{f},\Psi\widehat{g}\} \in \widehat{S}_2^{*}\) and
\begin{equation}\label{eq-aux2}
\mathsf{b}_0(\{f,g\}) + \widehat{\mathsf{b}}_2(\{\Psi\widehat{f},\Psi\widehat{g}\}) = 0.
\end{equation}
Hence \(\{\widehat{f},\widehat{g}\} \in \widehat{S}_1^{*}\) if and only if \(\{\Psi\widehat{f},\Psi\widehat{g}\} \in \widehat{S}_2^{*}\) and both \eqref{eq-aux1} and \eqref{eq-aux2} hold. Therefore \(\Psi\) has all the properties required of \(\Theta\) in (\ref{th-EqSt-i3}). Thus (\ref{th-EqSt-i3}) is proved.

\smallskip

(\ref{th-EqSt-i3})$\Rightarrow$(\ref{th-EqSt-i2}). Assume (\ref{th-EqSt-i3}) and set  \(S_0 = \widetilde{A}_1 \cap \mathfrak{H}^2 = \widetilde{A}_2 \cap \mathfrak{H}^2\). By Proposition~\ref{pr-bbAbasics}(\ref{pr-bbAbasics-i3})(\ref{pr-bbAbasics-i3-i1}), \eqref{eq-tA1tA2} holds. By \eqref{eq-tA1tA2} with \(k=1\) we have ${\scriptstyle
\left\{  \begin{bmatrix} f \\[1.5pt]  \widehat{f} \end{bmatrix} ,
 \begin{bmatrix} g \\[1.5pt] \widehat{g} \end{bmatrix} \right\}} \in \widetilde{A}_1$
if and only if \(\{f,g\} \in S_0^{*}, \{\widehat{f},\widehat{g}\} \in \widehat{S}_1^{*}\) and
\begin{equation}\label{eq-aux3}
\mathsf{b}_0(\{f,g\}) + \widehat{\mathsf{b}}_1(\{\widehat{f},\widehat{g}\}) = 0.
\end{equation}
Since \(\Theta:\widehat{\mathfrak{H}}_1 \to \widehat{\mathfrak{H}}_2\) is an isomorphism, \(\{\widehat{f},\widehat{g}\} \in \widehat{S}_1^{*}\) if and only if \(\{\Theta \widehat{f},\Theta \widehat{g}\} \in \widehat{S}_2^{*}\), and by the boundary mapping property of \(\Theta\), \eqref{eq-aux3} is equivalent to
\begin{equation*}  
\mathsf{b}_0(\{f,g\}) + \widehat{\mathsf{b}}_2\bigl(\{\Theta\widehat{f},\Theta\widehat{g}\}\bigr) = 0,
\end{equation*}
which is by \eqref{eq-tA1tA2} with \(k=2\)  equivalent to ${\scriptstyle
\left\{  \begin{bmatrix} f \\[1.5pt]  \Theta \widehat{f} \end{bmatrix} ,
 \begin{bmatrix} g \\[1.5pt] \Theta \widehat{g} \end{bmatrix} \right\}} \in \widetilde{A}_2$. Thus we have proved that ${\scriptstyle
\left\{  \begin{bmatrix} f \\[1.5pt]  \widehat{f} \end{bmatrix} ,
\begin{bmatrix} g \\[1.5pt] \widehat{g} \end{bmatrix} \right\}} \in \widetilde{A}_1$ if and only if ${\scriptstyle
\left\{  \begin{bmatrix} f \\[1.5pt]  \Theta \widehat{f} \end{bmatrix} ,
 \begin{bmatrix} g \\[1.5pt] \Theta \widehat{g} \end{bmatrix} \right\}} \in \widetilde{A}_2$. That is, \eqref{eq-4st-i2} holds with \(\Psi = \Theta\), which implies (\ref{th-EqSt-i2}).

Assume (\ref{th-EqSt-i1}). Then the equality $T_{\widetilde{A}_1}\mkern-2mu(\infty) = T_{\widetilde{A}_2}\mkern-2mu(\infty)$ follows from the last statement in Proposition~\ref{pr-bbAbasics}(\ref{pr-bbAbasics-i3})(\ref{pr-bbAbasics-i3-i2}).
\end{proof}

In the next corollary we complete the proof of Proposition~\ref{pr-bbAbasics}(\ref{pr-bbAbasics-i2}) by proving the contrapositive of (\ref{pr-bbA-i2-i6})$\Rightarrow$(\ref{pr-bbA-i2-i1}). That is, we prove that each equivalence class \([\widetilde{A}]\) in \(\mathbb{A}_{S,\mathsf b}\) is an infinite set provided that \(S_0\) is not self-adjoint.

\begin{corollary} \label{co-EqV}
Let \(\widetilde{A} \in \mathbb A_{S,\mathsf b}\), assume \(d_0 \gt 0\) (see Remark~{\rm\ref{rmk-nota}}) and let \(\widetilde{A}\) be represented by \eqref{eq-tAhb},
as in Proposition~{\rm\ref{pr-bbAbasics}(\ref{pr-bbAbasics-i3})(\ref{pr-bbAbasics-i3-i1})}. For every \(\theta \in (0,2\pi)\) set
\begin{equation*}  
\widetilde{A}_\theta =  \Biggl\{
{\scriptstyle \left\{  \begin{bmatrix} f \\[1.5pt] \widehat{f} \end{bmatrix} ,
 \begin{bmatrix}
g \\[1.5pt]  \widehat{g} \end{bmatrix} \right\}}
 \, :
 \begin{array}{l}
 \{f,g\} \in  S_0^{*}, \ \ \{ \widehat{f},  \widehat{g}\} \in  \widehat{S}^{*}, \\[2.5pt]
 \mathsf{b}_0(\{f, g\})
 + e^{\iu \theta} \mkern1.5mu \widehat{\mathsf{b}} (\{ \widehat{f}, \widehat{g}\}) = 0
 \end{array}
  \Biggr\}.
\end{equation*}
Then \(\widetilde{A}_\theta\in \mathbb{A}_{S,\mathsf{b}}\), \(S_0 = \widetilde{A}_\theta \cap \mathfrak{H}^2\), \(\widehat{S} = \widetilde{A}_\theta \cap \widehat{\mathfrak{H}}^2\), \(\widetilde{A}_\theta\neq \widetilde{A}\) and \(\widetilde{A}_\theta\sim \widetilde{A}\).
\end{corollary}
\begin{proof}
Since \(e^{\iu \theta} \mkern1.5mu \widehat{\mathsf{b}}: \widehat{S}^{*} \to \mathbb{C}^{2d_0}\) is a boundary mapping for \(\widehat{S}\) with Gram matrix \(-\mathsf{Q}_0\), the claims \(\widetilde{A}_\theta\in \mathbb{A}_{S,\mathsf{b}}\), \(S_0 = \widetilde{A}_\theta \cap \mathfrak{H}^2\), \(\widehat{S} = \widetilde{A}_\theta \cap \widehat{\mathfrak{H}}^2\), and \(\widetilde{A}_\theta\neq \widetilde{A}\) follow from the last statement in Proposition~\ref{pr-bbAbasics}(\ref{pr-bbAbasics-i3})(\ref{pr-bbAbasics-i3-i1}).

The claim \(\widetilde{A}_\theta\sim \widetilde{A}\) follows from Theorem~\ref{th-EqSt}(\ref{th-EqSt-i3})$\Rightarrow$(\ref{th-EqSt-i2}) since \(\Theta:\widehat{\mathfrak{H}}\to\widehat{\mathfrak{H}}\) defined by  \(\Theta\widehat{f} = e^{-\iu \theta}\widehat{f}\) for \(\widehat{f}\in \widehat{\mathfrak{H}}\) is an isomorphism on \(\widehat{\mathfrak{H}}\) such that for all \(\{\widehat{f}, \widehat{g}\} \in \widehat{S}^{*}\) we have
\[
e^{\iu \theta} \mkern1.5mu \widehat{\mathsf{b}} \bigl(\{\Theta\widehat{f}, \Theta\widehat{g}\}\bigr) = e^{\iu \theta} \mkern1.5mu \widehat{\mathsf{b}} \bigl(\{e^{-\iu \theta}\widehat{f}, e^{-\iu \theta}\widehat{g}\}\bigr) = \widehat{\mathsf{b}} \bigl(\{\widehat{f}, \widehat{g}\}\bigr),
\]
where the last equality holds since \(\widehat{\mathsf{b}}: \widehat{S}^{*}\to \mathbb{C}^{2d_0}\) is linear.
\end{proof}

\section{Linearization} \label{linearization}

\begin{definition} \label{def-lin}
Let \(\mathcal{P}(z) \in \mathbb{P}_{\mathsf{Q}}\) and $\widetilde{A} \in \mathbb{A}_{S,\mathsf{b}}$. We say that \(\widetilde{A}\) is a \emph{linearization of the boundary eigenvalue problem} $\BEP\mkern-2mu\bigl(S, \mathsf b, \mathcal{P}(z)\bigr)$ if the following identity holds:
\begin{equation}\label{lin1}
\TtA(z) =
\Big\{ \{f,g\}\in S^* \, :\, \mathcal{P}(z)\mathsf b(\{f,g\})=0\Big\}, \quad z \in \overline{\mathbb{C}}.
\end{equation}
If \eqref{lin1} holds, we also say that \(\widetilde{A}\) is a \emph{linearization of the family} $T_{\mathcal{P}}(z)$, $z \in \overline{\mathbb C}$, from \eqref{preShtraus}.
\end{definition}

Theorem~\ref{th-EqSt} implies that if $\widetilde A$ is a linearization of the
$\BEP\mkern-2mu\bigl(S,\mathsf b,\mathcal P(z)\bigr)$, then every member of the
equivalence class $[\widetilde A]$ in $\mathbb A_{S,\mathsf b}$ is also a
linearization. Moreover, by Lemma~\ref{le-PeSSts}, if each member of
\([\widetilde{A}]\) linearizes the $\BEP\mkern-2mu\bigl(S,\mathsf b,\mathcal P(z)\bigr)$, then it also linearizes the $\BEP\mkern-2mu\bigl(S,\mathsf b,\mathcal S(z)\bigr)$ for every $\mathcal S(z) \in [\mathcal{P}(z)]$ in ${\mathbb P}_{\mathsf Q}$. Thus linearization is a relation between the quotient sets $\mathbb A_{S,\mathsf b}/\mathord{\sim}$ and ${\mathbb P}_{\mathsf Q}/\mathord{\sim}$.
The next theorem shows that this relation is remarkably simple.

Given a boundary eigenvalue problem $\BEP\mkern-2mu\bigl(S,\mathsf{b},\mathcal{P}(z)\bigr)$,  the most important part of the next theorem is the construction of the linearization in  \eqref{eq-tA3c}; see Steps 3b and 3c for details.

\begin{theorem} \label{X}
The binary relation
\begin{equation*}  
\Xi = \biggl\{
\bigl\{  \bigl[\widetilde{A}\bigr], \bigl[\mathcal{P}(z)\bigr] \bigr\} : \mkern-5mu
 \begin{array}{l}
 \widetilde A \in \mathbb{A}_{S,\mathsf{b}},\ \ \mathcal P(z) \in \mathbb{P}_{\mathsf{Q}},
 \\[2.5pt]
T_{\widetilde{A}}(z)=T_{S,\mathsf{b}, \mathcal P}(z) \ \text{for all}\ z \in \mathbb C
 \end{array}\mkern-2mu\biggr\}
\end{equation*}
is a bijection between the equivalence classes in \(\mathbb{A}_{S,\mathsf{b}}/\mathord{\sim}\) and $\mathbb{P}_{\mathsf{Q}}/\mathord{\sim}$.
\end{theorem}

\begin{proof}
\emph{Step~1.}
The relation \(\Xi\) is single-valued. Let 
\[
\bigl\{\bigl[\widetilde{A}\bigr], \bigl[\mathcal{P}(z)\bigr] \bigr\}, \bigl\{\bigl[\widetilde{B}\bigr], \bigl[\mathcal{S}(z)\bigr] \bigr\} \in \Xi 
\]
and assume that \(\bigl[\widetilde{A}\bigr] = \bigl[\widetilde{B}\bigr]\). Then \(\widetilde{A}\sim\widetilde{B}\), and by the definition of \(\Xi\),  \(T_{\widetilde{A}}(z) \equiv T_{\mathcal{P}}(z)\),  and \(T_{\widetilde{B}}(z) \equiv T_{\mathcal{S}}(z)\).
Theorem~\ref{th-EqSt}(\ref{th-EqSt-i2})$\Rightarrow$(\ref{th-EqSt-i1}) yields \(T_{\widetilde{A}}(z) \equiv T_{\widetilde{B}}(z)\). Hence \(T_{\mathcal{P}}(z) \equiv T_{\mathcal{S}}(z)\). By  Lemma~\ref{le-PeSSts}, \(\mathcal{P}(z)\sim \mathcal{S}(z)\). Therefore \(\bigl[\mathcal{P}(z)\bigr] = \bigl[\mathcal{S}(z)\bigr]\).

\smallskip

\noindent\emph{Step~2.}
In this step we prove that the domain of \(\Xi\) is \(\mathbb{A}_{S,\mathsf{b}}/\mathord{\sim}\), that is, \(\Xi\) is total. Let $\widetilde{A} \in \mathbb{A}_{S,\mathsf{b}}$. We show that it is a linearization of a $\BEP\mkern-2mu\bigl(S, \mathsf b, \mathcal{P}(z)\bigr)$ for some $\mathcal{P}(z) \in \mathbb{P}_{\mathsf{Q}}$. As before, set $S_0 = \widetilde{A} \cap \mathfrak{H}^2$. It is a closed symmetric extension of $S$ in $\mathfrak{H}$  with equal defect numbers $d_0$ with $0 \leq d_0 \leq d$.

\smallskip
\noindent\emph{Step~2a.} If $d_0 = 0$, then by Proposition~\ref{pr-bbAbasics}(\ref{pr-bbAbasics-i2}) $\widetilde{A} = S_0$ is a self-adjoint extension of \(S\) and by Theorem~\ref{list}(\ref{list-iB})(\ref{list-iB-i3}) has the form $\widetilde{A} = \bigl\{\{f,g\}\,:\, \mathsf{P} \mathsf{b}(\{f,g\}) = 0\bigr\}$ where $\mathsf{P}$ is a constant $d \times 2d$ matrix satisfying $\rank \mathsf P=d$ and $\mathsf P \mathsf Q^{-1} \mathsf{P}^* = \mathsf{0}$. Hence $\{ [\widetilde{A}],[\mathsf P] \} \in \Xi$.

\smallskip

\noindent\emph{Step~2b.} If $d_0=d$, then the argument of Step~2c applies with the block $\mathsf B$ absent, $\mathsf B_0=\mathsf I_{2d}$, $\mathsf b_0=\mathsf b$,  $\mathsf Q_0 = \mathsf Q$, and $\widehat{\mathcal P}(z) = \mathcal P(z)\in\mathbb P_{\mathsf Q}$. Hence \eqref{eq-SPL} below becomes
\[
\TtA(z)=\Bigl\{\{f,g\}\in S^*:\, \mathcal{P}(z)\mathsf b(\{f,g\})=0 \Bigr\}.
\]
Therefore \(\{[\widetilde A],[\mathcal P(z)]\}\in\Xi\).

\smallskip
\noindent\emph{Step~2c.}
Assume $0< d_0 < d$. We apply Theorem~\ref{list} with
\begin{equation*}  
 T=S_0,\ d^\pm=d, \ \ \delta=2d, \ \  \tau = d-d_0, \ \  e^{\pm} = d_0, \ \   \text{and} \ \ \rho = 2d_0.
\end{equation*}
By Theorem~\ref{list}(\ref{list-iA})(\ref{list-iic}) there exists a \((d-d_0)\times2d\) matrix \(\mathsf{B}\) of full rank such that
\begin{equation}\label{eq-S0sB}
S_0^{*} = \Bigl\{ \{f,g\} \in S^* : \mathsf{B} \mathsf{b}(\{f,g\}) = 0 \Bigr\}.
\end{equation}
Let \(\widehat{\mathcal{P}}(z) \in \mathbb{P}_{\mathsf{Q}_0}\) be the \(d_0 \times 2d_0\) matrix polynomial from Proposition~\ref{pr-bbAbasics}(\ref{pr-bbAbasics-i3})(\ref{pr-bbAbasics-i3-i2}) such that
\begin{equation}\label{eq-SPL}
\TtA(z) =
\Bigl\{\{f,g\}\in S_0^{*}\,:\,\widehat{\mathcal{P}}(z)\mathsf{b}_0(\{f,g\}) = 0\Bigr\}.
\end{equation}

Recall that by Theorem~\ref{list}(\ref{list-iC})(\ref{list-iC-i2}) there exists a \(2d_0\times 2d\) matrix \(\mathsf{B}_0\) such that the \((d+d_0)\times 2d\) block matrix $\begin{bmatrix} \mathsf{B}_0 \\ \mathsf{B}\end{bmatrix}$ is of full rank,
\begin{equation} \label{eq-S0B0B}
S_0 = \biggl\{ \{f,g\} \in S^* : \begin{bmatrix} \mathsf{B}_0 \\ \mathsf{B}\end{bmatrix} \mathsf{b}(\{f,g\}) = 0 \biggr\},
\end{equation}
$\mathsf{Q}_0 = \bigl(\mathsf{B}_0\mathsf{Q}^{-1}\mathsf{B}_0^*\bigr)^{-1}$ and \(\mathsf{b}_0 = \mathsf{B}_0 \mathsf{b}\bigl.\bigr|_{S_0^*}\). Using the last equality we can write \eqref{eq-S0sB} and \eqref{eq-SPL} together as follows
\begin{equation*} 
\TtA(z) =
\biggl\{\{f,g\}\in S^*\,:\,\begin{bmatrix} \widehat{\mathcal{P}}(z) \mathsf{B}_0 \\ \mathsf{B}\end{bmatrix} \mathsf{b}(\{f,g\}) = 0\biggr\}.
\end{equation*}

Define the $d \times 2d$ matrix polynomial $\mathcal{P}(z)$ by
\begin{equation}\label{reprP1}
\mathcal{P}(z) =
 \begin{bmatrix}  \widehat{\mathcal{P}}(z) & \mathsf{0} \\[5pt] \mathsf{0} & \mathsf{I}_{d-d_0} \end{bmatrix}
\begin{bmatrix}  \mathsf{B}_0 \\[5pt] \mathsf{B}  \end{bmatrix},
\end{equation}
as in Theorem~\ref{modelP}. Then
\begin{equation} \label{eq-SPL1}
\TtA(z) =
\Bigl\{\{f,g\}\in S^*\,:\, \mathcal{P}(z) \mathsf{b}(\{f,g\}) = 0\Bigr\}.
\end{equation}
From Theorem~\ref{list} we know $\mathsf{B} \mathsf{Q}^{-1}\mathsf{B}^*=0$ (see (\ref{list-iB})(\ref{list-iB-i4})) and $\mathsf{B} \mathsf{Q}^{-1}\mathsf{B}_0^*=0$
(see (\ref{list-iC})(\ref{list-iC-i2})(\ref{list-iC-i2-i2})). Therefore
\begin{align*}
\mathcal{P}(z) \mathsf{Q}^{-1} \mathcal{P}(z^*)^*
  & =
\begin{bmatrix}  \widehat{\mathcal{P}}(z) & \mathsf{0} \\[5pt] \mathsf{0} & \mathsf{I}_{d-d_0} \end{bmatrix}
\begin{bmatrix}  \mathsf{B}_0 \\[5pt] \mathsf{B}  \end{bmatrix} \mathsf{Q}^{-1}
\begin{bmatrix}  \mathsf{B}_0^* & \mathsf{B}^*  \end{bmatrix}
\begin{bmatrix}  \widehat{\mathcal{P}}(z^*)^* & \mathsf{0} \\[5pt] \mathsf{0} & \mathsf{I}_{d-d_0} \end{bmatrix} \\
& =
\begin{bmatrix}  \widehat{\mathcal{P}}(z) & \mathsf{0} \\[5pt] \mathsf{0} & \mathsf{I}_{d-d_0} \end{bmatrix}
\begin{bmatrix}  \mathsf{B}_0\mathsf{Q}^{-1}\mathsf{B}_0^* &  \mathsf{B}_0\mathsf{Q}^{-1}\mathsf{B}^*  \\[5pt]
\mathsf{B}\mathsf{Q}^{-1}\mathsf{B}_0^* & \mathsf{B}\mathsf{Q}^{-1}\mathsf{B}^*   \end{bmatrix}
\begin{bmatrix}  \widehat{\mathcal{P}}(z^*)^* & \mathsf{0} \\[5pt] \mathsf{0} & \mathsf{I}_{d-d_0} \end{bmatrix} \\
& =
\begin{bmatrix}  \widehat{\mathcal{P}}(z) & \mathsf{0} \\[5pt] \mathsf{0} & \mathsf{I}_{d-d_0} \end{bmatrix}
\begin{bmatrix}  \mathsf{Q}_0^{-1} &  \mathsf{0} \\[5pt] \mathsf{0} & \mathsf{0} \end{bmatrix}
\begin{bmatrix}  \widehat{\mathcal{P}}(z^*)^* & \mathsf{0} \\[5pt] \mathsf{0} & \mathsf{I}_{d-d_0} \end{bmatrix} \\
& =
\begin{bmatrix}
\widehat{\mathcal{P}}(z) \mathsf{Q}_0^{-1} \widehat{\mathcal{P}}(z^*)^* & \mathsf{0} \\[5pt]
\mathsf{0} & \mathsf{0}
\end{bmatrix}  \\
& = \mathsf{0}.
\end{align*}
The last equality holds because \(\widehat{\mathcal{P}}(z)\in \mathbb{P}_{\mathsf{Q}_0}\).

Since the range of the matrix $\begin{bmatrix} \mathsf{B}_0 \\ \mathsf{B}\end{bmatrix}$ is \(\mathbb{C}^{d+d_0}\) and  \(\widehat{\mathcal{P}}(z)\in \mathbb{P}_{\mathsf{Q}_0}\), for every \(z \in \mathbb{C}\) the \(d\times(d+d_0)\) matrix \(\begin{bmatrix}  \widehat{\mathcal{P}}(z) & \mathsf{0} \\[5pt] \mathsf{0} & \mathsf{I}_{d-d_0} \end{bmatrix}\) has full range, and therefore the matrix in \eqref{reprP1} has full range.

Let \(\mu_j = \deg \widehat{\mathcal{P}}(z)\bigl.\bigr|_j\) for \(j \in \{1,\ldots,d_0\}\). Since \(\widehat{\mathcal{P}}(z)\in \mathbb{P}_{\mathsf{Q}_0}\), its degrees satisfy \(\mu_1 \geq \cdots \geq \mu_{d_0}\) and the \(d_0\times 2d_0\) matrix
\[
\widehat{\mathsf{P}}_{\mkern-2mu\infty} = \lim_{z\to \infty} \diag\bigl(z^{-\mu_1}, \ldots,z^{-\mu_{d_0}}\bigr)\widehat{\mathcal{P}}(z)
\]
has full rank. By definition \eqref{reprP1} of \(\mathcal{P}(z)\) we have \(\deg \mathcal{P}(z)\bigl.\bigr|_j \leq \mu_j\) for \(j \in \{1,\ldots,d_0\}\) and \(\deg \mathcal{P}(z)\bigl.\bigr|_j = 0\) for \(j \in \{d_0+1,\ldots,d\}\). Therefore
\[
\lim_{z\to \infty} \diag\bigl(z^{-\mu_1}, \ldots,z^{-\mu_{d_0}}, 1,\ldots,1\bigr)\mathcal{P}(z) =
\begin{bmatrix}  \widehat{\mathsf{P}}_{\mkern-2mu\infty} & \mathsf{0} \\[5pt] \mathsf{0} & \mathsf{I}_{d-d_0} \end{bmatrix} \begin{bmatrix} \mathsf{B}_0 \\[5pt] \mathsf{B}\end{bmatrix}
\]
and a reasoning as in the preceding paragraph yields that the limit matrix on the right has full rank. Therefore \(\deg \mathcal{P}(z)\bigl.\bigr|_j = \mu_j\) for \(j \in \{1,\ldots,d_0\}\) and \(\mathsf{P}_{\mkern-2mu\infty} = \begin{bmatrix} \widehat{\mathsf{P}}_{\mkern-2mu\infty} \mathsf{B}_0 \\ \mathsf{B}\end{bmatrix}\). Hence $\mathcal{P}(z) \in \mathbb{P}_{\mathsf{Q}}$.
Now \eqref{eq-SPL1} implies $\bigl\{[\widetilde A],[\mathcal P(z)] \bigr\} \in \Xi$ and hence $\dom \Xi =\mathbb{A}_{S,\mathsf{b}}/\mathord{\sim}$.

\medskip
\noindent\emph{Step~3.}
In this step we show that $\Xi$ is surjective. Consider $\mathcal{P}(z) \in \mathbb{P}_{\mathsf{Q}}$ and set $\mu_j=\deg \mathcal{P}(z)\bigl.\bigr|_j$, $j \in \{1, \ldots,d\}$.  Let $d_0 \in \{0, \ldots, d\}$ be such that
\[
d-d_0 = \#\big\{ j\in \{1,\ldots, d\}\,:\, \mu_j=0 \big\}.
\]

\smallskip
\noindent\emph{Step~3a.}
If $d_0=0$, then $\mathcal P(z)=\mathsf P_\infty$ and since $\rank \mathsf P_\infty =d$ and $\mathsf P_\infty \mathsf Q^{-1} \mathsf{P}_\infty^* = 0$,
\[
\widetilde{A} =\Big\{ \{f,g\}\in S^*\,:\, \mathsf{P}_\infty \mathsf{b}(\{f,g\}) = 0 \Big\}
\]
is a canonical self-adjoint extension of $S$; see Theorem~\ref{list}(\ref{list-iB})(\ref{list-iB-i3}). It trivially belongs to $\mathbb{A}_{S,\mathsf{b}}$ with $\widehat{\mathfrak H}=\{0\}$ and $\widehat S=\big\{\{0,0\}\big\}$, and $\bigl\{ [\widetilde A], [\mathsf{P}_\infty] \bigr\} \in \Xi$.

\smallskip
\noindent\emph{Step~3b.}
If $d_0 = d$, then all $\mu_j\geq 1$. We apply Theorem~\ref{Coupling}(\ref{Coupling-i4}) to the Krein space \(\mathfrak{H}\), the closed symmetric linear relation \(S\), its boundary mapping \(\mathsf{b}:S^* \to \mathbb{C}^{2d}\) with Gram matrix \(\mathsf{Q}\) and the Pontryagin space \(\bigl(\mathfrak{C}_\mu, K_{\mathsf{Q},\mathcal{P}}\bigr)\), a symmetric operator \(S_\mu\) without eigenvalues and its boundary mapping \(\mathsf{b}_{\mu,\mathcal{P}}\) with Gram matrix \(-\mathsf{Q}\) introduced in Theorem~\ref{t-fCrks}, and define in
\(\mathfrak{H}\oplus\mathfrak{C}_\mu\)
\begin{equation*}  
\widetilde{A} =  \Biggl\{
{\scriptstyle \left\{  \begin{bmatrix} f \\[1.5pt] f_\mu \end{bmatrix} ,
 \begin{bmatrix}
g \\[1.5pt]  g_\mu \end{bmatrix} \right\}}
 \, : \,
 \begin{array}{l}
 \{f,g\} \in  S^*, \ \ \{ f_\mu,  g_\mu\} \in  S_\mu^*, \\[2.5pt]
 \mathsf{b}(\{f, g\}) + \mathsf{b}_{\mu,\mathcal{P}}\mkern-0.5mu (\{ f_\mu, g_\mu\}) = 0
 \end{array}
  \Biggr\}.
\end{equation*}

By Theorem~\ref{Coupling}(\ref{Coupling-i4}) \(\widetilde{A}\) is a self-adjoint extension of \(S\oplus S_\mu\), \(S = \widetilde{A}\cap \mathfrak{H}^2\) and \(S_\mu = \widetilde{A}\cap\mathfrak{C}_\mu^2\). Thus \(\widetilde{A} \in \mathbb{A}_{S,\mathsf{b}}\), and, by the ``only if'' part of Proposition~\ref{pr-bbAbasics}(\ref{pr-bbAbasics-i3})(\ref{pr-bbAbasics-i3-i2}) we have \(\bigl\{\bigl[\widetilde{A}\bigr],\bigl[\mathcal{P}(z)\bigr]\bigr\}\in \Xi\).

\smallskip
\noindent\emph{Step~3c.}
It remains to consider the case $0 < d_0 < d$. Let $\mathcal P(z)$ be factored as in Theorem~\ref{modelP}. Then, by Remark~\ref{rem1}, the matrix polynomial on  the right-hand side of \eqref{reprP} belongs to $[\mathcal P(z)]$. With \(\mathsf{B}_0\) and \(\mathsf{B}\) from \eqref{reprP} we define a closed symmetric linear relation \(S_0\) by \eqref{eq-S0B0B}. Then, by Theorem~\ref{list}, the adjoint of \(S_0\) is given by
\begin{equation} \label{S0*}
S_0^{*}= \Bigl\{ \{f,g\}\in S^*\,:\, \mathsf B\mathsf b(\{f,g\}) = 0 \Bigr\},
\end{equation}
and, by Theorem~\ref{list}(\ref{list-iC})(\ref{list-iC-i2}),  $\mathsf{b}_0 = \mathsf{B}_0 \mathsf{b}|_{S_0^{*}}$ is a boundary mapping for \(S_0\) with Gram matrix $\mathsf{Q}_0 = \bigl(\mathsf{B}_0\mathsf{Q}^{-1}\mathsf{B}_0^*\bigr)^{-1}$.

By Theorem~\ref{modelP} the \(d_0\times 2d_0\) matrix polynomial \(\widehat{\mathcal{P}}(z)\) from \eqref{reprP} belongs to \(\mathbb{P}_{\mathsf{Q}_0}\), its row degrees satisfy \(\mu_1 \geq \cdots \geq \mu_{d_0} \geq 1\) and we set \(\widehat{\mu} = (\mu_1,\ldots,\mu_{d_0})\). Now we proceed as in Step~3b of this proof, almost verbatim in the next paragraph.

We apply Theorem~\ref{Coupling}(\ref{Coupling-i4}) to the Krein space \(\mathfrak{H}\), the closed symmetric linear relation \(S_0\), its boundary mapping \(\mathsf{b}_0:S_0^* \to \mathbb{C}^{2d_0}\) with Gram matrix \(\mathsf{Q}_0\) and the Pontryagin space \(\bigl(\mathfrak{C}_{\widehat{\mu}}, K_{\mathsf{Q}_0,\widehat{\mathcal{P}}}\bigr)\), a symmetric operator \(S_{\widehat{\mu}}\) without eigenvalues and its boundary mapping \(\mathsf{b}_{\widehat{\mu},\widehat{\mathcal{P}}}\) with Gram matrix \(-\mathsf{Q}_0\) introduced in Theorem~\ref{t-fCrks}, and define in \(\mathfrak{H}\oplus\mathfrak{C}_{\widehat{\mu}}\)
\begin{equation} \label{eq-tA3c}
\widetilde{A} =  \Biggl\{
{\scriptstyle \left\{  \begin{bmatrix} f \\[1.5pt] f_{\widehat{\mu}} \end{bmatrix} ,
 \begin{bmatrix}
g \\[1.5pt]  g_{\widehat{\mu}} \end{bmatrix} \right\}}
 \, : \,
 \begin{array}{l}
 \{f,g\} \in  S_0^{*}, \ \ \{ f_{\widehat{\mu}},  g_{\widehat{\mu}}\} \in  S_{\widehat{\mu}}^*, \\[2.5pt]
 \mathsf{b}_0(\{f, g\}) + \mathsf{b}_{\widehat{\mu},\widehat{\mathcal{P}}}\mkern-0.5mu (\{ f_{\widehat{\mu}}, g_{\widehat{\mu}}\}) = 0
 \end{array}
 \mkern-3mu \Biggr\}.
\end{equation}
By Theorem~\ref{Coupling}(\ref{Coupling-i4}), \(\widetilde{A}\) is a self-adjoint extension of \(S_0\oplus S_{\widehat{\mu}}\), \(S_0 = \widetilde{A}\cap \mathfrak{H}^2\) and \(S_{\widehat{\mu}} = \widetilde{A}\cap\mathfrak{C}_{\widehat{\mu}}^2\). Hence \(\widetilde{A} \in \mathbb{A}_{S_0,\mathsf{b}_0}\), and consequently \(\widetilde{A} \in \mathbb{A}_{S,\mathsf{b}}\). By the ``only if'' part of  Proposition~\ref{pr-bbAbasics}(\ref{pr-bbAbasics-i3-i2}) we have
\begin{equation} \label{eq-SP1}
\TtA(z) =
\Bigl\{\{f,g\}\in S_0^{*}\,:\, \widehat{\mathcal{P}}(z) \mathsf{b}_0(\{f,g\}) = 0\Bigr\}.
\end{equation}
Recall that $\mathsf{b}_0 = \mathsf{B}_0 \mathsf{b}|_{S_0^{*}}$ and write \eqref{S0*} and \eqref{eq-SP1} together as
\begin{equation*} 
\TtA(z) =
\biggl\{\{f,g\}\in S^*\,:\,\begin{bmatrix}  \mathcal{P}(z) \mathsf{B}_0 \\ \mathsf{B}\end{bmatrix} \mathsf{b}(\{f,g\}) = 0\biggr\},
\end{equation*}
or, in the spirit of \eqref{reprP},
\begin{equation*} 
\TtA(z) =
\Biggl\{\{f,g\}\in S^*\,:\,
\begin{bmatrix} \widehat{\mathcal{P}}(z) & \mathsf{0} \\[5pt] \mathsf{0} & \mathsf{I}_{d-d_0} \end{bmatrix}
\begin{bmatrix}  \mathsf{B}_0 \\[5pt] \mathsf{B}  \end{bmatrix} \mathsf{b}(\{f,g\}) = 0\Biggr\},
\end{equation*}
that is,
\begin{equation*}
\TtA(z) =
\Bigl\{\{f,g\}\in S^*\,:\, \mathcal{W}(z)\mathcal{P}(z) \mathsf{b}(\{f,g\}) = 0\Bigr\}.
\end{equation*}
Since \(\mathcal{W}(z)\) is unimodular, the last equality proves  \(\bigl\{\bigl[\widetilde{A}\bigr],\bigl[\mathcal{P}(z)\bigr]\bigr\}\in \Xi\).

\smallskip
\noindent\emph{Step~4.} The relation \(\Xi\) is injective. Let \(\bigl\{\bigl[\widetilde{A}\bigr], \bigl[\mathcal{P}(z)\bigr] \bigr\}, \bigl\{\bigl[\widetilde{B}\bigr], \bigl[\mathcal{S}(z)\bigr] \bigr\}\) be in \(\Xi\) and assume that \(\bigl[\mathcal{P}(z)\bigr] = \bigl[\mathcal{S}(z)\bigr]\). Then \(\mathcal{P}(z)\sim\mathcal{S}(z)\), \(T_{\widetilde{A}}(z) \equiv T_{\mathcal{P}}(z)\),  and \(T_{\widetilde{B}}(z) \equiv T_{\mathcal{S}}(z)\).
Lemma~\ref{le-PeSSts} yields \(T_{\mathcal{P}}(z) \equiv T_{\mathcal{S}}(z)\), and thus \(T_{\widetilde{A}}(z) \equiv T_{\widetilde{B}}(z)\). By Theorem~\ref{th-EqSt}(\ref{th-EqSt-i1})$\Rightarrow$(\ref{th-EqSt-i2}),  \(\widetilde{A} \sim \widetilde{B}\). Hence \(\bigl[\widetilde{A}\bigr] = \bigl[\widetilde{B}\bigr]\), proving injectivity.

This completes the proof that \(\Xi\) is a bijection between \(\mathbb{A}_{S,\mathsf{b}}/\mathord{\sim}\) and \(\mathbb{P}_{\mathsf{Q}}/\mathord{\sim}\).
\end{proof}

Theorem~\ref{X} and Corollary~\ref{co-TPadjoint} imply the following result.

\begin{corollary}
If \(\widetilde{A} \in \mathbb A_{S,\mathsf b}\), then $\TtA(z^*) = \TtA(z)^*$ for \(z\in\overline{\mathbb{C}}\).  In particular, $\TtA(z)$ with $z \in \mathbb R\cup\{\infty\}$ is a self-adjoint relation in $\mathfrak H$.
\end{corollary}

When restricted to the class \(\mathbb A_{S,\mathsf b}\) 
\cite[Theorem~7.4]{DLdS84}, referred to in the Introduction, becomes a consequence of
Theorem~\ref{X}. Indeed, written in terms of a binary relation it reads as: The binary relation
\begin{equation*}
\Biggl\{
{ \bigl\{  T(z) ,[\mathcal P(z)] \bigr\}
}
 : \mkern-5mu
 \begin{array}{l}
 T(z) \in \mathbb{M},\ \ \mathcal P(z) \in \mathbb P_\mathsf Q
 \\[2.5pt]
 T(z) = T_{S, \mathsf b, \mathcal P}(z) \ \text{for all}\ z \in \mathbb C
 \end{array}\mkern-2mu\Biggr\}
\end{equation*}
with $\mathbb{M} = \bigl\{\TtA(z)\,:\,\widetilde A \in \mathbb A_{S, \mathsf b} \bigr\}$ is a bijection between $\mathbb M$ and $\mathbb P_\mathsf Q/\sim$.

We conclude this section with a proposition that is, in some sense,  dual to Corollary~\ref{co-EqV}.

\begin{proposition} \label{co-EqVdual}
Let \(\widetilde{A} \in \mathbb A_{S,\mathsf b}\), set \(S_0 = \widetilde{A} \cap \mathfrak{H}^2\), \(\widehat{S} = \widetilde{A} \cap \widehat{\mathfrak{H}}^2\) and assume \(d_0 \gt 0\). Then there exists a relation \(\widetilde{A}_1 \in \mathbb A_{S,\mathsf b}\) with \(\widehat{\mathfrak H}_1 = \widehat{\mathfrak H}\) such that  \(\widetilde{A}_1 \cap \mathfrak{H}^2 = S_0\), \(\widetilde{A}_1 \cap \widehat{\mathfrak{H}}^2 = \widehat{S}\) and \(\widetilde{A}_1\sim\widetilde{A}\) does not hold.
\end{proposition}

\begin{proof}
Let \(\widehat{\mathcal{P}}(z) \in \mathbb{P}_{\mathsf{Q}_0}\) be such that the quadruple  \(\bigl(\mathfrak{C}_{\widehat{\mu}}, K_{\mathsf{Q}_0,\widehat{\mathcal{P}}}, S_{\widehat{\mu}}, \mathsf{b}_{\widehat{\mu},\widehat{\mathcal{P}}}\bigr)\) is a model for \(\bigl(\widehat{\mathfrak{H}}, \kip_{\widehat{\mathfrak{H}}}, \widehat{S}, \widehat{\mathsf{b}}\bigr)\). Assume that \(\widetilde{A}\) is represented by \eqref{eq-tAhb}. By Theorem~\ref{Coupling}(\ref{Coupling-i3}) the formula
\begin{equation} \label{eq-noneqA}
\widetilde{A}_{\mathsf{V}} =  \Biggl\{
{\scriptstyle \left\{  \begin{bmatrix} f \\[1.5pt] \widehat{f} \end{bmatrix} ,
 \begin{bmatrix}
g \\[1.5pt]  \widehat{g} \end{bmatrix} \right\}}
 \, :
 \begin{array}{l}
 \{f,g\} \in  S_0^{*}, \ \ \{ \widehat{f},  \widehat{g}\} \in  \widehat{S}^{*}, \\[2.5pt]
 \mathsf{b}_0(\{f, g\})
 + \mathsf{V}\widehat{\mathsf{b}} (\{ \widehat{f}, \widehat{g}\}) = 0
 \end{array} \mkern-5mu
  \Biggr\},
\end{equation}
gives a bijection between the set of all canonical self-adjoint extensions \(\widetilde{A}_{\mathsf{V}}\) of \(S_0\oplus \widehat{S}\) in \(\widetilde{\mathfrak{H}} = \mathfrak{H} \oplus \widehat{\mathfrak{H}}\) such that \(\widetilde{A}_{\mathsf{V}} \cap \mathfrak{H}^2 = S_0\), \(\widetilde{A}_{\mathsf{V}} \cap \widehat{\mathfrak{H}}^2 = \widehat{S}\) and the set of all \(2d_0\times 2 d_0\) matrices \(\mathsf{V}\) such that \(\mathsf{V}^* \mathsf{Q}_0 \mathsf{V} = \mathsf{Q}_0\). Since in this bijection \(\widetilde{A} = \widetilde{A}_{\mathsf{I}_{2 d_0}}\), \cite[Theorem~4.5]{CDCAOT} yields that  \(\widetilde{A}_{\mathsf{V}} \sim \widetilde{A}\) if and only if there exists a \(d_0\times d_0\) unimodular matrix polynomial \(\mathcal{W}(z)\) such that \(\mathcal{W}(z) \widehat{\mathcal{P}}(z) = \widehat{\mathcal{P}}(z) \mathsf{V}\) for all \(z \in \mathbb{C}\). In particular, for \(z=0\) we have
\begin{equation} \label{eq-impl}
\widetilde{A}_{\mathsf{V}} \sim \widetilde{A} \quad \Rightarrow \quad \mathsf{W} \widehat{\mathsf{P}}_{\mkern-2mu0} = \widehat{\mathsf{P}}_{\mkern-2mu0} \mathsf{V} \quad \Rightarrow \quad \nul \widehat{\mathsf{P}}_{\mkern-2mu0} = \nul(\widehat{\mathsf{P}}_{\mkern-2mu0} \mathsf{V}),
\end{equation}
where \(\mathsf{W} = \mathcal{W}(0)\) is an invertible \(d_0\times d_0\) matrix and \(\widehat{\mathsf{P}}_{\mkern-2mu0} = \widehat{\mathcal{P}}(0)\) is a \(d_0\times 2 d_0\) matrix such that \(\widehat{\mathsf{P}}_{\mkern-2mu0} \mathsf{Q}_0^{-1} \widehat{\mathsf{P}}_{\mkern-2mu0}^* = \mathsf{0}\).

Since \(\widehat{\mathsf{P}}_{\mkern-2mu0}\) has full rank, \(\widehat{\mathsf{P}}_{\mkern-2mu0} \mathsf{Q}_0^{-1} \widehat{\mathsf{P}}_{\mkern-2mu0}^* = \mathsf{0}\) implies that
\(
\nul \widehat{\mathsf{P}}_{\mkern-2mu0} = \ran \bigl(\mathsf{Q}_0^{-1} \widehat{\mathsf{P}}_{\mkern-2mu0}^* \bigr)
\)
is a maximal neutral subspace of the Pontryagin space \(\bigl(\mathbb{C}^{2d_0},\kip_{\mathsf{Q}_0}\bigr)\) with the inner product \([x,y]_{\mathsf{Q}_0} = y^* \mathsf{Q}_0 x\). A fundamental symmetry on this space is \(\mathsf{J} = \sgn \mathsf{Q}_0\), the matrix sign function. Clearly, \(\widehat{\mathsf{P}}_{\mkern-2mu0} \mathsf{J} \in \mathbb{P}_{\mathsf{Q}_0}\) and \(\nul (\widehat{\mathsf{P}}_{\mkern-2mu0} \mathsf{J}) = \mathsf{J}\nul \widehat{\mathsf{P}}_{\mkern-2mu0}\) is another maximal neutral subspace of this Pontryagin space.
Since \(\nul \widehat{\mathsf{P}}_{\mkern-2mu0}\) is neutral,  if \(x, \mathsf{J} x \in \nul \widehat{\mathsf{P}}_{\mkern-2mu0}\),  then \([\mathsf{J}x,x]_{\mathsf{Q}_0} = x^* |\mathsf{Q}_0| x = 0\), implying \(x=0\). Hence
\[
\bigl(\nul \widehat{\mathsf{P}}_{\mkern-2mu0} \bigr) \cap \bigl(\nul (\widehat{\mathsf{P}}_{\mkern-2mu0} \mathsf{J})\bigr) = \{0\}.
\]
Therefore \(\nul( \widehat{\mathsf{P}}_{\mkern-2mu0} \mathsf{J}) \neq \nul \widehat{\mathsf{P}}_{\mkern-2mu0}\). With \(\mathsf{V} = \mathsf{J}\) we have \(\mathsf{V}^* \mathsf{Q}_0 \mathsf{V} = \mathsf{Q}_0\) and \(\nul \widehat{\mathsf{P}}_{\mkern-2mu0} \neq \nul(\widehat{\mathsf{P}}_{\mkern-2mu0} \mathsf{V})\). Thus by the contrapositive of \eqref{eq-impl} it follows that $\widetilde A_1 = \widetilde A_{\mathsf V}$ defined by \eqref{eq-noneqA} is not similar to $\widetilde A$.
\end{proof}

\section{Spectral equivalence}\label{speceq}

\begin{definition}
Let \(\widetilde{A} \in \mathbb{A}_{S,\mathsf{b}}\) and \(\mathcal{P}(z) \in \mathbb{P}_{\mathsf{Q}}\). The self-adjoint relation $\widetilde A$ is said to be {\em spectrally equivalent} to the boundary eigenvalue problem $\BEP\mkern-2mu\bigl(S,\mathsf{b},\mathcal{P}(z)\bigr)$ if the following equalities hold:
\begin{equation*}
\rho( \BEP )=\rho(\widetilde A),\ \
\sigc( \BEP ) = \sigc(\widetilde A), \ \
\sigp( \BEP )=\sigp(\widetilde A), \ \
\sigr( \BEP )=\sigr(\widetilde A),
\end{equation*}
and the Jordan chains of the $\BEP$ at its eigenvalue $\lambda \in \overline{\mathbb{C}}$ correspond bijectively to the Jordan chains of $\widetilde{A}$ at $\lambda$ as defined in Subsection~\ref{partition}.
\end{definition}

In this section we prove that the linearization \(\widetilde A\) defined in Definition~\ref{def-lin} is spectrally equivalent to a boundary eigenvalue problem which it linearizes provided that \(\sigp(\widetilde A) = \sigp(\widetilde A)^*\). In this case, \(\sigr( \BEP ) = \sigr(\widetilde A) = \emptyset\).  A major step in proving this is provided by the following statements, some of which are known; see for example \cite[p.~1]{DLdS86} and \cite[Theorem~7.4]{DLdS93}.

\begin{proposition}\label{pr-rs}
Let $\widetilde A \in \mathbb{A}_{S,\mathsf{b}}$ and  let $\TtA(z)$, $z \in \overline{\mathbb{C}}$, be given by \eqref{Tagain}.
The following statements hold for all $\lambda \in \overline{\mathbb{C}}$.
\begin{enumerate}
\renewcommand*\theenumi{\alph{enumi}}
\renewcommand*\labelenumi{\rm{(\theenumi)}}
\setlength{\itemsep}{1pt}
\item \label{pr-rs-i21}
$ \lambda \in \sigp(\widetilde{A})$ if and only if $\lambda \in \sigp\bigl(\TtA(\lambda)\bigr)$ and
\begin{alignat*}{2}
\widetilde{P}_{\mathfrak H} \nul\bigl(\widetilde{A} - \lambda I_{\widetilde{\mathfrak H}}\bigr)
& =\nul\bigl(\TtA(\lambda)-\lambda I_{\mathfrak H}\bigr) & \quad  &\text{if}\quad  \lambda \in \mathbb C,\\
\widetilde P_{\mathfrak H}\mul\widetilde A &=\mul \TtA(\infty) & \quad  &\text{if}\quad  \lambda = \infty.
\end{alignat*}

\item  \label{pr-rs-i22}
If $\lambda \in \rho(\widetilde A)$, then $\lambda \in \rho\bigl(\TtA(\lambda)\bigr)$.
\end{enumerate}

\noindent Also assume that \(\widetilde{A}\) satisfies {\rm(\ref{pr-rs-i11})-(\ref{pr-rs-i13})} from Lemma~{\rm\ref{lemR}}. The following equalities hold for all $\lambda \in \overline{\mathbb{C}}$.

\begin{enumerate}
\renewcommand*\theenumi{\alph{enumi}}
\renewcommand*\labelenumi{\rm{(\theenumi)}}
\setcounter{enumi}{2}
\setlength{\itemsep}{1pt}
\item \label{pr-rs-i23}
$\lambda \in \rho(\widetilde A)$
 if and only if $\lambda \in  \rho\bigl(\TtA(\lambda)\bigr) $.

\item \label{pr-rs-i24}
$\lambda \in \sigc(\widetilde A)$ if and only if $\lambda \in \sigc\bigl(\TtA(\lambda)\bigr)$.

\item\label{pr-rs-i25}
$\lambda \in \rho\bigl(\TtA(\lambda)\bigr)
 \cup \sigp\bigl(\TtA(\lambda)\bigr)
 \cup \sigc\bigl(\TtA(\lambda)\bigr)$ (disjoint union), that is,  $\lambda \notin \sigr\bigl(\TtA(\lambda)\bigr)$.
\end{enumerate}
\end{proposition}

\begin{proof}
(\ref{pr-rs-i21}) If we write elements like $\widetilde f \in \widetilde{\mathfrak H} = \mathfrak{H} \oplus \widehat{\mathfrak H}$ in the form $\widetilde{f} = \begin{bmatrix} f \\[1.5pt] \widehat f \end{bmatrix}$ with $f \in \mathfrak{H}$ and $\widehat f \in \widehat{\mathfrak H}$, then
$\TtA(z)$ can be expressed in the form
\[
\TtA(z)=
\left\{
\begin{array}{ll}
\left\{\{f,g\} \in S^*  \,:\, {\scriptstyle
\left\{
\begin{bmatrix} f \\[1.5pt] \widehat f \end{bmatrix} ,
\begin{bmatrix} g \\[1.5pt] \widehat g \end{bmatrix} \right\}} \in \widetilde{A}, \ \widehat g-z \widehat f=0\right\},
&  z \in \mathbb{C},
\\[5mm]
\left\{\{f,g\}\in S^*   \,:\, {\scriptstyle
\left\{
\begin{bmatrix} f \\[1.5pt] \widehat f \end{bmatrix} ,
\begin{bmatrix} g \\[1.5pt] \widehat g \end{bmatrix} \right\}} \in \widetilde{A}, \ \widehat f=0\right\},
&  z=\infty.
\end{array}
\right.
\]
Assume $\lambda \in \mathbb C$. If $f \in \nul\bigl(\TtA(\lambda)-\lambda I_{\mathfrak H}\bigr)$ and $f \neq 0$, then there is a vector  $\widehat{f} \in \widehat{\mathfrak H}$ such that $\widetilde{f} = \begin{bmatrix} f \\ \widehat f\end{bmatrix} \in \nul\bigl(\widetilde A-\lambda I_{\widetilde{\mathfrak H}}\bigr)$ and $\widetilde{f} \neq 0$. Conversely, if $\widetilde{f} = \begin{bmatrix} f \\ \widehat{f}\end{bmatrix} \in \nul\bigl(\widetilde A - \lambda I_{\widetilde{\mathfrak H}}\bigr)$ and $\widetilde f \neq 0$, then $f \in \nul\bigl( \TtA(\lambda) - \lambda I_{\mathfrak H}\bigr)$. We show that $f \neq0$.
If $f=0$, then $\{\widehat f, \lambda  \widehat f\} \in  \widehat S$. Since $\sigp(\widehat S)=\emptyset$, $\widehat f=0$, whence the contradiction $0 \neq \widetilde f =0$. This shows $f \neq0$.
Now the equivalence and the equality readily follow. The case $\lambda=\infty$ can be proved in a similar way and therefore omitted.

(\ref{pr-rs-i22})
The equality \eqref{S1} with $z=\lambda \in \mathbb C$ implies the implication $\lambda \in \rho(\widetilde A) \Rightarrow \lambda \in \rho\bigl(\TtA(\lambda)\bigr)$. If $\infty \in \rho(\widetilde A)$, then $\widetilde A$ is bounded, and then, by \eqref{Tagain}, $\TtA(\infty)$ is also bounded. Hence  $\infty \in \rho\bigl( \TtA(\infty) \bigr)$.

For the next three items, assume that the equivalent conditions in Lemma~\ref{lemR}
hold for \(\widetilde A\). 

(\ref{pr-rs-i23})
Assume $\lambda \in \rho\bigl(\TtA(\lambda)\bigr)$. Then $\lambda \not \in \sigp\bigl(\TtA(\lambda)\bigr)$, hence, by (\ref{pr-rs-i21}), $\lambda \not \in \sigp(\widetilde A)$. From Lemma~\ref{lemR}(\ref{pr-rs-i13}) it follows that the closed operator $(\widetilde A-\lambda I_{\widetilde{\mathfrak H}})^{-1}$ is at least densely defined and by \eqref{S1} its compression to $\mathfrak H$ is bounded. Since $\widehat{\mathfrak H}$ is finite dimensional,  \cite[Corollary~1.3(ii)]{ACD16} implies that $(\widetilde A-\lambda I_{\widetilde{\mathfrak H}})^{-1}$ is bounded on $\widetilde{\mathfrak H}$. Hence $\lambda \in \rho(\widetilde A)$. The converse implication follows from (\ref{pr-rs-i22}). The statement for $\lambda=\infty$ can be proved along the same lines and therefore omitted.

(\ref{pr-rs-i24}) Assume $\lambda \in \mathbb C$. Consider $\lambda \in \sigc\bigl(\TtA(\lambda)\bigr)$. Then
$\lambda \not \in \sigp\bigl(\TtA(\lambda)\bigr)$, hence, by (\ref{pr-rs-i21}), $\lambda \not \in \sigp(\widetilde A)$; and $\lambda \not \in \rho\bigl(\TtA(\lambda)\bigr)$, hence, by (\ref{pr-rs-i22}), $\lambda \not \in \rho(\widetilde A)$. Now Lemma~\ref{lemR}(\ref{pr-rs-i13}) implies $\lambda \in \sigc(\widetilde A)$.
Now consider $\lambda \in \sigc(\widetilde A)$. Then $\lambda \not \in \sigp(\widetilde A)$, hence, by (\ref{pr-rs-i21}), $\lambda \not \in \sigp\bigl(\TtA(\lambda)\bigr)$; and $\lambda \not \in \rho(\widetilde A)$, hence, by (\ref{pr-rs-i23}), $\lambda \not \in \rho\bigl(\TtA(\lambda)\bigr)$; and
\[
\overline{\dom}\,(\widetilde A - \lambda I_{\widetilde{\mathfrak H}})^{-1}
= \overline{\ran}\,(\widetilde A-\lambda I_{\widetilde{\mathfrak H}})
= \widetilde{\mathfrak H},
\]
which, by \cite[Corollary~1.3 (i)]{ACD16} and \eqref{S1}, implies
\[
\mathfrak H =
\overline{\dom}\,\widetilde{P}_\mathfrak H\bigl(\widetilde{A} - \lambda  I_{\widetilde{\mathfrak H}}\bigr)^{-1}\bigl.\bigr|_{\mathfrak{H}}
= \overline{\dom}\,\bigl(\TtA(\lambda)-\lambda I_{\mathfrak H}\bigr)^{-1}
= \overline{\ran}\,\bigl(\TtA(\lambda)- \lambda I_{\mathfrak H}\bigr).
\]
Hence $\lambda \in \sigc\bigl(\TtA(\lambda)\bigr)$. The proof for the case $\lambda = \infty$ is similar and therefore omitted.

(\ref{pr-rs-i25}) For $\lambda \in \mathbb{C}$ this item follows from Lemma~\ref{lemR}(\ref{pr-rs-i13}), and items  (\ref{pr-rs-i21}), (\ref{pr-rs-i23}) and (\ref{pr-rs-i24}) above. For $\lambda=\infty$, we note that, by the
Krein space version \cite[Theorem~4.1]{ADW} of Stenger's lemma in \cite{S},
the compression $\TtA(\infty)$ is self-adjoint. Hence, by Lemma~\ref{lemR}, $\infty \notin \sigr\bigl(\TtA(\infty)\bigr)$.
\end{proof}

The following observation is related to Remark~\ref{re-defop}.

\begin{remark}
If $\widetilde{A} \in \mathbb{A}_{S,\mathsf{b}}$ is definitizable, then every $\widetilde{B} \in \mathbb{A}_{S,\mathsf{b}}$ with \(\rho\bigl(\widetilde{B}\bigr) \neq \emptyset\) is definitizable. This follows from \cite[Theorem~2.2 and Remark~2.3]{ABT} and \cite[Theorem~2.2]{B}.
\end{remark}

\begin{lemma} \label{simplercase}
Assume $\mathcal{P}(z) \in \mathbb{P}_{\mathsf{Q}}$ is represented as in Theorem~{\rm\ref{modelP}} and let $\lambda \in \overline{\mathbb C}$. For a sequence $f_0, f_1, \ldots, f_n \in \mathfrak{H}$ the following statements  are equivalent.

\begin{enumerate}
\renewcommand*\theenumi{\alph{enumi}}
\renewcommand*\labelenumi{\rm{(\theenumi)}}

\item \label{simplercase-i1}
The sequence is a (maximal) Jordan chain for the $\BEP\mkern-2mu%%
\bigl(S, \mathsf b, \mathcal P(z)\bigr)$ at $\lambda$.
\item \label{simplercase-i2}
The  sequence  is  a (maximal) Jordan chain for the $\BEP\mkern-2mu\bigl(S_0, \mathsf{b}_0, \widehat{\mathcal{P}}(z)\bigr)$ at $\lambda$.
\end{enumerate}
\noindent  These statements with $\lambda = \infty$ are also equivalent to

\begin{enumerate}
\renewcommand*\theenumi{\alph{enumi}}
\renewcommand*\labelenumi{\rm{(\theenumi)}}
\setcounter{enumi}{2}
\item \label{simplercase-i3}
The sequence is a (maximal) Jordan chain for the $\BEP\mkern-2mu\bigl(S_0^{-1}, \mathsf{b}_0\mkern0.5mu\iota, \widehat{\mathcal{R}}(z)\bigr)$ at the eigenvalue $0$, where $\iota(\{u,v\}) = \{v,u\}$, $\{u,v\} \in \mathfrak{H}^2$, and \(\mathsf{b}_0\mkern0.5mu\iota:S_0^{-*}\to \mathbb{C}^{2d_0}\) defines a boundary mapping for $S_0^{-1}$ with Gram matrix $-\mathsf{Q}_0$ and the matrix polynomial $\widehat{\mathcal{R}}(z)$ is the row-reversal of $\widehat{\mathcal{P}}(z)$.
\end{enumerate}
\end{lemma}

\begin{proof}
We use the convention \(f_{-1}=0\). Assume \(\lambda\in\mathbb C\). Set
\[
\mathsf b_k
=
\mathsf b\bigl(\{f_k,\lambda f_k+f_{k-1}\}\bigr),
\qquad
\mathsf b_{0k}
=
\mathsf B_0\mathsf b_k,
\qquad
k\in\{0,1,\ldots,n\}.
\]
By Remark~\ref{rem1}, the polynomial \(\mathcal P(z)\) is equivalent to the matrix polynomial on the right-hand side of \eqref{reprP}; denote this polynomial by \(\mathcal T(z)\). By Lemma~\ref{Chains}, the Jordan chains for \(\BEP\mkern-2mu\bigl(S,\mathsf b,\mathcal P(z)\bigr)\) at \(\lambda\) are unchanged if \(\mathcal P(z)\) is replaced by \(\mathcal{T}(z)\). For the \(\BEP\mkern-2mu\bigl(S,\mathsf b,\mathcal T(z)\bigr)\), the boundary equations in Definition~\ref{def-Jc} of a Jordan chain are equivalent to 
\[
\mathsf B\mathsf b_k=0,
\qquad
\sum_{j=0}^k \frac{1}{j!}
\widehat{\mathcal P}^{(j)}(\lambda)\mathsf b_{0(k-j)}=0,
\qquad
k\in\{0,1,\ldots,n\}.
\]

By the representation \eqref{S0*} of \(S_0^*\), the condition
\[
\{f_k,\lambda f_k+f_{k-1}\}\in S^*
\quad\text{and}\quad
\mathsf B\mathsf b_k=0
\]
is equivalent to \(\{f_k,\lambda f_k+f_{k-1}\}\in S_0^*\). 
Moreover, on \(S_0^*\) we have \(\mathsf b_0=\mathsf B_0\mathsf b\bigl.\bigr|_{S_0^*}\), and hence
\[
\mathsf b_0\bigl(\{f_k,\lambda f_k+f_{k-1}\}\bigr)
=
\mathsf B_0\mathsf b_k
=
\mathsf b_{0k}, \qquad k\in\{0,1,\ldots,n\}.
\]
Therefore the preceding conditions are precisely the defining conditions for the sequence \(f_0,f_1,\ldots,f_n\) to be a Jordan chain for the boundary eigenvalue problem \(\BEP\mkern-2mu\bigl(S_0,\mathsf b_0,\widehat{\mathcal P}(z)\bigr)\) at \(\lambda\). This proves the equivalence of items~(\ref{simplercase-i1}) and~(\ref{simplercase-i2}) for \(\lambda\in\mathbb C\).

For \(\lambda=\infty\), the same argument applies to the row-reversal polynomials, using the second claim in Remark~\ref{rem1} and Lemma~\ref{Chains}. 

It remains to prove (\ref{simplercase-i2})$\Leftrightarrow$(\ref{simplercase-i3}) when $\lambda =\infty$. By definition, a Jordan chain for the 
\(\BEP\mkern-2mu\bigl(S_0^{-1},\mathsf b_0\iota,\widehat{\mathcal R}(z)\bigr)\) at \(0\) satisfies \(\{f_k,f_{k-1}\}\in S_0^{-*}\) and 
\[
\sum_{j=0}^k \frac{1}{j!}
\widehat{\mathcal R}^{(j)}(0)
(\mathsf b_0\iota)\bigl(\{f_{k-j},f_{k-j-1}\}\bigr)=0.
\]
Since \(\{f_k,f_{k-1}\}\in S_0^{-*}\) if and only if \(\{f_{k-1},f_k\}\in S_0^*\) 
and 
\[
(\mathsf b_0\iota)\bigl(\{f_{k-j},f_{k-j-1}\}\bigr)
=
\mathsf b_0\bigl(\{f_{k-j-1},f_{k-j}\}\bigr),
\]
these are precisely the defining conditions for a Jordan chain of the boundary eigenvalue problem \(\BEP\mkern-2mu\bigl(S_0,\mathsf b_0,\widehat{\mathcal P}(z)\bigr)\) at \(\infty\). This proves
(\ref{simplercase-i2})\(\Leftrightarrow\)(\ref{simplercase-i3}).

Finally, the same comparisons hold for chains of arbitrary finite length. Therefore a vector \(f_{n+1}\) extends the given chain in one of the corresponding boundary eigenvalue problems if and only if it extends it in the other one. Hence maximality is preserved under the above equivalences.
\end{proof}

We are now ready to formulate the Spectral Equivalence Theorem, the main result of this section. Proposition~\ref{pr-rs} already provides information on the point spectra and, under an additional hypothesis, yields the coincidence of the resolvent sets and the continuous spectra. Hence to obtain spectral equivalence it remains to relate the Jordan chains. This is accomplished in the next theorem.

\begin{theorem} \label{th-speq} 
Let $\mathcal{P}(z) \in \mathbb{P}_{\mathsf{Q}}$ and $\widetilde{A}\in \mathbb{A}_{S,\mathsf{b}}$. If \(\bigl\{[\widetilde{A}],[\mathcal{P}(z)]\bigr\}\in\Xi\)
then the following statements hold:
\begin{enumerate}
\renewcommand*\theenumi{\Roman{enumi}}
\renewcommand*\labelenumi{\rm{(\theenumi)}}
\setlength{\itemsep}{1pt}
\item \label{th-speq-A1}
$\sigp(\BEP)=\sigp(\widetilde A)$.
\item \label{th-speq-A2}
Let $\lambda \in \overline{\mathbb C}$. There is a bijection between
the Jordan chains $f_0, f_1, \ldots, f_n \in \mathfrak H$ of the $\BEP$ at $\lambda$
and the Jordan chains $\widetilde{f}_0, \widetilde{f}_1, \ldots, \widetilde {f}_n \in \mathfrak{H} \oplus (\mathfrak{C}_{\widehat{\mu}}, K_{\mathsf{Q}_0, \widehat{\mathcal{P}}})$ of $\widetilde{A}$ at $\lambda$.  In particular, corresponding Jordan chains have the same maximal lengths. The bijection between the Jordan chains is given by the orthogonal projection ${\widetilde{P}}_\mathfrak{H}$ in $\widetilde{\mathfrak H}$ onto $\mathfrak{H}$.
    \end{enumerate}
If  $\sigp(\widetilde{A})=\sigp(\widetilde{A})^*$, then the $\BEP$ and $\widetilde A$ are spectrally equivalent, in detail:
\begin{enumerate}
\renewcommand*\theenumi{\Roman{enumi}}
\renewcommand*\labelenumi{\rm{(\theenumi)}}
\setcounter{enumi}{2}
\setlength{\itemsep}{1pt}
\item \label{th-speq-B1o}
\begin{enumerate}
\renewcommand*\theenumii{\roman{enumii}}
\renewcommand*\labelenumii{\rm{(\theenumii)}}
\setlength{\itemsep}{1pt}
\item \label{th-speq-B1}
$\rho(\BEP)=\rho\bigl(\widetilde{A}\bigr)$.
\item \label{th-speq-B2}
If {\rm (\ref{th-speq-B1})} holds and $\lambda \in \rho(\BEP)$, the solution $f \in \mathfrak H $ of the boundary eigenvalue problem $\BEP\mkern-2mu\bigl(S, \mathsf{b}, \mathcal{P}(z)\bigr)$ at $\lambda$ (see \eqref{BEPsys2} and \eqref{BEPsys2inf})  is given by
\begin{equation*}
 f= \left\{
 \begin{array}{ll} {\widetilde P}_{\mathfrak H} \bigl(\widetilde {A} - \lambda I_{\widetilde{\mathfrak H}}\bigr)^{-1}\begin{bmatrix}
h \\ 0 \end{bmatrix},& \lambda \in \mathbb C, \\[9pt]
{\widetilde{P}}_{\mathfrak{H}} \widetilde{A} \begin{bmatrix}
h \\ 0 \end{bmatrix}, & \lambda = \infty.
\end{array} \right.
\end{equation*}
\end{enumerate}

\item \label{th-speq-C} 
    $\sigc(\BEP) = \sigc(\widetilde{A})$.
\item\label{th-speq-D}  
$\overline{\mathbb{C}}=\rho(\BEP) \cup \sigc(\BEP) \cup \sigp(\BEP)$ (disjoint union), that is, \(\sigr(\BEP) = \emptyset\).
\end{enumerate}
 \end{theorem}

\begin{proof}
Items (\ref{th-speq-A1}), (\ref{th-speq-B1o})(\ref{th-speq-B1}), (\ref{th-speq-C}), and (\ref{th-speq-D}) follow from Proposition~\ref{pr-rs}, parts (\ref{pr-rs-i21}), (\ref{pr-rs-i23}), (\ref{pr-rs-i24}), and (\ref{pr-rs-i25}), respectively. Item (\ref{th-speq-B1o})(\ref{th-speq-B2}) follows for $\lambda \in \mathbb C$ from \eqref{S1} and for $\lambda = \infty$ from the equality $T_{\mkern-2mu\widetilde{A}}(\infty) = \widetilde{P}_{\mathfrak H}\widetilde A\bigl.\bigr|_\mathfrak H$.

Proof of (\ref{th-speq-A2}) for $\lambda \in \mathbb{C}$.

In view of Lemma~\ref{simplercase} and item (\ref{simplercase-i2}) we may replace the $\BEP$ by the $\BEP_0 =\BEP\mkern-2mu\bigl(S_0, \mathsf{b}_0, \widehat{\mathcal{P}}(z)\bigr)$ and write \(\widetilde{A}\) (see \eqref{eq-tA3c}) in the form
\begin{equation*}  
\widetilde{A}  =  \Biggl\{
{\scriptstyle \left\{  \begin{bmatrix} f \\[1.5pt] f_{\widehat{\mu}} \end{bmatrix} ,
 \begin{bmatrix}
g \\[1.5pt]  g_{\widehat{\mu}} \end{bmatrix} \right\}}
 \, : \,
 \begin{array}{l}
 \{f,g\} \in  S^*_0, \ \ \{ f_{\widehat{\mu}},  g_{\widehat{\mu}}\} \in  S_{\widehat{\mu}}^*, \\[2.5pt]
  \mathsf{b}_0(\{f, g\}) +  \mathsf{b}_{\widehat{\mu},\widehat{\mathcal{P}}}(\{ f_{\widehat{\mu}}, g_{\widehat{\mu}}\}) = 0
 \end{array}\mkern-2mu\Biggr\}.
\end{equation*}
Assume $\widetilde{f}_0, \widetilde{f}_1, \ldots, \widetilde {f}_n \in \dom \widetilde{A}$ form a Jordan chain of $\widetilde{A}$ at
$\lambda$ and write
\[
\widetilde f_k=\begin{bmatrix} f_k \\[1mm]  f_{\widehat{\mu} k} \end{bmatrix}, \quad \ k\in \{0, \ldots, n\}.
\]
Then, since $\widetilde{f}_0\neq 0$ and $S_{\widehat{\mu}} = \widetilde{A}\cap \mathfrak{C}_{\widehat{\mu}}^2$ has no eigenvalues, $f_0\neq 0$, and
\[
\bigl\{ f_k, \lambda f_k+ f_{k-1} \bigr\} \in S^*_0,
\quad
\bigl\{ f_{\widehat{\mu} k}, \lambda  f_{\widehat{\mu} k} + f_{\widehat{\mu} (k-1)} \bigr\} \in  S_{\widehat{\mu}}^* \quad \text{and} \quad
\mathsf{b}_{0k} + \mathsf{b}_{\widehat{\mu} k} = 0,
\]
where
\[
\mathsf{b}_{0k} = \mathsf{b}_0\bigl(\bigl\{f_k, \lambda f_k + f_{k-1}\bigr\}\bigr) \quad \text{and} \quad
\mathsf{b}_{\widehat{\mu} k} =
\mathsf{b}_{\widehat{\mu},\widehat{\mathcal{P}}}\mkern-2mu\bigl(\bigl\{ f_{\widehat{\mu} k}, \lambda f_{\widehat{\mu} k} + f_{\widehat{\mu} (k-1)}\bigr\}\bigr).
\]
It follows from \eqref{Smustar} and the definition of the boundary mapping following it that
\[
(z-\lambda) f_{\widehat{\mu} k}(z)-f_{\widehat{\mu} (k-1)}\mkern-1mu(z)\equiv \widehat{\mathcal{P}}(z)\mathsf{b}_{\widehat{\mu} k} \equiv - \widehat{\mathcal{P}}(z) \mathsf{b}_{0k}, \quad \ k\in \{0, \ldots, n\}.
\]
From the telescoping form of the expression on the left-hand side of this equality we obtain
\begin{equation} \label{telescope1}
(z-\lambda)^{n+1} f_{\widehat{\mu} n}(z) \equiv - \widehat{\mathcal{P}}(z)\sum_{j=0}^n (z-\lambda)^j  \mathsf b_{0j}.
\end{equation}
Next, we decompose $\widehat{\mathcal{P}}(z)$ in powers of $(z-\lambda)$:
\[
\widehat{\mathcal{P}}(z)=\sum_{j=0}^n \dfrac{1}{j!}\widehat{\mathcal{P}}^{(j)}\mkern-2mu(\lambda)(z-\lambda)^j + O\bigl((z-\lambda)^{n+1}\bigr) \quad \text{as} \quad z \rightarrow \lambda,
\]
substitute this in the right-hand side of \eqref{telescope1}
and equate the coefficients of $(z-\lambda)^k$ on both sides of the resulting equality, to get
\[
\sum_{j=0}^k \dfrac{1}{j!} 
\widehat{\mathcal{P}}^{(j)}\mkern-2mu(\lambda)\mathsf b_{0(k-j)}=0  
\quad \text{for all} \quad k\in \{0, \ldots, n\}.
\]
Hence the vectors $f_0, f_1, \ldots, f_n \in \dom S_0^{*}$ form a Jordan chain for the $\BEP_0$ at the eigenvalue $\lambda$.

Now we prove the converse.  Let $f_0, f_1, \ldots, f_n \in \dom S_0^*$ be a Jordan chain for the $\BEP_0$  at the eigenvalue $\lambda$. That is, assume that for all $k \in \{0,\ldots,n\}$ we have
\begin{equation}\label{eqpJc21}
\{f_k, \lambda f_k + f_{k-1}\}\in S_0^* \quad \text{and} \quad \sum_{j=0}^k \frac{1}{j!}\widehat{\mathcal{P}}^{(j)}\mkern-2mu(\lambda) \mathsf{b}_{0(k-j)} = 0,
\end{equation}
with $f_{-1} = 0$, $f_0 \neq 0$ and $\mathsf{b}_{0k}=\mathsf{b}_0(\{f_{k}, \lambda f_{k}+f_{k-1}\})$.

Define a sequence of $d_0 \times 1$ vector polynomials \(f_{\widehat{\mu} k}(z) \) with $k \in \{0,\ldots, n\}$ as follows
\[
f_{\widehat{\mu} k}(z) =  - \sum_{l=0}^{k} \sum_{i = 0}^{\infty} \dfrac{(z-\lambda)^i}{(l+1+i)!}\widehat{\mathcal{P}}^{(l+1+i)}\mkern-2mu(\lambda) \mathsf{b}_{0(k-l)}, \quad z \in \mathbb{C}.
\]
In the above formula and in this proof, we adopt the convention \(0^0 = 1\). Since \(\widehat{\mathcal{P}}(z)\) is a polynomial, the above sum is finite, so \(f_{\widehat{\mu} k}(z)\) is a vector polynomial in \(\mathbb{C}^{d_0}\mkern-2mu[z]\).

Let \(k = 0\). For \(z = \lambda\) we have
\[
f_{\widehat{\mu} 0}(\lambda) = - \widehat{\mathcal{P}}^{\prime}(\lambda) \mathsf{b}_{00}
\]
and for \(z \in \mathbb{C}\setminus\{ \lambda\}\)
\begin{align*}
f_{\widehat{\mu} 0}(z) & = \frac{-1}{z-\lambda} \sum_{i = 1}^{\infty} \dfrac{(z-\lambda)^i}{i!}\widehat{\mathcal{P}}^{(i)}\mkern-2mu(\lambda) \mathsf{b}_{00} \\
& = \frac{-1}{z-\lambda} \bigl(\widehat{\mathcal{P}}(z) - \widehat{\mathcal{P}}(\lambda) \bigr) \mathsf{b}_{00}.
\end{align*}
Therefore \(f_{\widehat{\mu} 0} \in \mathfrak{C}_{\widehat{\mu}}\).

For \(z = \lambda\) and $k \in \{1,\ldots, n\}$  we have a particularly simple formula
\begin{align} \nonumber
 f_{\widehat{\mu} (k-1)}\mkern-1mu(\lambda) & = - \sum_{l=0}^{k-1} \dfrac{1}{(l+1)!}\widehat{\mathcal{P}}^{(l+1)}\mkern-2mu(\lambda) \mathsf{b}_{0(k-1-l)} \\
\label{eq-simpl}
 & = - \sum_{j=1}^{k} \dfrac{1}{j!}\widehat{\mathcal{P}}^{(j)}\mkern-2mu(\lambda) \mathsf{b}_{0(k-j)} \\
 & = \widehat{\mathcal{P}}(\lambda) \mathsf{b}_{0k}, \nonumber
\end{align}
where the last equality follows from the equality in \eqref{eqpJc21}.

For all $k \in \{1,\ldots, n\}$ and all \(z\in\mathbb{C}\) we have
\begin{align*}
f_{\widehat{\mu} (k-1)}\mkern-1mu(z)-f_{\widehat{\mu} (k-1)}\mkern-1mu(\lambda)
& = - \sum_{l=0}^{k-1} \sum_{i = 1}^{\infty} \dfrac{(z-\lambda)^i}{(l+1+i)!}\widehat{\mathcal{P}}^{(l+1+i)}\mkern-2mu(\lambda) \mathsf{b}_{0(k-1-l)} \\
& = - (z-\lambda) \sum_{l=0}^{k-1} \sum_{i = 1}^{\infty} \dfrac{(z-\lambda)^{i-1}}{(l+1+i)!}\widehat{\mathcal{P}}^{(l+1+i)}\mkern-2mu(\lambda) \mkern1mu\mathsf{b}_{0(k-1-l)} \\
& = - (z-\lambda) \sum_{l=1}^{k} \sum_{j = 0}^{\infty} \dfrac{(z-\lambda)^{j}}{(l+1+j)!}\widehat{\mathcal{P}}^{(l+1+j)}\mkern-2mu(\lambda) \mkern1mu\mathsf{b}_{0(k-l)} \\
& = (z-\lambda) f_{\widehat{\mu} k}(z) + (z-\lambda) \sum_{j = 0}^{\infty} \dfrac{(z-\lambda)^{j}}{(1+j)!}\widehat{\mathcal{P}}^{(1+j)}\mkern-2mu(\lambda) \mkern1mu\mathsf{b}_{0k} \\
& = (z-\lambda) f_{\widehat{\mu} k}(z) + \bigl( \widehat{\mathcal{P}}(z) -  \widehat{\mathcal{P}}(\lambda) \bigr) \mkern1mu \mathsf{b}_{0k}.
\end{align*}
Therefore for every \(z\in\mathbb{C}\setminus\{\lambda\}\) and every $k \in \{1,\ldots, n\}$
\begin{equation}\label{eq-rec}
f_{\widehat{\mu} k}(z) = \frac{1}{z-\lambda}\Bigl(f_{\widehat{\mu} (k-1)}\mkern-1mu(z)-f_{\widehat{\mu} (k-1)}\mkern-1mu(\lambda) - \bigl( \widehat{\mathcal{P}}(z) -  \widehat{\mathcal{P}}(\lambda) \bigr)  \mathsf{b}_{0k} \Bigr) \in \mathbb{C}^{d_0}\mkern-2mu[z].
\end{equation}
As \(f_{\widehat{\mu} 0} \in \mathfrak{C}_{\widehat{\mu}}\), applying the recursion in \eqref{eq-rec} implies that \(f_{\widehat{\mu} k} \in \mathfrak{C}_{\widehat{\mu}}\) for all $k \in \{1,\ldots, n\}$.

Since we established \(f_{\widehat{\mu} (k-1)}\mkern-1mu(\lambda) = \widehat{\mathcal{P}}(\lambda) \mathsf{b}_{0k}\) in \eqref{eq-simpl}, we can rewrite \eqref{eq-rec} as follows
\begin{equation}\label{eq-dd}
(z-\lambda)f_{\widehat{\mu} k}(z)-f_{\widehat{\mu} (k-1)}\mkern-1mu(z)\equiv
-\widehat{\mathcal{P}}(z)\mathsf{b}_{0 k} \quad \text{for all} \quad k \in \{1,\ldots, n\}.
\end{equation}
It follows from \eqref{eq-dd} that for $k\in \{0,\ldots,n\}$
\[
\bigl\{ f_{\widehat{\mu} k}, \lambda f_{\widehat{\mu} k}+ f_{\widehat{\mu} (k-1)} \bigr\} \in S_{\widehat{\mu}}^* \quad \text{and} \quad
-\mathsf{b}_{0k} = \mathsf{b}_{\widehat{\mu},\widehat{\mathcal{P}}}\mkern-1.5mu\bigl(\bigl\{f_{\widehat{\mu} k}, \lambda f_{\widehat{\mu} k}+ f_{\widehat{\mu} (k-1)}\bigr\}\bigr).
\]
Hence the vectors $\widetilde{f}_0, \widetilde{f}_1, \dots \widetilde{f}_n$ with
\begin{equation*} 
\widetilde{f}_k = \begin{bmatrix}  f_k \\[1mm]  f_{\widehat{\mu} k} \end{bmatrix}, \quad k \in \{0,\ldots,n\},
\end{equation*}
belong to $\dom \widetilde{A}$ and form a Jordan chain for $\widetilde{A}$ at the eigenvalue $\lambda$.

\smallskip

We now come to the case $\lambda=\infty$. Recall from Subsection~\ref{partition}
that a sequence $\widetilde f_0,\ldots,\widetilde f_n$ is a Jordan chain of
$\widetilde A$ at $\infty$ if and only if it is a Jordan chain of $\widetilde
A^{-1}$ at $0$.

Moreover, by Lemma~\ref{simplercase}(\ref{simplercase-i3}) and the definition of Jordan chains for a $\BEP$ at $\infty$, the same sequence $f_0,\ldots,f_n$ is a Jordan chain of the $\BEP_0$ at $\infty$ if and only if it is a Jordan chain of the \(\BEP\mkern-2mu\bigl(S_0^{-1}, \mathsf{b}_0\iota, \widehat{\mathcal{R}}(z)\bigr)\) at the eigenvalue $0$, where $\widehat{\mathcal{R}}(z)$ is the row-reversal of
$\widehat{\mathcal{P}}(z)$.

A direct inspection of \eqref{Tagain} shows that, for all $z\in\overline{\mathbb{C}}$ we have \(T_{\widetilde A^{^{-1}}}\mkern-2mu(z)=\bigl(\TtA(1/z)\bigr)^{-1}\). Since $\widetilde A$ is a linearization of the $\BEP_0$, for all $z\in\overline{\mathbb{C}}$ we have $\TtA(z) = T_{S_0,\mathsf{b}_0,\widehat{\mathcal{P}}}(z)$. Hence \(T_{\widetilde A^{^{-1}}}\mkern-2mu(z)=\bigl(T_{S_0,\mathsf{b}_0,\widehat{\mathcal{P}}}(1/z)\bigr)^{-1}\)
for all $z\in\overline{\mathbb{C}}$. Together with \eqref{eq-la1ola} this yields
\[
T_{\widetilde A^{^{-1}}}\mkern-2.5mu(z) = T_{S_0^{-1}\mkern-3mu, \mkern2mu\mathsf{b}_0 \iota, \widehat{\mathcal{R}}}(z),
\qquad z\in\overline{\mathbb{C}}.
\]
Thus $\widetilde A^{-1}$ is a linearization of the $\BEP\mkern-2mu\bigl(S_0^{-1}, \mathsf{b}_0\iota, \widehat{\mathcal{R}}(z)\bigr)$. Applying the already proved case $\lambda\in\mathbb{C}$ with $\lambda=0$ to this inverse problem yields the desired bijection between Jordan chains via the projection $\widetilde P_{\mathfrak H}$. Translating it back to $\lambda=\infty$ completes the proof.
\end{proof}

\appendix

\section{Adapted previous results} \label{AppA}

For the first three items in the following Coupling Theorem we refer to \cite[Theorem~3.1]{CDCAOT}.

\begin{theorem}\label{Coupling}
For $k\in \{1,2\}$ let $S_k$ be a closed symmetric linear relation in a Krein space $(\mathfrak{H}_k, \kip_{\mathfrak{H}_k})$ with defect numbers $d^-_k, d^+_k \in \{0\}\cup\mathbb{N}$, $\delta_k =d^-_k + d^+_k$.

\begin{enumerate}
 \renewcommand*\theenumi{\roman{enumi}}
\renewcommand*\labelenumi{\rm{(\theenumi)}}
\setlength{\itemsep}{1pt} 

\item \label{Coupling-i1}
$S_1 \oplus S_2$ has a canonical self-adjoint extension $\widetilde{A}$ in the direct sum Krein space $\widetilde{\mathfrak{H}} = \mathfrak{H}_1 \oplus \mathfrak{H}_2$ such that $\widetilde{A}\cap \mathfrak{H}_k^2 = S_k$ for all $ k \in \{1,2\}$ if and only if
\begin{equation}
\label{definds} d_1^+=d_2^- \quad \text{and} \quad d_1^-=d_2^+.
\end{equation}

\item \label{Coupling-i2}
Assume that \eqref{definds} holds. If $S_1$ ($S_2$) is self-adjoint in $\mathfrak{H}_1$ ($\mathfrak{H}_2$), then $S_2$ ($S_1$) is self-adjoint, and $\widetilde{A}$ in {\rm (\ref{Coupling-i1})} is unique and given by $\widetilde{A}=S_1\oplus S_2$.

 \item \label{Coupling-i3}
Assume that \eqref{definds} holds and that $S_1$ and $S_2$ are not self-adjoint. Set $\delta=\delta_1=\delta_2\in\mathbb{N}$. Let
$\mathsf{b}_k: S^*_k \rightarrow \mathbb{C}^{\delta}$ be a boundary mapping for $S_k$ with Gram matrix $\mathsf{Q}_k$, $ k \in \{1,2\}$. The formula
\begin{equation} \label{eq-PswACDR}
\widetilde{A} =  \Biggl\{
{\scriptstyle \left\{  \begin{bmatrix} f_1 \\[1.5pt] f_2 \end{bmatrix} ,
 \begin{bmatrix}
g_1 \\[1.5pt]  g_2 \end{bmatrix} \right\}}
 \, : \,
 \begin{array}{l}
 \{f_1,g_1\} \in  S_1^*, \mkern10mu \{ f_2,  g_2\} \in  S_2^*, \\[2.5pt]
 \mathsf{b}_1\mkern-2.5mu (\{f_1, g_1\}) + \Gamma \mathsf{b}_2\mkern-1mu (\{ f_2, g_2\}) = 0
 \end{array}
  \Biggr\}
\end{equation}
gives a bijection between all canonical self-adjoint extensions $\widetilde{A}$ of $S_1 \oplus S_2$ in $\mathfrak{H}_1 \oplus \mathfrak{H}_2$ with $\widetilde{A} \cap \mathfrak{H}_k^2=S_k$, $k \in \{1,2\}$, and all $\delta \times \delta$ invertible matrices $\Gamma$  with
\begin{equation*} 
\mathsf{Q}_2+ \Gamma^*\mathsf{Q}_1\Gamma = \mathsf{0}.
\end{equation*}

 \item \label{Coupling-i4}
Assume that \eqref{definds} holds and that $S_1$ and $S_2$ are not self-adjoint. Set $\delta=\delta_1=\delta_2\in\mathbb{N}$. Let $\mathsf{b}_1: S^*_1 \rightarrow \mathbb{C}^{\delta}$ be a boundary mapping for $S_1$ with Gram matrix $\mathsf{Q}_1$. The formula
\begin{equation} \label{eq-PswACDRGisI}
\widetilde{A} =  \Biggl\{
{\scriptstyle \left\{  \begin{bmatrix} f_1 \\[1.5pt] f_2 \end{bmatrix},
 \begin{bmatrix}
g_1 \\[1.5pt]  g_2 \end{bmatrix} \right\}}
 \, : \,
\begin{array}{l}
 \{f_1,g_1\} \in  S_1^*, \mkern10mu \{ f_2,  g_2\} \in  S_2^*, \\[2.5pt]
\mathsf{b}_1\mkern-2.5mu (\{f_1, g_1\}) + \mathsf{b}_2\mkern-1mu (\{ f_2, g_2\}) = 0
 \end{array}
  \Biggr\}
\end{equation}
gives a bijection between all canonical self-adjoint extensions $\widetilde{A}$ of $S_1 \oplus S_2$ in $\mathfrak{H}_1 \oplus \mathfrak{H}_2$ with $\widetilde{A} \cap \mathfrak{H}_k^2 = S_k$, $k \in \{1,2\}$, and all boundary mappings $\mathsf{b}_2: S^*_2 \rightarrow \mathbb{C}^{\delta}$ for $S_2$ with Gram matrix $-\mathsf{Q}_1$.
\end{enumerate}
\end{theorem}

\begin{proof}
Let $\mathsf{b}_2: S^*_2 \rightarrow \mathbb{C}^{\delta}$ be a boundary mapping for $S_2$ with Gram matrix $-\mathsf{Q}_1$. By (\ref{Coupling-i3}) formula \eqref{eq-PswACDR} establishes a bijection between all canonical self-adjoint extensions $\widetilde{A}$ of $S_1 \oplus S_2$ in $\mathfrak{H}_1 \oplus \mathfrak{H}_2$ with $\widetilde{A} \cap \mathfrak{H}_k^2=S_k$, $k \in \{1,2\}$, and all $\delta \times \delta$ invertible matrices $\Gamma$ such that $\Gamma^*\mathsf{Q}_1\Gamma = \mathsf{Q}_1$. By \cite[Lemma~2.1(ii),(iii)]{CDCAOT}, the formula \(\mathsf b=\Gamma\mathsf b_2\) provides a bijection between all boundary mappings for $S_2$ with Gram matrix $-\mathsf Q_1$ and all invertible matrices $\Gamma$ satisfying \(\Gamma^*\mathsf Q_1\Gamma=\mathsf Q_1\).
\end{proof}

\begin{definition} \label{def-Cmu}
For $d \in \mathbb{N}$ and a $d$-tuple $\mu=(\mu_1, \ldots, \mu_d) \in \mathbb{N}^d$ with $\mu_1\geq \cdots \geq \mu_d\geq 1$ we define, as in \cite{CDLAA12}, the \emph{canonical subspace $\mathfrak{C}_{\mu}$ of $\mathbb{C}^{d} [z]$}  by
\begin{equation*}
\mathfrak{C}_\mu = \Bigl\{ \bigl[p_1(z) \cdots p_{d}(z)\bigr]^\top \in \mathbb{C}^{d}\mkern-1mu[z] \,:\, \deg p_j(z)< \mu_j, j \in \{1,\ldots,d\} \Bigr\}
\end{equation*}
and denote its elements by $f_\mu, g_\mu$, \(0_\mu\), etc.  By  $S_\mu$ we denote the operator in $\mathfrak{C}_\mu$ of multiplication by $z$. Its graph is given by
\begin{equation*}  
S_\mu = \Bigl\{ \bigl\{f_\mu, g_\mu \bigr\} \in \mathfrak{C}_{\mu}^2\,:\, zf_\mu(z)-g_\mu(z)\equiv 0 \Bigr\}.
\end{equation*}

\end{definition}

The facts that $\mathfrak{C}_\mu$ is finite-dimensional and $S_\mu$ is an operator without  eigenvalues play a key role in this paper; see  Theorems~\ref{th-EqSt} and~\ref{X}.

Below we recall \cite[Theorem~2.1]{CDLAA23}. It shows that a canonical subspace $\mathfrak{C}_{\mu}$ arises as the reproducing kernel space related to a certain matrix polynomial and, together with $S_\mu$, serves as a model for any finite-dimensional Pontryagin space with a symmetric operator that has no eigenvalues;  see Theorem~\ref{t-poso}. Furthermore, this kind of matrix polynomials appears in the characterization of the Shtraus family of extensions of a closed symmetric linear relation in a Krein space; see  Theorem~\ref{th-SfMt} and  Proposition~\ref{pr-bbAbasics}(\ref{pr-bbAbasics-i3})(\ref{pr-bbAbasics-i3-i2}).

\begin{theorem} \label{t-fCrks}
Let $d \in \mathbb{N}$. Let $\mathsf{Q}$ be a self-adjoint $2d \times 2d$ matrix with $d$ positive and $d$ negative eigenvalues. Let $\mathcal{P}(z)$ be a $d \times 2d$ matrix polynomial, and for $j \in \{1, \ldots, d\}$, let $\mu_j$ denote the degree of the $j$-th row of $\mathcal{P}(z)$. Assume that $\mathcal{P}(z) \in \mathbb{P}_{\mathsf{Q}}$ with \(\mu_1 \geq \cdots \geq \mu_d \geq 1\) and set \(\mu = (\mu_1,\ldots,\mu_d)\).
Then
\begin{enumerate}
\renewcommand*\theenumi{\roman{enumi}}
\renewcommand*\labelenumi{\rm{(\theenumi)}}
\item \label{i-t-m-c1}  
The reproducing kernel Pontryagin space $\mathfrak{K}_{\mathsf{Q}, \mathcal{P}}$ with reproducing kernel defined for \(z, w \in \mathbb{C}\) by
\begin{equation} \label{eq-K_P}
K_{\mathsf{Q}, \mathcal{P}}(z,w)
=
\begin{cases}
\dfrac{\iu}{z-w^*} \mathcal{P}(z) \mathsf{Q}^{-1} \mathcal{P}(w)^*  &
\text{for} \quad  w \neq z^*,  \\[9pt]
\iu\mkern 2mu \mathcal{P}'(z) \mathsf{Q}^{-1} \mathcal{P}(z^*)^*  &
\text{for} \quad w = z^*,
\end{cases}
\end{equation}
is the canonical subspace $\mathfrak{C}_{\mu}$ of $\mathbb{C}^d\mkern-1mu[z]$.

\item \label{i-t-m-c2}  
The operator $S_\mu$ is symmetric in the Pontryagin space $\mathfrak{K}_{\mathsf{Q}, \mathcal{P}}$, its defect numbers are equal to $d$ and its adjoint is given by
\begin{equation} \label{Smustar}
S_\mu^*=\Bigl\{
\{f_\mu, g_\mu\} \in \mathfrak{C}_\mu^2 :
zf_\mu(z)-g_\mu(z) \equiv \mathcal{P}(z)c
\mkern10mu \text{for some} \mkern10mu c \in \mathbb{C}^{2d}  \Bigr\}.
\end{equation}
\item \label{i-t-m-c3}  
The linear relation
\begin{equation*}
\Big\{ \big\{\{f_\mu,g_\mu\},c \big\} \in \mathfrak{C}_\mu^2 \times \mathbb{C}^{2d} \,:\, zf_\mu(z) -g_\mu(z)\equiv \mathcal{P}(z)c\Big\}
 \end{equation*}
is (the graph of) an operator $\mathsf{b}_{\mu,\mathcal{P}}: S_\mu^* \rightarrow \mathbb{C}^{2d}$  which is a boundary mapping for $S_\mu$ with Gram matrix $-\mathsf{Q}$.
\end{enumerate}
\end{theorem}

The inner product on $\mathfrak{C}_\mu$, with respect to which $S_\mu$ is symmetric, is determined by the kernel \eqref{eq-K_P}. Therefore we will often denote the reproducing kernel space $\mathfrak{K}_{\mathsf{Q}, \mathcal{P}}$ by the pair $\bigl(\mathfrak{C}_\mu, K_{\mathsf{Q}, \mathcal{P}}\bigr)$.

Note that in \cite[Theorem~2.1(iii)]{CDLAA23} there is a mistake: the Gram matrix $\mathsf{Q}$ there should be replaced by $-\mathsf{Q}$, as in Theorem~{\rm~\ref{t-fCrks}}(\ref{i-t-m-c3}).

By \cite[Theorem~3.2]{CDLAA23} the polynomial $\mathcal{P}(z)$ in Theorem~\ref{t-fCrks} also has the trivial common null space property:
\begin{equation*}    
\bigcap_{z \in \mathbb{C}} \nul\mkern-1mu\mathcal{P}(z)=\{0\}.
\end{equation*}
Equivalently, the columns of $\mathcal{P}(z)$ as vector polynomials over \(\mathbb{C}\) are linearly independent.

The following two lemmas are \cite[Lemma~2.6, Lemma~2.7]{CDCAOT}.

\begin{lemma} \label{le-reeqv}
Let \(\mathcal{P}(z)\), \(\mu\), \(\mathfrak{C}_\mu\),  \(S_\mu\), \(S_\mu^*\), and \(\mathsf{b}_{\mu,\mathcal{P}}\) be as in  Theorem~{\rm\ref{t-fCrks}}, let $c \in \mathbb{C}^{2d}$ and $\lambda \in \mathbb{C}$. The following equivalences hold.
\begin{enumerate}
\renewcommand*\theenumi{\roman{enumi}}
\renewcommand*\labelenumi{\rm{(\theenumi)}}
\item \label{le-reeqv-i2}
$\mathcal{P}(\lambda) c = 0 \quad\Leftrightarrow \quad \mathcal{P}(z) c \equiv (z- \lambda) f_\mu(z)$ for some $f_\mu \in \mathfrak{C}_\mu$.

\item \label{le-reeqv-i1}
$\mkern 12mu \mathsf{P}_\infty c = 0 \quad \Leftrightarrow \quad  \mathcal{P}(z) c \in \mathfrak{C}_\mu$.

\end{enumerate}

Moreover,
\begin{equation*}
\nul \mathcal{P}(\lambda) =
\mathsf{b}_{\mu,\mathcal{P}}\mkern-1.5mu \bigl(S_\mu^* \cap \lambda I_{\mathfrak{C}_\mu} \bigr)
\quad \text{for all} \quad \lambda \in \overline{\mathbb{C}},
\end{equation*}
where \(\infty I_{\mathfrak{C}_\mu} =  \{0_\mu\} \times \mathfrak{C}_\mu\).
\end{lemma}

\begin{lemma} \label{le-eqnuls}
Let \(d \in \mathbb{N}\) and let \(\mathcal{P}(z)\) and \(\mathcal{Q}(z)\) be \(d\times 2d\) matrix polynomials of rank \(d\) for all \(z\in \mathbb{C}\). Then
\begin{equation*}  
\nul \mathcal{P}(z)\  = \nul  \mathcal{Q}(z) \quad \text{for all} \quad z\in \mathbb{C}.
\end{equation*}
if and only if there exists a unimodular \(d\times d\) matrix polynomial \(\mathcal{W}(z)\) such that
\begin{equation*}  
\mathcal{W}(z) \mathcal{P}(z) \equiv \mathcal{Q}(z).
\end{equation*}
\end{lemma}

The following result, which we call the \emph{Model Theorem}, states that the quadruple \(\bigl(\mathfrak{C}_\mu, K_{\mathsf{Q},\mathcal{P}}, S_\mu, \mathsf{b}_{\mu,\mathcal{P}}\bigr)\) is a model for \(\bigl(\mathfrak{G}, \kip_{\mathfrak{G}}, S, \mathsf{b}\bigr)\) under the assumptions analogous to  (\ref{def-bbA-i3}) and (\ref{def-bbA-i4}) in Definition~\ref{def-bbA}: \(\dim \mathfrak{G}\in\mathbb{N}\) and \(\sigp(S)=\emptyset\).

\begin{theorem} \label{t-poso}
Let $(\mathfrak{G},\kip_{\mathfrak{G}})$ be a nonzero finite-dimensional Pontryagin space and let $S$ be a symmetric operator in $\mathfrak{G}$ without eigenvalues. Then the defect numbers of $S$  coincide and are equal to $d = \codim(\dom S)$. Moreover, $S$ is nilpotent. Let $\mathsf{b} : S^* \rightarrow \mathbb{C}^{2d}$  be a boundary map for $S$ with Gram matrix $-\mathsf{Q}$.

Let \( m \) denote the nilpotency index of \( S \), and define
\begin{equation*}
\delta_j =\dim \bigl(\dom  S^{j-1}\bigr), \quad j \in \{1,\ldots, m+1\},\quad (\text{so that}\ \delta_1=\dim \mathfrak{G}, \ \delta_{m+1}=0).
\end{equation*}
Let $\mu= (\mu_1,\ldots, \mu_d)$ be the $d$-tuple with entries
\begin{equation*}  
\mu_j = \#\big\{i \in \{1, \ldots, m\}:\delta_i -\delta_{i+1}\geq j\big\}, \quad j\in \{1, \ldots, d\},
\end{equation*}
so that $\mu_1 \geq \cdots \geq \mu_d \geq 1$.

Then there exists a $d \times 2d$ matrix polynomial $\mathcal{P}(z)$ in \(\mathbb{P}_{\mathsf{Q}}\) such that \(\mu_j = \deg \mathcal{P}(z)\bigl.\bigr|_j\) for all \(j \in \{1,\dots,d\}\), and there exists an isomorphism
\begin{equation*}
\Phi: \bigl(\mathfrak{G},\kip_{\mathfrak{G}}\bigr)
\rightarrow \bigl( \mathfrak{C}_\mu, K_{\mathsf{Q}, \mathcal{P}} \bigr)
\end{equation*}
such that:
\begin{enumerate}
\renewcommand*\theenumi{\alph{enumi}}
\renewcommand*\labelenumi{\rm{(\theenumi)}}
\setlength{\itemsep}{1pt}
\item  \label{t-poso-i2}
$\Phi S = S_{\mu}\Phi$, and
\item  \label{t-poso-i3}
\(\{\Phi f, \Phi g\} \in S_{\mu}^*\) and  \(\mathsf{b}(\{f,g\}) = \mathsf{b}_{\mu,\mathcal{P}}\mkern-1mu (\{\Phi f, \Phi g\})\) for all \(\{f,g\}\in S^*\).
\end{enumerate}

The matrix polynomial \(\mathcal{P}(z) \) is unique up to multiplication on the left by a unimodular $d \times d$ matrix polynomial $\mathcal{W}(z) = \bigl[w_{jk}(z)\bigr]_{j,k=1}^d$ satisfying \eqref{coeff}.
\end{theorem}

The Model Theorem implies that the positive (negative) index of $\mathfrak{G}$ is equal to the number of positive (negative) squares of the kernel
\begin{equation*}  
K_{\mathsf{Q}, \mathcal{P}}(z,w)
=
\begin{cases}
\iu\mkern 2mu \dfrac{\mathcal{P}(z) \mathsf {Q}^{-1} \mathcal{P}(w)^*}{z-w^*}
\quad & \text{for} \quad  w \neq z^*,  \\[9pt]
\iu\mkern 2mu \mathcal{P}'(z) \mathsf{Q}^{-1} \mathcal{P}(z^*)^* \quad & \text{for} \quad w = z^*,
\end{cases}
\end{equation*}
$\dim \mathfrak{G} = \sum_{j=1}^d \mu_j$ and the nilpotency index of \(S\) and \(S_\mu\) equal \(m = \mu_1 = \deg\mathcal{P}(z)\).

The Model Theorem and Lemma~\ref{le-reeqv} yield
\begin{equation*} 
\nul \mathcal{P}(\lambda) = \mathsf{b}\bigl(S^* \cap \lambda I_{\mathfrak{G}} \bigr)
\quad \text{for all} \quad \lambda \in \overline{\mathbb{C}},
\end{equation*}
where \(\infty I_{\mathfrak{G}} =  \{0\} \times \mathfrak{G}\).

The next theorem is an improvement of \cite[Theorem~4.4]{CDCAOT} adapted to the new item~(\ref{Coupling-i4}) in the Coupling Theorem~{\rm\ref{Coupling}}.

\begin{theorem} \label{th-SfMt}
In the setting of Theorem~{\rm\ref{Coupling}(\ref{Coupling-i4})}, additionally assume that \(\mathfrak{H}_2\) is finite dimensional and $S_2$ is a symmetric operator without eigenvalues. Then the defect numbers of $S_1$ and $S_2$ are all equal to $d$, so that $d = d_k^- = d_k^+$, \(k\in\{1,2\}\). The  \(2d\times 2d\) Gram matrix \(\mathsf{Q}_1\) has \(d\) positive and \(d\) negative eigenvalues. Let \(\mathcal{P}(z) \in \mathbb{P}_{\mathsf{Q}_1}\) be a \(d\times 2d\) matrix polynomial and assume that $\mu_j = \deg \mathcal{P}(z)\mkern-1mu\bigl.\bigr|_j \in \mathbb{N}$ for $j \in \{1, \ldots,d\}$. Set \(\mu = \bigl(\mu_1,\ldots,\mu_d\bigr)\), and let $\widetilde{A}$ be defined in \eqref{eq-PswACDRGisI}. Then the family of Shtraus subspaces associated with $\widetilde{A}$ is given by the formula
\begin{equation}\label{eq-SP}
\TtA(z) =
\Bigl\{\{f,g\}\in S_1^*\,:\,\mathcal{P}(z)\mathsf b_1(\{f,g\})=0\Bigr\}
\quad \text{for all} \quad z \in \mathbb{C} 
\end{equation}
if and only if the quadruple \(\bigl(\mathfrak{C}_\mu, K_{\mathsf{Q}_1,\mathcal{P}}, S_\mu, \mathsf{b}_{\mu,\mathcal{P}}\bigr)\) is a model for the quadruple  \(\bigl(\mathfrak{H}_2, \kip_{\mathfrak{H}_2}, S_2, \mathsf{b}_2\bigr)\). Moreover, if \eqref{eq-SP} holds, then 
\[
\TtA(\infty) = \Bigl\{ \{f,g\}\in S_1^*\,:\,\mathsf{P}_\infty\mathsf b_1(\{f,g\})=0 \Bigr\}.
\]
\end{theorem}

\begin{proof}
The first claim about the defect numbers of \(S_1\) and \(S_2\) and the ``if'' part of the theorem follows from \cite[Theorem~4.4]{CDCAOT} with the boundary mapping \(\mathsf{b}_2\) chosen to have Gram matrix \(\mathsf{Q}_2 = -\mathsf{Q}_1\).

To prove the ``only if'' part, let \(\mathcal{S}(z) \in \mathbb{P}_{\mathsf{Q}_1}\) and assume that
\begin{equation}\label{eq-SPS}
\TtA(z) =
\Bigl\{\{f,g\}\in S_1^*\,:\, \mathcal{S}(z) \mkern2.5mu \mathsf{b}_1\mkern-2mu(\{f,g\}) = 0\Bigr\}
\quad \text{for all} \quad z \in \mathbb{C}.
\end{equation}
Let \(\mathcal{P}(z) \in \mathbb{P}_{\mathsf{Q}_1}\) be a \(d\times 2d\) matrix polynomial such that $\mu_j=\deg \mathcal{P}(z)\mkern-1mu\bigl.\bigr|_j \in \mathbb{N}$ for $j \in \{1, \ldots,d\}$. By \cite[Theorem~4.4]{CDCAOT} and (\ref{eq-PswACDRGisI}), \eqref{eq-SP} also holds. As \(\mathcal{S}(z), \mathcal{P}(z) \in \mathbb{P}_{\mathsf{Q}_1}\), applying Lemma~\ref{le-PeSSts} to \eqref{eq-SPS} and \eqref{eq-SP} yields that \(\mathcal{S}(z) \sim \mathcal{P}(z)\). That is, there exists a unimodular \(d\times d\) matrix polynomial \(\mathcal{W}(z)\) such that \(\mathcal{W}(z) \mathcal{S}(z) \equiv \mathcal{P}(z)\). It follows from Lemma~\ref{4statements}(\ref{4statements-i2}) and (\ref{4statements-i3}) that \(\deg\mathcal{S}(z)\bigl.\bigr|_j = \mu_j\) for all \(j\in\{1,\ldots,d\}\) and for the entry $w_{jk}(z)$ in the \(j\)th row and the \(k\)th column of \(\mathcal{W}(z)\) we have
\begin{equation*}
\deg w_{jk}(z) \leq \mu_j - \mu_k  \quad \text{for all} \quad j,k \in \{1, \ldots, d\}.
\end{equation*}
The uniqueness part of the Model Theorem~\ref{t-poso} implies that  \(\bigl(\mathfrak{C}_{\mu}, K_{\mathsf{Q}_1,\mathcal{S}}, S_{\mu}, \mathsf{b}_{\mu,\mathcal{S}}\bigr)\) is a model for the quadruple \(\bigl(\mathfrak{H}_2, \kip_{\mathfrak{H}_2}, S_2, \mathsf{b}_2\bigr)\).

The claim about \(\TtA(\infty)\) follows from \cite[Theorem~4.4]{CDCAOT}.
\end{proof}

\section{Proof of Theorem~\ref{modelP}} \label{proof}

The following lemma is a variant of \cite[Lemma~2.5]{CDLAA12}.
\begin{lemma} \label{le-scks}
Let $n \in\mathbb{N}$ and $(\mathfrak{L},\kip)$ be a finite-dimensional inner  product space such that $\dim \mathfrak{L} \leq 2n$. If $\mathfrak{N}$ is a neutral subspace of $(\mathfrak{L},\kip)$, $\dim\mathfrak{N}=n$, and $\mathfrak{N}\cap\mathfrak{L}^{[\perp]} = \{0\}$, then $\dim \mathfrak{L} = 2n$, $(\mathfrak{L},\kip)$ is a Krein space  with positive and negative index equal to $n$.
\end{lemma}
\begin{proof}
Let $\mathfrak{L}^{\circ} = \mathfrak{L}\cap\mathfrak{L}^{[\perp]}$ be the isotropic part of $\mathfrak{L}$. Then $\kip$ is well defined on the factor space $\mathfrak{L}/\mathfrak{L}^{\circ}$ and $\bigl(\mathfrak{L}/\mathfrak{L}^{\circ},\kip\bigr)$ is a Krein space. Denote by $n_-$ the negative and by $n_+$ the positive index of the Krein space  $\bigl(\mathfrak{L}/\mathfrak{L}^{\circ},\kip\bigr)$. The factor subspace $\mathfrak{N}/\mathfrak{L}^{\circ}$ is a neutral subspace of this factor Krein space and $\dim \bigl(\mathfrak{N}/\mathfrak{L}^{\circ}\bigr) = n$. Since the restriction of a fundamental projection on a neutral subspace in a Krein space is injective, we have $n \leq n_-$ and $n \leq n_+$. Therefore
\[
2n \leq n_- + n_+ = \dim \bigl(\mathfrak{L}/\mathfrak{L}^{\circ}\bigr) \leq \dim \mathfrak{L} \leq 2n.
\]
This yields $\dim \mathfrak{L} = 2n$, $\mathfrak{L}^{\circ} = \{0\}$, hence $(\mathfrak{L},\kip)$ is a Krein space. Since $n \leq n_-$ and $n \leq n_+$,
\[
2 n \leq 2 \min\{n_-,n_+\} \leq n_- + n_+ = 2 n.
\]
Consequently, $n = n_- = n_+$.
\end{proof}

\begin{proof}[Proof of Theorem {\rm \ref{modelP}}] \emph{Step~1.} We derive a formula for $\mathcal W(z)$.
Let $\mathcal{S}(z)$ be the $d_0\times 2d$ matrix polynomial of degree $p$ such that with $\mathsf{B}$ introduced in the theorem we have
\[
\mathcal{P}(z) = \begin{bmatrix}  \mathcal S(z) \\ \mathsf{B}  \end{bmatrix}.
\]
From the relationships  between the coefficients of the polynomials $\mathcal{P}(z)$ and $\mathcal S(z)$:
\begin{equation}\label{eq-CPCS}
\mathsf{P}_{\infty} = \begin{bmatrix}  \mathsf{S}_{\infty} \\ \mathsf{B} \end{bmatrix}, \quad \mathsf{P}_0 = \begin{bmatrix}  \mathsf{S}_0 \\ \mathsf{B} \end{bmatrix}, \quad \mathsf{P}_k = \begin{bmatrix}  \mathsf{S}_k \\ \mathsf{0} \end{bmatrix}, \quad k \in \{1,\ldots,p\},
\end{equation}
and the definitions of $\mathsf{C}_{\mathcal{P}}$  and $\mathsf{C}_{\mathcal{S}}$ we obtain
\[
\row \begin{bmatrix} \mathsf{S}_{\infty} \\ \mathsf{B} \end{bmatrix}=\row \mathsf{P}_{\infty} \subset \row \mathsf{C}_{\mathcal{P}}=\row \begin{bmatrix}  \mathsf{C}_{\mathcal S} \\ \mathsf{B} \end{bmatrix}.
\]
The space on the left has dimension $\rank \mathsf{P}_\infty=d$ and the space on the right has dimension $\rank \mathsf{C}_{\mathcal{P}}=d+d_0$. The latter follows from the Rank-Nullity Theorem and by applying \eqref{eq-kTzkT} and \cite[Theorem~3.2]{CDLAA23} which yield
\begin{equation}\label{dimofkers}
\dim \nul \mathsf{C}_{\mathcal{P}} = \dim \Bigl(\bigcap_{z\in \mathbb C}\nul \mathcal{P}(z)\Bigr) = d-d_0.
\end{equation}
The $d$ linearly independent rows of $\begin{bmatrix} \mathsf{S}_{\infty} \\ \mathsf{B} \end{bmatrix}$ can be extended to a basis
of the space $\row \begin{bmatrix}  \mathsf{C}_{\mathcal S} \\ \mathsf{B} \end{bmatrix}$ by selecting a suitable $d_0\times 2d$ submatrix $\mathsf{S}_b$ of $\mathsf{C}_{\mathcal{S}}$ such that the rows of the $(d+d_0) \times 2d$ matrix
\begin{equation*}\label{eq-3mb}
 \begin{bmatrix}  \mathsf{S}_{\infty} \\ \mathsf{S}_b \\ \mathsf{B} \end{bmatrix}
\end{equation*}
are linearly independent.
By \eqref{eq-CPCS} , for every $k \in\{0,\ldots,p\}$ we have
\[
\row \mathsf{S}_k \subset \row \mathsf{C}_{\mathcal{S}} \subset \row \mathsf{C}_{\mathcal{P}} = \row  \begin{bmatrix}  \mathsf{S}_{\infty} \\ \mathsf{S}_b \\ \mathsf{B} \end{bmatrix},
\]
hence for every $k \in\{0,\ldots,p\}$ there exists a unique $d_0\times(d-d_0)$ matrix $\mathsf{A}_k$ such that the rows of the $d_0 \times 2d$ matrix
\begin{equation}\label{eq-dTk}
\mathsf{T}_k = \mathsf{S}_k - \mathsf{A}_k \mathsf{B}
\end{equation}
belong to $\row \begin{bmatrix}  \mathsf{S}_{\infty} \\ \mathsf{S}_b \end{bmatrix}$. For all $k \in \{1,\ldots,d_0\}$ and all $l \in \mathbb{N}$ such that $\mu_k \lt l \leq p$ we have
$\bigl.\mathsf{S}_{\mu_k}\mkern-2mu\bigr|_{k} = \bigl.\mathsf{S}_{\infty}\mkern-2mu\bigr|_{k}$ and $\bigl.\mathsf{S}_{l}\mkern-1mu\bigr|_{k} = \mathsf{0},$
hence for the same values of $k$ and $l$
\begin{equation}\label{rowA}
\bigl.\mathsf{A}_{\mu_k}\mkern-2mu\bigr|_{k} = \mathsf{0} \quad \text{and} \quad
\bigl.\mathsf{A}_{l}\mkern-1mu\bigr|_{k} = \mathsf{0}.
\end{equation}
We define the $d_0\times 2d$ matrix polynomial $\mathcal T(z)$, the $d_0\times (d-d_0)$ matrix polynomial $\mathcal A(z)$ and the unimodular $d \times d$ matrix polynomial $\mathcal W(z)$ by
\[
\mathcal T(z)=\sum_{k=1}^p \mathsf T_k z^k, \quad \mathcal A(z)=-\sum_{k=1}^p \mathsf A_k z^k \quad \text{and} \quad \mathcal W(z) =  \begin{bmatrix}
\mathsf{I}_{d_0} & \mathcal A(z) \\[6pt]
\mathsf{0} & \mathsf{I}_{d-d_0}  \end{bmatrix}.
\]
It follows from \eqref{eq-dTk} that
\begin{equation}\label{equalities}
\mathcal{T}(z) = \mathcal{S}(z) + \mathcal{A}(z)\mathsf{B}
\quad \text{and} \quad
\mathcal{W}(z)\mathcal{P}(z) = \mathcal{W}(z)\begin{bmatrix} \mathcal{S}(z) \\ \mathsf{B} \end{bmatrix}=
\begin{bmatrix} \mathcal T(z) \\ \mathsf B \end{bmatrix}.
\end{equation}
The relations \eqref{rowA} imply that $\deg \mathcal{A}(z)\mkern-1mu\bigl.\bigr|_k < \mu_k$, $k \in \{1, \ldots, d_0\}$. Hence $\mathcal W(z)$ satisfies
\eqref{coeff}, and by applying \eqref{Sinfinity} to both sides of the first equality in \eqref{equalities} we see that
 $\mathsf{T}_{\infty} = \mathsf{S}_{\infty}$,
hence
\begin{equation}\label{ranks}
\rank \mathsf T_\infty = \rank \mathsf S_\infty=d_0.
\end{equation}
The second equality in \eqref{equalities} and assumption (\ref{i-t-m-a}) in Definition~\ref{def-bbP} imply that for all $z \in \mathbb{C}$ we have
\begin{equation}\label{eq-cTQiMs}
\mathcal T(z)\mathsf{Q}^{-1}\mathcal T(z^*)^*  = \mathsf{0}
\end{equation}
and $
\mathcal T(z)\mathsf{Q}^{-1}\mathsf{B}^* =  \mathsf{0}$ which implies that
\begin{equation}\label{eq-cTQiMs2}
\ran\mathsf{B}^* \subseteq \mathsf{Q} \Bigl(\bigcap_{z \in \mathbb C} \nul \mathcal T(z) \Bigr).
\end{equation}

\noindent\emph{Step~2.}  We prove \eqref{reprP} by showing that $\mathcal{T}(z)=\widehat{\mathcal{P}}(z) \mathsf{B}_0$. Since the matrix polynomial $\mathcal W(z)$ is unimodular and by \eqref{eq-kTzkT}, we have
\begin{equation}\label{fourkers} \nul \mathsf C_{\mathcal W \mathcal{P}}=
\bigcap_{z\in \mathbb C}\nul \bigl(\mathcal W(z) \mathcal{P}(z) \bigr) = \bigcap_{z\in \mathbb C}\nul \mathcal{P}(z)=\nul \mathsf C_{\mathcal{P}},
\end{equation}
hence by \eqref{dimofkers} and the Rank-Nullity Theorem, $\rank \mathsf{C}_{\mathcal W \mathcal{P}} = d+d_0$. The second equality in  \eqref{equalities}, the equality  $\rank \mathsf B=d-d_0$ and basic properties of the rank of a matrix imply
\begin{multline*}
d + d_0 = \rank \mathsf{C}_{\mathcal W\mathcal{P}}
  = \rank\! \begin{bmatrix}  \mathsf{C}_{\mathcal T} \\[2pt] \mathsf{B} \end{bmatrix}
=\rank\left(\begin{bmatrix}  \mathsf{C}_{\mathcal T} \\[2pt] \mathsf{0} \end{bmatrix}+ \begin{bmatrix}  \mathsf{0} \\[2pt] \mathsf{B} \end{bmatrix}\right) \\ \leq \rank \mathsf{C}_{\mathcal T} + \rank \mathsf{B} = \rank \mathsf{C}_{\mathcal T} + d-d_0,
\end{multline*}
and consequently, $\rank \mathsf{C}_{\mathcal T} \geq 2d_0$.
The definition \eqref{eq-dTk} of the matrix coefficients $\mathsf{T}_k$, $k\in \{1,\ldots,p\}$, implies that
$
\row \mathsf{C}_{\mathcal T} \subseteq
\row \begin{bmatrix}  \mathsf{S}_{\infty} \\[4pt] \mathsf{S}_b \end{bmatrix}$,
hence
$
\rank \mathsf{C}_{\mathcal T} \leq 2d_0$, and therefore
\begin{equation}\label{eq-rTeq}
\rank \mathsf{C}_{\mathcal T} = 2d_0.
\end{equation}

Now consider the subspace $\mathfrak{K}_{\mathcal T}$ of $\mathbb C^{2d}$ defined by
\begin{equation}\label{eq-defKT}
\mathfrak{K}_{\mathcal T}  = \bigcap_{z\in \mathbb C}\nul \mathcal{T}(z) 
= \nul \mathsf{C}_{\mathcal T}
\end{equation}
and denote by $(\mathfrak{K}_{\mathcal T})^{\perp}$ its orthogonal complement in the Euclidean space $\mathbb{C}^{2d}$.
On account of the Rank-Nullity Theorem and \eqref{eq-rTeq}
\begin{equation}\label{eq-dimKT}
\dim \mathfrak{K}_{\mathcal T} =  2(d-d_0),
\end{equation}
hence
\[
\dim\, (\mathfrak{K}_{\mathcal T})^{\perp} = 2 d_0.
\]

Let $\mathsf B_0$ be a $2d_0\times 2d$ matrix such that the $2d_0$  columns of its adjoint $\mathsf{B}_0^*$ form a basis for $(\mathfrak{K}_{\mathcal{T}})^{\perp}$. Taking the orthogonal complement in \eqref{eq-defKT}
yields
\begin{equation}\label{eq-K1peP1}
\ran \mathsf{B}_0^*  = (\mathfrak{K}_{\mathcal T})^{\perp} =
\sum_{z\in \mathbb{C}} \ran \mathcal{T}(z)^*.
\end{equation}
Notice that the $2d\times 2d$ matrix $\mathsf{B}_0^* \bigl(\mathsf{B}_0\mathsf{B}_0^*\bigr)^{-1}\mathsf{B}_0$ is the orthogonal projection in the Euclidean space $\mathbb{C}^{2d}$ onto the subspace $(\mathfrak{K}_{\mathcal T})^{\perp}$. By \eqref{eq-K1peP1}, for all $z \in \mathbb{C}$
\begin{equation} \label{eq-prT}
\mathsf{B}_0^* \bigl(\mathsf{B}_0\mathsf{B}_0^*\bigr)^{-1}\mathsf{B}_0 \mathcal{T}(z)^* = \mathcal T(z)^*.
\end{equation}
Define the $d_0 \times 2d_0$ matrix polynomial $\widehat{\mathcal{P}}(z)$ by
\[
\widehat{\mathcal{P}}(z) = \mathcal T(z) \mathsf{B}_0^* \bigl(\mathsf{B}_0\mathsf{B}_0^*\bigr)^{-1},
\]
then, taking the adjoint in \eqref{eq-prT}, we obtain
\begin{equation} \label{eq-P0T}
\widehat{\mathcal{P}}(z) \mathsf{B}_0 = \mathcal T(z) \quad \text{for all}\ z \in \mathbb{C}.
\end{equation}

\noindent\emph{Step~3.}
Consider the Krein space $\bigl( \mathbb{C}^{2d}, \kip_{\mathsf{Q}^{-1}} \bigr)$ with inner product
\[
[u,v]_{\mathsf{Q}^{-1}} = v^* \mathsf{Q}^{-1} u \quad \text{for all} \quad u, v \in \mathbb{C}^{2d}
\]
and positive and negative index equal to $d$. We show that $(\mathfrak{K}_{\mathcal T})^{\perp}$
and $\mathsf{Q}\mathfrak{K}_{\mathcal T}$ are complementary regular subspaces in this space. For this we apply  Lemma~\ref{le-scks} to the inner product space
\[
(\mathfrak{L}, \kip ) = \bigl(\mathsf{Q} \mathfrak{K}_{\mathcal T}, \kip_{\mathsf{Q}^{-1}} \bigr).
\]
By \eqref {eq-dimKT} $\dim \mathfrak L=2(d-d_0)$, and by the definition of $\mathsf B$ and assumption (\ref{i-t-m-a}) in Definition~\ref{def-bbP} we have that $\mathfrak{N} = \ran
\mathsf{B}^*$ is a neutral subspace of $(\mathfrak{L}, \kip )$ and $\dim \mathfrak{N} = d-d_0$. To prove the remaining condition $\mathfrak{N}\cap\mathfrak{L}^{[\perp]} = \{0\}$ in Lemma~\ref{le-scks} we use that $\mathfrak L^{[\perp]}\subseteq \mathfrak L^{[\perp]_{\mathsf Q^{-1}}}=(\mathfrak{K}_{\mathcal T})^{\perp}$ and prove the more general equality
\begin{equation}\label{eq-M*nip}
\bigl(\ran \mathsf{B}^* \bigr) \cap   (\mathfrak{K}_{\mathcal T})^{\perp} = \{0\}.
\end{equation}
Equivalently,
\begin{equation} \label{eq-M*nip1}
(\nul \mathsf{B} )  + \mathfrak{K}_{\mathcal T} = \mathbb{C}^{2d}.
\end{equation}
Notice that by \eqref{fourkers} and \eqref{dimofkers}
\[
\dim \bigl((\nul \mathsf{B}) \cap \mathfrak{K}_{\mathcal T}\bigr) = \dim \nul \begin{bmatrix} \mathsf C_{\mathcal T} \\ \mathsf B \end{bmatrix}= \dim \nul \mathsf{C}_{\mathcal W \mathcal{P}}=\dim \nul \mathsf{C}_{\mathcal{P}}=d-d_0.
\]
This equality, \eqref{eq-dimKT} and the fact that $(d-d_0)\times 2d$ matrix $\mathsf{B}$ has full rank $d-d_0$, imply
\begin{align*}
  \dim\bigl( (\nul \mathsf{B})  + \mathfrak{K}_{\mathcal T} \bigr)
  & = \dim  \nul \mathsf{B}  + \dim \mathfrak{K}_{\mathcal T} - \dim\bigl((\nul \mathsf{B}) \cap \mathfrak{K}_{\mathcal T}\bigr)  \\
 &= d+d_0 + 2(d-d_0) - (d-d_0)  \\
 & = 2d.
\end{align*}
Thus \eqref{eq-M*nip1} holds, and hence \eqref{eq-M*nip} holds as well.
Lemma~\ref{le-scks}
yields that $\mathsf{Q} \mathfrak{K}_{\mathcal T}$ is a Krein subspace of the Krein space $\bigl( \mathbb{C}^{2d}, \kip_{{\mathsf Q}^{-1}} \bigr)$ with positive and negative index equal to $d-d_0$. Consequently, its orthogonal complement $(\mathsf Q \mathfrak{K}_{\mathcal T})^{[\perp]}=(\mathfrak{K}_{\mathcal T})^{\perp}$ is a Krein subspace of $\bigl( \mathbb{C}^{2d}, \kip_{\mathsf{Q}^{-1}} \bigr)$ with positive and negative index equal to $d_0$.

\smallskip

\noindent{\em Step~4.} Finally we prove claims (\ref{abcd-c1})-(\ref{abcd-c7}).

Claims (\ref{abcd-c1}) and (\ref{abcd-c2}) follow from the first equality in \eqref{eq-K1peP1}, the inclusion $\ran \mathsf{B}^* \subset \mathsf Q \mathfrak K_{\mathcal T}$ from \eqref{eq-cTQiMs2} and the fact that the subspaces
$(\mathfrak{K}_{\mathcal T})^{\perp}$ and $\mathsf{Q}\mathfrak{K}_{\mathcal T}$ are complementary regular subspaces of the Krein space $\bigl( \mathbb{C}^{2d}, \kip_{\mathsf{Q}^{-1}} \bigr)$ proved in \emph{Step} 3.

Claim (\ref{abcd-c3}) follows from the fact that the columns of $\mathsf{B}_0^*$ form a basis of the Krein subspace $(\mathfrak{K}_{\mathcal T})^{\perp}$ of the Krein space $\bigl( \mathbb{C}^{2d}, \kip_{\mathsf{Q}^{-1}} \bigr)$.

Claim (\ref{abcd-c4}) follows from \eqref{eq-cTQiMs} and \eqref{eq-P0T}.

Claim (\ref{abcd-c5}) follows from \eqref{eq-P0T} and assumption (\ref{i-t-m-b}) in Definition~\ref{def-bbP}.

Claims (\ref{abcd-c6}) and (\ref{abcd-c7}) follow from \eqref{ranks} and \eqref{eq-P0T}.
\end{proof}

\section*{Declarations}

\paragraph{\textbf{Funding}}
No funds, grants, or other support was received.

\paragraph{\textbf{Competing interests}}
The authors have no relevant financial or nonfinancial interests to disclose.

\paragraph{\textbf{Data Availability}}  
\noindent No datasets were generated or analyzed during this study.

\end{document}